\documentclass[11pt]{article}

\usepackage[usenames,dvipsnames]{xcolor}

\usepackage[margin={2.5cm,2.2cm}]{geometry} % Definir márgenes

\usepackage{graphicx}
\graphicspath{{./img/}}
\usepackage{float}
\usepackage{subcaption}

\usepackage[nocompress]{cite}

\usepackage{tikz}

\usepackage{enumerate}
\usepackage{amsmath,amsthm,amsfonts,bm}

\usepackage{multirow}
\usepackage{rotating}
\usepackage{hyperref}

\theoremstyle{plain}% default
\newtheorem{theorem}{Theorem}[section]
\newtheorem{proposition}[theorem]{Proposition}

\newtheorem{corollary}[theorem]{Corollary}
\newtheorem{remark}[theorem]{Remark}

\theoremstyle{definition}

\newtheorem{hypothesis}{Hypothesis}

\usepackage{mathtools}

\definecolor{bananamania}{rgb}{0.98, 0.91, 0.71}
\definecolor{bubbles}{rgb}{0.91, 1.0, 1.0}
\definecolor{lavenderblush}{rgb}{1.0, 0.94, 0.96}
\definecolor{lightred}{RGB}{255,210,202}
\definecolor{lightorange}{RGB}{255,234,199}
\definecolor{lightblue}{HTML}{EEF4F8}
\definecolor{lightyellow}{HTML}{FAF0E1}
\definecolor{lightgreen}{HTML}{F6FFE6}

\newcommand*{\defeq}{\mathrel{\vcenter{\baselineskip0.5ex \lineskiplimit0pt
			\hbox{\scriptsize.}\hbox{\scriptsize.}}}%
	=}
\newcommand{\R}{\mathbb{R}}

\newcommand{\bv}{\overline{v}}
\newcommand{\bu}{\overline{u}}
\newcommand{\bmu}{\overline{\mu}}

\newcommand{\nn}{\mathbf{n}}

\newcommand{\vK}{v_{K}}

\newcommand{\vL}{v_{L}}
\newcommand{\uK}{u_{K}}

\newcommand{\uL}{u_{L}}

\newcommand{\muK}{\mu_{K}}
\newcommand{\muL}{\mu_{L}}

\newcommand{\nabland}{\nabla_{\nn_e}^0}

\def\escalar#1#2{\left(#1,#2\right)}

\def\escalarL#1#2{\escalar{#1}{#2}}
\def\escalarLd#1#2{\escalar{#1}{#2}}
\def\escalarML#1#2{\escalar{#1}{#2}_h}

\def\norma#1{\left\|#1\right\|}

\def\salto#1{\left[\!\left[#1\right]\!\right]}
\def\media#1{\left\{\!\!\left\{#1\right\}\!\!\right\}}

\def\T{\mathcal{T}}
\def\E{\mathcal{E}}

\def\N{\mathbb{N}}

\def\X{\mathcal{X}}

\def\Ehi{\mathcal{E}_h^{\text{i}}}

\def\Pd{\mathbb{P}^{\text{disc}}}
\def\Pc{\mathbb{P}^{\text{cont}}}

\def\aupw#1#2#3{a_h^{\text{upw}}(#1;#2,#3)}

\title{\textbf{
An unconditionally energy stable and positive upwind DG scheme for the Keller-Segel model}}

\makeatletter
\def\@fnsymbol#1{\ensuremath{\ifcase#1\or \dagger\or \ddagger\or \mathsection\or * \or\mathparagraph\or *\or **\or \dagger\dagger \or \ddagger\ddagger \else\@ctrerr\fi}}
\makeatother

\author{Daniel Acosta-Soba\thanks{Departamento de Matemáticas, Universidad de Cádiz, Puerto Real, 11510 Cádiz, Spain -- Email: \texttt{daniel.acosta@uca.es}} \thanks{Department of Mathematics, University of Tennessee at Chattanooga, Chattanooga, TN 37403, USA}~,
~Francisco Guillén-González\thanks{Departamento de Ecuaciones Diferenciales y Análisis Numérico \& IMUS, Universidad de Sevilla, 41012 Seville, Spain -- Email: \texttt{guillen@us.es}}~,
~J. Rafael Rodríguez-Galván\thanks{Departamento de Matemáticas, Universidad de Cádiz, Puerto Real, 11510 Cádiz, Spain -- Email: \texttt{rafael.rodriguez@uca.es} -- Corresponding author}}

\begin{document}

\bibliographystyle{abbrv} % Bibliografía: enumerada

\maketitle

\begin{abstract}
The well-suited discretization of the Keller-Segel equations for chemotaxis has become a very challenging problem due to the convective nature inherent to them. This paper aims to introduce a new upwind, mass-conservative, positive and energy-dissipative discontinuous Galerkin scheme for the Keller-Segel model. This approach is based on the gradient-flow structure of the equations. In addition, we show some numerical experiments in accordance with the aforementioned properties of the discretization. The numerical results obtained emphasize the really good behaviour of the approximation in the case of chemotactic collapse, where very steep gradients appear.
\end{abstract}

\paragraph{Keywords:} Keller-Segel equations, chemotaxis, discontinuous Galerkin, upwind scheme, positivity preserving, energy stability.

\section{Introduction}
Since the introduction in the 70’s of the biological Keller–Segel
model for chemotaxis phenomena
\cite{keller_segel_1970,keller_segel_1971}, it and many related
variants have attracted a great deal of interest in the mathematical community.
Chemotaxis, a biological process through which organisms (e.g. cells)
migrate in response to a chemical stimulus, is modelled by means of
nonlinear systems of partial differential equations (PDE).
The classical one can be written as follows: find two real valued functions, $u=u(\mathbf{x},t)$
and $v=v(\mathbf{x},t)$, defined
in $\Omega\times [0,T]$ such that:
\begin{subequations}
	\label{problem:KS}
	\begin{align}
	\label{eq:KS_u}
	\partial_t u&=k_0\Delta u-k_1\nabla\cdot(u\nabla v),\quad&\text{in }\Omega\times (0,T),\\
	\label{eq:KS_v}
	\tau\partial_t v&=k_2\Delta v-k_3v+k_4u,\quad&\text{in }\Omega\times (0,T),\\
	\partial_\mathbf{n}u& :=\nabla u\cdot \mathbf{n}=0,\quad
	\partial_\mathbf{n} v 
	=0, \quad &\text{on }\partial\Omega\times (0,T),\\
	\label{eq:CI.cahn-hilliard+adveccion}
	u(0)&=u_0,\quad v(0)=v_0\text{ if }\tau>0,\quad&\text{in }\Omega.%\\
	\end{align}
\end{subequations}
Herein $\Omega$ is a bounded and smooth domain of $\mathbb{R}^d$, with
$d\in\mathbb{N}$ the spatial dimension,  and the parameters are $k_i>0$ for
$i\in\{0,1,2,3,4\}$.
The mathematical formulation of~\eqref{problem:KS} can be interpreted in
biological terms as follows: $u$ denotes a certain cell distribution
(or population of organisms, in general) at the position $\mathbf{x}\in \Omega$
and time $t\in [0,T]$, whereas $v$ stands for the concentration of
chemoattractant
(i.e. a chemical signal towards which cells are induced to migrate).
Both cells and chemoattractant experiment some diffusion
in the spatial domain.

This auto-diffusion phenomena (experimented by cells and chemoattractant)
are designed by the terms $-k_0\Delta u$ and $-k_2\Delta v$,
while the migration mechanism is modeled by the nonlinear
cross-diffusion term $-k_1\nabla\cdot(u\nabla v)$. This term is the major
difficulty for theoretical analysis and also for numerical modelling
of system~\ref{problem:KS}. Further, the degradation and production
of chemoattractant are associated with the terms $-k_3v$ and $k_4u$,
respectively. Note that the production of chemoattractant by the
cells, to which cells are attracted, may eventually result in a
chemotactic collapse, a phenomenon in which uncontrolled aggregation
for u give rise to blowing up or exploding in finite time.
This feature is well known and constitutes
one of the outstanding characteristics of classical Keller-Segel
model, and also one of its main challenges, specifically for numerical
methods. Finally, the coefficient $\tau\in\{0, 1\}$ is considered to write at the same time  the parabolic system when $\tau=1$, or the parabolic-elliptic for $\tau=0$.

Regarding the mathematical analysis for the system~\eqref{problem:KS}: some
results on sufficient conditions to ensure global existence and
boundedness of solutions along time can be shown (see e.g. the
review of Bellomo et al~\cite{bellomo2015toward} and the references
therein). They are based on mass conservation for $u$ and on
an energy dissipation law for this model (see Section~\ref{sec:keller-segel}).  For dimension $d\ge 2$, these results
require the initial density of cells, $\int_\Omega u_0$, to be bigger
than certain threshold. On the other hand, considerable research has
been done in the direction of finding cases where chemotactic collapse
arise. Among them, it is worth mentioning the first result in this
direction, due to Herrero and Velázquez~\cite{herrero1997blow}, where
radially symmetric two-dimensional solutions which finite-time blow up
are found. Other authors shed light on more general cases, for
instance Horstmann and Wang~\cite{horstmann2001blow} (non symmetric
blow-up solutions) or Winkler~\cite{winkler2010aggregation} (higher
dimensional case).

In the last decades, a lot of papers have been published dealing with
these kinds of theoretical issues both for the classical
model~\eqref{problem:KS} and for other models based on some extensions
or generalizations. In general, they start from the Keller–Segel
classical equations and modify them with the purpose of avoiding 
the non-physical blow up of solutions,
producing solutions which are closer to the ``real chemotaxis''
phenomena observed in biology. Models include logistic, non-linear
diffusion or production terms, chemo-repulsion effects or coupling
with fluid equations
\cite{tello2007chemotaxis-logistic,tao2012global,frassu2021improvements,chen2021uniquenessKS-NS,winkler2012global-NS}.
See e.g.~\cite{bellomo2015toward, arumugam_keller-segel_2021} for more examples. Thus, it is
hoped that understanding the classical Keller–Segel equations may open
new insights for dealing in depth with those other chemotaxis models.

On the other hand, taking into account the considerable efforts of the mathematical community in
the theoretical analysis of Keller-Segel models, the number of papers
dedicated to numerical analysis and simulation of chemotaxis equations
is much lower, in relative terms. The main difficulty is the numerical
approximation of the cross-diffusion term. Not only for its
non-linearity, but due to its convective nature, which
makes particularly difficult to deal with using the finite element (FE) method (see, for instance, \cite{brenner_mathematical_2008,ern_theory_2010} for more details about this method). This difficulty is specifically significant in steep-gradient
regions for $v$, which are precisely relevant in blow-up settings. Furthermore, preserving the physical properties of the continuous model (mass conservation, positivity and energy dissipation) in the discrete case adds an extra level of complication when it comes to designing a well-suited approximation.

Despite that, many interesting works have been published on numerical
simulation of chemotaxis equations using different kinds of approaches. For instance, the work by
Saito~\cite{saito_conservative_2007} uses FE with upwind stabilization for the
parabolic-elliptic Keller–Segel model ($\tau=0$ in~\eqref{problem:KS}) showing mass conservation, positivity and error estimates. Also, Guti\'errez-Santacreu and Rodríguez-Galv\'an demonstrate positivity, an energy law and a priori bounds for their FE scheme on acute meshes in \cite{gutierrez-santacreu_analysis_2021}. In addition, other sophisticated techniques have been applied to this problem. This is the case of the finite volume (FV) method (we recommend \cite{leveque_2002} on this topic) which has become a very popular and successful approach as we can observe in papers like \cite{chertock_second-order_2008} by Chertock and
Kurganov in which the authors devised positive preserving
methods for the parabolic-parabolic formulation ($\tau=1$ in~\eqref{problem:KS}), with demonstrated
high accuracy and robustness, specifically on chemotactic collapse. It
might also be pointed out the works of Saad and others, for instance
in~\cite{saad2014efficacy}, where a volume finite element scheme for
the capture of spatial patterns for a volume-filling chemotaxis is
analyzed. 

In this sense, it is also worth mentioning the very recent works \cite{badia_bound-preserving_2022, shen_unconditionally_2020, chen_error_2022, huang_bound_preserving} that show and analyze different techniques to approximate the solution of \eqref{problem:KS} while achieving mass conservation, positivity and energy-dissipation for a strictly positive initial cell condition. In the paper by Bad\'ia et al. \cite{badia_bound-preserving_2022} a discrete scheme using stabilized FE with a graph-Laplacian operator and a shock detector is proposed. This discretization satisfies, in general, both the mass conservation and positivity properties, and, in the case of acute meshes, it is also energy-dissipative. Otherwise, Huang and Shen develop in \cite{huang_bound_preserving} a time-discrete approximation,  admitting any spatial discretization, that preserves the positivity for the cell distribution $u$ and is energy stable for a modified energy. This latest approach uses a suitable transformation of the solution for the positivity and the scalar auxiliary variable (SAV) technique for the energy stability. Alternatively, Shen and Xu introduced in \cite{shen_unconditionally_2020} another general, positive (for the cell distribution $u$) and energy-dissipative, time-discrete scheme based on the gradient flow structure of the continuous model that admits different kinds of spatial discretization such as FE, spectral methods or even  finite-differences. The order of the latest approach and the blow-up phenomenon of the discrete solution, under a CFL condition, is studied in \cite{chen_error_2022} by Chen et al.

Furthermore, discontinuous Galerkin (DG)  methods (we refer the reader to \cite{di_pietro_mathematical_2012, dolejsi_discontinuous_2015, riviere_discontinuous_2008} for a further insight) aroused the interest of
researchers in recent years due to their flexibility for
approximating, using standard meshes and computer libraries, different
types of PDEs: elliptic, parabolic, hyperbolic.  In the chemotaxis
context, it is worth mentioning the paper of Y.~Epshteyn and
A. Kurganov~\cite{epshteyn-kurganov_2009new}, where the FV scheme
given in~\cite{chertock_second-order_2008} for the 2D Keller-Segel
model~\eqref{problem:KS} is extended to a DG scheme on cartesian
meshes, obtaining good approximations even on blow-up regimes.
Y. Epshteyn introduced two other related schemes
in~\cite{epshteyn2009discontinuous}. In all cases, different
discontinuous Interior Penalty (IP) methods and upwinding techniques
were considered for defining DG approximations. The schemes are even
applied to the simulation of an haptotaxis model of tumor invasion
into healthy tissue. Error estimates are shown but no energy property
or maximum principle for the schemes is proven. In fact, spurious
oscillations and negative values in the solution are reported.

More recent works make further progress in this direction, for instance
in~\cite{zhang2016operator-split-DG}, where the classical
equations~\eqref{problem:KS} are approximated by a
positivity-preserving DG method with strong stability preserving (SSP)
high order time discretizations. Error order estimates as well as
positivity are shown in this work. Finally,
in~\cite{li_local_2017,guo_energy_2019}, the local discontinuous
Galerkin method is applied, showing respectively positivity and
energy dissipation.

In this work, we propose a new upwind DG scheme for the Keller-Segel model~\eqref{problem:KS} that preserves the mass-conservation, positivity and energy stability properties of the continuous problem. As in \cite{shen_unconditionally_2020}, the proposed discretization takes advantage of the gradient flow structure of the model. First, section~\ref{sec:notation} sets the notation that we are going to consider throughout the paper. In section~\ref{sec:keller-segel} we discuss the physical properties of the continuous model. Section~\ref{sec:esquema_completamente_discreto} is the main part of the paper, in which we introduce the upwind DG scheme \eqref{esquema_DG_upw_KS_non_truncated}. In particular, in section~\ref{sec:def_upw}, we define the upwind approximation based on the ideas introduced in \cite{acosta-soba_upwind_2022} along with some geometrical considerations for the mesh family $\T_h$, and we discuss the properties of the scheme in section~\ref{sec:properties_scheme}. Finally, in section~\ref{sec:numer-experiments}, we show several numerical tests in which we reproduce some blow-up results shown in the literature with one steady peak, \cite{chertock_second-order_2008}, and a peak moving towards the corner of the domain, \cite{saito_conservative_2007}, as well as pattern formation results with several peaks, \cite{andreianov_finite_2011, chamoun_monotone_2014, tyson_fractional_2000, fatkullin_study_2013}. These numerical experiments endorse the good behaviour of the approximation obtained with the new scheme, which allows us to capture peaks reaching values up to the order of $10^{7}$. These kinds of numerical results are rare in the literature due to the steep-gradients inherent to such sort of tests.

\section{Notation}
\label{sec:notation}
In this section we introduce the notation used throughout the paper. First, we consider a shape-regular triangular mesh $\T_h=\{K\}_{K\in \T_h}$ of size $h$ over  a bounded polygonal domain $\Omega\subset\R^d$. Moreover, we note the set of edges or faces of $\T_h$ by $\E_h$, which can be split into the \textit{interior edges} $\E_h^i$ and the \textit{boundary edges} $\E_h^b$. Then, $\E_h=\E_h^i\cup\E_h^b$.

Now, we fix the following orientation for the unit normal vector $\nn_e$ associated to and edge $e\in\E_h$ of the mesh $\T_h$:
\begin{itemize}
	\item If $e\in\E_h^i$ is shared by the elements $K$ and $L$, i.e. $e=\partial K\cap\partial L$, then $\nn_e$ is exterior to $K$ pointing to $L$ (see Figure~\ref{fig:orientation_n_e}).
	\item If $e\in\E_h^b$, then $\nn_e$ points outwards of the domain $\Omega$.
\end{itemize}

\begin{figure}
	\centering
	\includegraphics[scale=1.0]{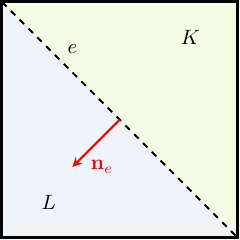}
	\caption{Orientation of the unit normal vector $\nn_e$.}
	\label{fig:orientation_n_e}
\end{figure}

Moreover, we denote the baricenter of the triangle $K\in\T_h$ by $C_K$.

Now, we define the approximation spaces of discontinuous, $\Pd_k(\T_h)$,  and continuous, $\Pc_k(\T_h)$, finite element functions over $\T_h$ whose restriction to $K\in\T_h$ are polynomials of degree $k\ge 0$. In addition, the \textit{average} $\media{\cdot}$ and the \textit{jump} $\salto{\cdot}$ of a scalar function $v$ on an edge $e\in\E_h$  are defined as follows:

\begin{equation*}
\media{v}\defeq
\begin{cases}
\dfrac{\vK+\vL}{2}&\text{if } e\in\E_h^i\\
\vK&\text{if }e\in\E_h^b
\end{cases},
\qquad
\salto{v}\defeq
\begin{cases}
\vK-\vL&\text{if } e\in\E_h^i\\
\vK&\text{if }e\in\E_h^b
\end{cases}.
\end{equation*}

Regarding the time discretization, we take an equispaced partition $0=t_0<t_1<\cdots<t_N=T$ of the time domain $[0,T]$ with $\Delta t=t_{m+1}-t_m$ the time step. We denote $v^m\simeq v(t_m)$ for any function $v$ defined on $[0,T]$ and define the discrete time derivative $\delta_t v^{m+1}=(v^{m+1}-v^m)/\Delta t$.

Finally we set the following notation for the positive and negative parts of a scalar function $v$:
$$
v_\oplus\defeq\frac{|v|+v}{2}=\max\{v,0\},
\quad
v_\ominus\defeq\frac{|v|-v}{2}=-\min\{v,0\},
\quad
v=v_\oplus - v_\ominus.
$$

\section{Keller-Segel model}
\label{sec:keller-segel}

Let us consider the Keller-Segel system~\eqref{problem:KS}.

\begin{remark}
	\label{rmk:ppo_maximo_KS}
	There is a classic solution of the Keller-Segel problem~\eqref{problem:KS} at least local in time which is positive, i.e., $u, v\ge 0$ in $\Omega\times (0,T)$ whenever $u_0\ge 0$ and $v_0\ge0$ in $\Omega$. See, for instance, \cite{bellomo2015toward,diaz1995symmetrization}.
	
	To the best knowledge of the authors, the existence of global solutions in time is still not clear in the literature.
\end{remark}

Assume $u_0, v_0\ge0$ in $\Omega$. The weak formulation of the problem~\eqref{problem:KS} consists of finding $(u,v): [0,T]\times \Omega\to \R_+\times \R_+$ regular enough, i.e. $u(t), v(t)\in V$ for a certain regular Sobolev space $V$ (for instance, $V=W^{1,\infty}(\Omega)$), with $\partial_t u(t),\tau \partial_t v(t)\in V'$ a.e. $t\in(0,T)$,
satisfying the following variational problem a.e. $t\in(0,T)$:
\begin{subequations}
	\label{problem:KS_form_var}
	\begin{align}
	\label{eq:KS_form_var_1}
	\langle \partial_t u(t),\bu \rangle &=-k_0\escalarLd{\nabla u(t)}{\nabla\bu}+k_1\escalarLd{u(t)\nabla v(t)}{\nabla\bu}, &\forall\bu\in V,\\
	\label{eq:KS_form_var_2}
	\langle \tau\partial_tv(t),\bv \rangle&=-k_2 \escalarLd{\nabla v(t)}{\nabla\bv}-k_3\escalarL{ v(t)}{\bv}+k_4\escalarL{u(t)}{\bv},&\forall\bv\in V,
	\end{align}
\end{subequations}
and the initial conditions $u(0)=u_0$, $v(0)=v_0$ in $\Omega$. Hereafter, $\escalarL{\cdot}{\cdot}$ and $\langle\cdot,\cdot\rangle$ denote the scalar product in $L^2(\Omega)$ and the duality product in $V'$, respectively.

By taking, formally, the chemical potential of $u$,
\begin{equation}
\label{def:mu}
\mu=k_0\log(u) - k_1 v,
\end{equation}
we can rewrite \eqref{problem:KS_form_var} as a gradient flow system where the flux direction is given by $-\nabla\mu$ containing the effect of both the diffusion and the chemotaxis terms (see, for instance, \cite{blanchet_hybrid_2015}). This variational formulation consists of finding $(u,\mu, v): [0,T]\times \Omega\to \R_+\times\R\times\R_+$ regular enough, i.e. $u(t), \mu(t), v(t)\in V$ for a certain regular Sobolev space $V$ (for instance, $V=W^{1,\infty}(\Omega)$), with $\partial_t u(t),\tau \partial_t v(t)\in V'$ a.e. $t\in(0,T)$,
satisfying the following variational problem a.e. $t\in(0,T)$:
\begin{subequations}
	\label{problem:KS_form_var_mu}
	\begin{align}
	\label{eq:KS_form_var_mu_1}
	\langle \partial_t u(t),\bu \rangle &=-\escalarLd{u(t)\nabla \mu(t)}{\nabla\bu}, &\forall\bu\in V,\\
	\label{eq:KS_form_var_mu_2}
	\escalarL{\mu(t)}{\bmu} &=\escalarL{k_0\log(u(t))-k_1v(t)}{\bmu}, &\forall\bmu\in V,\\
	\label{eq:KS_form_var_mu_3}
	\langle \tau\partial_tv(t),\bv \rangle&=-k_2 \escalarLd{\nabla v(t)}{\nabla\bv}-k_3\escalarL{ v(t)}{\bv}+k_4\escalarL{u(t)}{\bv},&\forall\bv\in V,
	\end{align}
\end{subequations}
and the initial conditions $u(0)=u_0$, $v(0)=v_0$ in $\Omega$.

\begin{remark}
	\label{rmk:mass_conservation_KS}
	By taking $\bu=1$ in 
	\eqref{eq:KS_form_var_1} (or \eqref{eq:KS_form_var_mu_1}), any solution $u$
	conserves the mass, because
	$$\frac{d}{dt}\int_\Omega u(x,t)dx=0.$$
\end{remark}

\begin{remark}  By taking (formally) $\bu=\mu(t)$, $\bmu=\partial_t u(t)$
	and $\bv=(k_1/k_4)\partial_t v(t) $ 
	in \eqref{problem:KS_form_var_mu},
	and adding the resulting expressions, one has that any solution $(u,v)$
	satisfies the following  energy law
	\begin{align}
	\label{ley_energia_continua_KS}
	\frac{d }{dt}E(u(t),v(t))&+\tau\frac{k_1}{k_4}\int_\Omega |\partial_t v(t)|^2dx
	+ \int_\Omega u(t)\left\vert\nabla(\mu(t))\right\vert^2 dx =0,
	\end{align}
	where $E\colon H^1(\Omega)_+\times H^1(\Omega)\longrightarrow\R$ is the energy functional, defined as follows
	\begin{align}
	\label{energia_KS}
	E(u,v)&\defeq\int_\Omega \left(k_0u\log(u)-k_1uv+\frac{k_1k_2}{2k_4}\vert\nabla v\vert ^2+\frac{k_1k_3}{2k_4}v^2\right),
	\end{align}
	and $H^1(\Omega)_+=\{u\in H^1(\Omega)\colon u\ge 0\}$.
\end{remark}

\section{Fully discrete scheme}
\label{sec:esquema_completamente_discreto}

First, we regularize the chemical potential of $u$, defined in~\eqref{def:mu}, by
\begin{equation}
\label{eq:mu}
\mu_\varepsilon=k_0\log(u+\varepsilon)-k_1 v,
\end{equation}
for some $\varepsilon>0$. Note that $\varepsilon$ is a regularization parameter since $\log(u+\varepsilon)$ is regular for all $u\ge 0$. We will take $\varepsilon=\varepsilon(h,\Delta t)$ such that $\varepsilon(h,\Delta t)\to 0$ if $(h,\Delta t)\to 0$, being $h>0$ the mesh size and $\Delta t>0$ the time step.

Then, we propose the following decoupled fully discrete first order in time and upwind DG in space  scheme for the model~\eqref{problem:KS}:

\
Let $v^{m}\in \Pc_1(\T_h)$ and $u^{m}\in \Pd_0(\T_h)$ such that $v^m\ge 0$ in the case $\tau>0$ and $u^m\ge 0$ be given.

\

\begin{subequations}
	\underline{Step 1}: Find  $v^{m+1}\in \Pc_1(\T_h)$ solving
	\label{esquema_DG_upw_KS_non_truncated}
	\begin{align}
	\label{eq:esquema_DG_upw_KS_non_truncated_1}
	\tau\escalarML{\delta_tv^{m+1}}{\bv}&+k_2\escalarLd{\nabla v^{m+1}}{\nabla\bv}+k_3\escalarML{v^{m+1}}{\bv}-k_4\escalarL{u^m}{\bv}=0,
	\end{align}
	for all $\bv\in \Pc_1(\T_h)$.

	\
	
	\underline{Step 2}: Find $(u^{m+1},\mu^{m+1})\in \Pd_0(\T_h)\times \Pd_0(\T_h)$ with $u^{m+1}\ge 0$ solving the coupled problem
	\begin{align}
	\label{eq:esquema_DG_upw_KS_non_truncated_2}
	\escalarL{\delta_tu^{m+1}}{\bu}&+\aupw{\mu^{m+1}}{u^{m+1}}{\bu}=0,\\
	\label{eq:esquema_DG_upw_KS_non_truncated_3}
	\escalarL{\mu^{m+1}}{\bmu}&-k_0 \escalarL{\log(u^{m+1}+\varepsilon)}{\bmu}+ k_1\escalarL{v^{m+1}}{\bmu}=0,
	\end{align}
	for all $\bu,\bmu\in \Pd_0(\T_h)$,
\end{subequations}
where $\aupw{\cdot}{\cdot}{\cdot}$ will be defined below in section~\ref{sec:def_upw}.

Notice that \ref{eq:esquema_DG_upw_KS_non_truncated_1} is a linear problem for $v^{m+1}$ and that \eqref{eq:esquema_DG_upw_KS_non_truncated_2}--\eqref{eq:esquema_DG_upw_KS_non_truncated_3} is a coupled nonlinear problem for $(u^{m+1},\mu^{m+1})$. In fact, we are going to use Newton's method as iterative procedure approximating the scheme \eqref{eq:esquema_DG_upw_KS_non_truncated_2}--\eqref{eq:esquema_DG_upw_KS_non_truncated_3}.

In order to preserve the positivity of $v^{m+1}$, we have done mass lumping in the terms $\escalarML{\partial_t v^{m+1}}{\bv}$ and	$k_3\escalarML{v^{m+1}}{\bv}$ in \eqref{eq:esquema_DG_upw_KS_non_truncated_1}.

In section~\ref{sec:properties_scheme}, we will provide a way of computing the solution of \eqref{esquema_DG_upw_KS_non_truncated} enforcing the nonnegativity restriction $u^{m+1}\ge 0$.

\subsection{Definition of upwind bilinear form \boldmath{$\aupw{\cdot}{\cdot}{\cdot}$}}
\label{sec:def_upw}
Now, we are going to define the upwind bilinear form $\aupw{\cdot}{\cdot}{\cdot}$, introduced in the scheme \eqref{esquema_DG_upw_KS_non_truncated}.

In order to achieve the energy stability with the scheme \eqref{esquema_DG_upw_KS_non_truncated}, we must consider the following hypothesis that will let us approximate the flux $-\nabla\mu$ accordingly:
\begin{hypothesis}
	\label{hyp:mesh_n}
	The mesh $\T_h$ of $\overline\Omega$ is structured in the sense that the line between the baricenters of the triangles $K$ and $L$ is orthogonal to the interface $e=K\cap L\in\E_h^i$.
\end{hypothesis}

Then, we define the following upwind bilinear form
to be applied to the flux $-\nabla\mu$ which may be discontinuous over $\E_h^i$:
\begin{equation}
\label{def:aupw}
\aupw{\mu}{u}{\bu}\defeq \int_\Omega (\nabla\mu\cdot\nabla\bu)u
+ \sum_{e\in\E_h^i,e=K\cap L}\int_e\left(\left(-\nabland\mu\right)_\oplus\uK-(-\nabland\mu)_\ominus\uL\right)\salto{\bu},
\end{equation}
where, for every $e\in\E_h^i$ with $e=K\cap L$,
\begin{equation}
\label{eq:approx_gradn}
\nabland\mu =\frac{-\salto{\Pi_0\mu}}{\mathcal{D}_e(\T_h)}= \frac{\Pi_0\muL-\Pi_0\muK}{\mathcal{D}_e(\T_h)},
\end{equation}
with $\Pi_0$ being the projection on $\Pd_0(\T_h)$ and $\mathcal{D}_e(\T_h)$ the distance between the baricenters of the triangles $K$ and $L$ of the mesh $\T_h$ that share $e\in\Ehi$, denoted by $C_K$ and $C_L$, respectively.
This way, we can rewrite \eqref{def:aupw} as
\begin{align}
\label{def:aupw_baricenter}
\aupw{\mu}{u}{\bu}\defeq &\int_\Omega (\nabla\mu\cdot\nabla\bu)u
\nonumber\\ &+ \sum_{e\in\E_h^i,e=K\cap L}\frac{1}{\mathcal{D}_e(\T_h)}\int_e\left(\left(\salto{\Pi_0\mu}\right)_\oplus\uK-(\salto{\Pi_0\mu})_\ominus\uL\right)\salto{\bu}.
\end{align}

\begin{remark}
	\label{rmk:approx_grad_mu}
	Since the quadrature formula of the baricenter (or centroid) is exact for polynomials of order $1$, if $\mu\in\Pd_1(\T_h)$, we have that
	$$\Pi_0\mu|_K=\oint_K\mu=\frac{1}{\vert K\vert}\int_K\mu=\mu(C_K),$$
	where $C_K$ is the baricenter of $K\in\T_h$.
	
	Hence, if $\mu\in\Pd_1(\T_h)$, the expression \eqref{eq:approx_gradn} is the slope of the line between the points $(C_K, \mu(C_K))$ and $(C_L,\mu(C_L))$, which, under the Hypothesis~\ref{hyp:mesh_n}, this line is parallel to the vector $\nn_e$, with $e=K\cap L\in\E_h^i$. This expression is considered as an approximation of the discontinuous numerical normal flux $\nabla\mu\cdot\nn_e$ for $\mu\in\Pd_1(\T_h)$.
	
	In addition, observe that, since the baricenters are located $1/3$ of the median from the side and $2/3$ of the median from the vertex of the triangle, the expression \eqref{eq:approx_gradn} does not degenerate when $h\to 0$, i.e., $\mathcal{D}_e(\T_h)>0$ for every $e\in\E_h^i$ and $h>0$. A visual representation of the regular polygonal structure given by the lines between the adjacent baricenters is given in Figure~\ref{fig:polygonal_baricenters}.
	
	\begin{figure}
		\centering
		\includegraphics[scale=1.0]{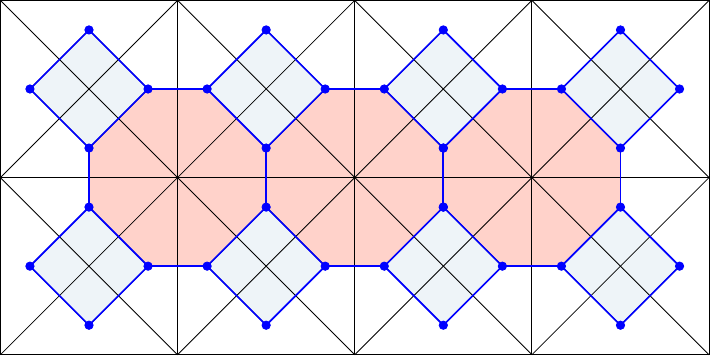}
		\caption{Polygonal structure between adjacent baricenters.}
		\label{fig:polygonal_baricenters}
	\end{figure}
\end{remark}

Furthermore, in order to preserve the positivity of the variable $v$ in our fully discrete scheme, we assume the following hypothesis.
\begin{hypothesis}
	\label{hyp:mesh_acute}
	The mesh $\T_h$ is acute, i.e., the angles of the triangles of $T_h$ are less than or equal to $\pi/2$.
\end{hypothesis}

\begin{figure}
	\centering
	\begin{subfigure}{0.49\textwidth}
		\centering
		\includegraphics[scale=1.0]{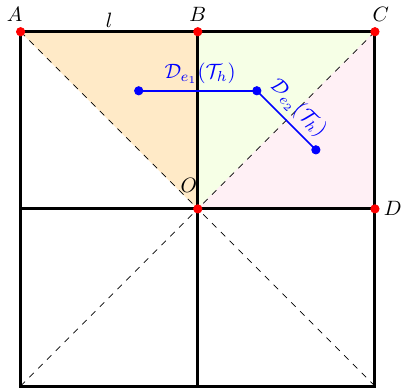}
		\caption{Mesh 1}
	\end{subfigure}
	\begin{subfigure}{0.49\textwidth}
		\centering
		\includegraphics[scale=1.0]{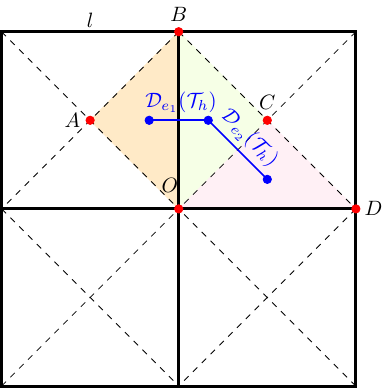}
		\caption{Mesh 2}
	\end{subfigure}
	\caption{Representation of $\mathcal{D}_e(\T_h)$.}
	\label{fig:representation_D_bar}
\end{figure}

We give some examples, the meshes represented in Figure~\ref{fig:representation_D_bar}, which satisfy both Hypotheses~\ref{hyp:mesh_n} and~\ref{hyp:mesh_acute}.

\begin{theorem}
	Given the meshes represented in Figure~\ref{fig:representation_D_bar} we can define $\mathcal{D}_e(T_h)$ for these meshes as follows
	$$
	\text{a) \underline{Mesh 1}: } \mathcal{D}_e(\T_h)=\frac{2l^2}{3\vert e\vert},
	\quad
	\text{b) \underline{Mesh 2}: } \mathcal{D}_e(\T_h)=\frac{l^2}{3\vert e\vert},
	$$
	where $l$ is the length of the side of the highlighted squares of the mesh.
\end{theorem}
\begin{proof}
	We will only prove the case a) Mesh 1 since the case b) Mesh 2 is analogous.
	
	Observe Mesh 1 in Figure~\ref{fig:representation_D_bar}. The baricenter of the triangles $\triangle OAB$, $\triangle OBC$ and $\triangle OCD$ are, respectively $\frac{O+A+B}{3}$, $\frac{O+B+C}{3}$ and $\frac{O+C+D}{3}$. Hence, if we denote by $e_1$ the edge between $\triangle OAB$ and $\triangle OBC$ and by $e_2$ the edge between $\triangle OBC$ and $\triangle OCD$, then
	$$
	\mathcal{D}_{e_1}(\T_h)=\frac{\vert C-A\vert}{3}=\frac{2l}{3},
	%\\
	\quad
	\mathcal{D}_{e_2}(\T_h)=\frac{\vert D-B\vert}{3}=\frac{\sqrt{2}l}{3}
	$$
	Now, since $l/\vert e_1\vert =1$ and $l/\vert e_2\vert=1/\sqrt{2}$, we can define $\mathcal{D}_e(\T_h)$ for any $e\in\Ehi$ as
	\begin{equation}
	\mathcal{D}_e(\T_h)=\frac{2l}{3}\cdot \frac{l}{\vert e\vert}=\frac{2l^2}{3\vert e\vert}.
	\end{equation}
\end{proof}

\subsection{Properties of the scheme}
\label{sec:properties_scheme}
Finally, we discuss the different properties of the scheme \eqref{esquema_DG_upw_KS_non_truncated} with the upwind bilinear form defined in \eqref{def:aupw} on meshes under the Hypotheses~\ref{hyp:mesh_n} and~\ref{hyp:mesh_acute}.

With the purpose of proving the existence of solution of the scheme \eqref{esquema_DG_upw_KS_non_truncated} and providing a way of computing it enforcing the restriction $u^{m+1}\ge 0$, we define the following auxiliary scheme where a cut-off operator is introduced in \eqref{eq:esquema_DG_upw_KS_non_truncated_2}:

\

\begin{subequations}
	\label{esquema_DG_upw_KS_truncated}
	\underline{Step 1}: Given $v^m\in\Pc_1(\T_h)$ such that $v^m\ge 0$ in the case $\tau >0$, find  $v^{m+1}\in \Pc_1(\T_h)$ solving
	\begin{align}
	\label{eq:esquema_DG_upw_KS_truncated_v}
	\tau\escalarML{\delta_tv^{m+1}}{\bv}&+k_2\escalarLd{\nabla v^{m+1}}{\nabla\bv}+k_3\escalarML{v^{m+1}}{\bv}-k_4\escalarL{u^m}{\bv}=0,
	\end{align}
	for all $\bv\in \Pc_1(\T_h)$.

	\
	
	\underline{Step 2}:  Given $u^m, \mu^m\in\Pd_0(\T_h)$ such that $u^m\ge 0$, find $u^{m+1},\mu^{m+1}\in \Pd_0(\T_h)$ solving
	\begin{align}
	\label{eq:esquema_DG_upw_KS_truncated_u}
	\escalarL{\delta_tu^{m+1}}{\bv}&+\aupw{\mu^{m+1}}{(u^{m+1})_\oplus}{\bu}=0,\\
	\label{eq:esquema_DG_upw_KS_truncated_mu}
	\escalarL{\mu^{m+1}}{\bmu}&-k_0 \escalarL{\log(u^{m+1}+\varepsilon)}{\bmu}+ k_1\escalarL{v^{m+1}}{\bmu}=0,
	\end{align}
	for all $\bu,\bmu\in \Pd_0(\T_h)$. 
\end{subequations}
\begin{remark}
	In order to preserve the positivity of the solution $u^{m+1}$ of \eqref{esquema_DG_upw_KS_truncated}, we have introduced a truncation of this function taking its positive part $(u^{m+1})_\oplus$ in the upwind part of \eqref{eq:esquema_DG_upw_KS_truncated_u}, which is consistent as the solution of the continuous model \eqref{problem:KS} satisfies $u\ge 0$.

	Since Theorem~\ref{thm:principio_del_maximo_DG_KS} below will guarantee  that
 $u^{m+1}\ge 0$, then $\escalarL{\log(u^{m+1}+\varepsilon)}{\bmu}$ in \eqref{eq:esquema_DG_upw_KS_truncated_u} is well-defined.
\end{remark}

\begin{proposition}
	\label{prop:discrete_mass_conservation}
	The schemes \eqref{esquema_DG_upw_KS_non_truncated} and \eqref{esquema_DG_upw_KS_truncated} conserve the mass of $u$:
	$$\int_\Omega u^{m+1}=\int_\Omega u^m.$$
\end{proposition}
\begin{proof}
	Just need to take $\overline{u}=1$ in \eqref{eq:esquema_DG_upw_KS_non_truncated_2} and \eqref{eq:esquema_DG_upw_KS_truncated_u}.
\end{proof}

\begin{theorem}[DG \ref{esquema_DG_upw_KS_truncated} preserves positivity]
	\label{thm:principio_del_maximo_DG_KS}
	
	If we assume that $u^m\ge 0$ and, in the case $\tau>0$, $v^m\ge 0$ in $\Omega$, then any solution of \ref{esquema_DG_upw_KS_truncated} satisfies that $u^{m+1}, v^{m+1}\ge 0$ in $\Omega$.
\end{theorem}
\begin{proof}
	Proving that if $u^m\ge 0$ and $v^m\ge 0$ (when $\tau>0$) then $v^{m+1}\ge 0$ using that the mesh is acute is a classic result which can be found, for example, in \cite{ciarlet_maximum_1973, fernandez-romero_theoretical_2021}.
	
	Moreover, if $u^m\ge 0$ and $v^m\ge 0$ (when $\tau>0$) it follows from the equation \eqref{eq:esquema_DG_upw_KS_truncated_u} that $u^{m+1}\ge 0$ using the same arguments that are shown to prove the positivity result in \cite{acosta-soba_upwind_2022} since the proof given in \cite{acosta-soba_upwind_2022} is independent of the flux $\nabla\mu^{m+1}$.
\end{proof}

\begin{proposition}
	\label{prop:existencia_solucion_KS_DG-UPW_v}
	There is a unique solution $v^{m+1}$ of the linear equation \eqref{eq:esquema_DG_upw_KS_truncated_v}.
\end{proposition}
\begin{proof}
	Since we are dealing with a discrete linear problem, existence and unicity of the solution are equivalent. Hence, we just need to assume that there are two solutions of \eqref{eq:esquema_DG_upw_KS_non_truncated_1}, $v_1$ and $v_2$, substract the expressions resulting of evaluating both solutions and test with $v_1-v_2$ to prove unicity of the solution.
\end{proof}

\begin{proposition}
	\label{prop:existencia_solucion_KS_DG-UPW_u}
	There is at least one solution of \eqref{eq:esquema_DG_upw_KS_truncated_u}--\eqref{eq:esquema_DG_upw_KS_truncated_mu}.
\end{proposition}
\begin{proof}
	Consider the following well-known theorem:
	%\vspace{-2em}
	%\begin{quote}
	\begin{theorem}[Leray-Schauder fixed point theorem]
		\label{thm:Leray-Schauder}
		Let $\X$ be a Banach space and let $T\colon\X\longrightarrow\X$ be a continuous and compact operator. If the set $$\{x\in\X\colon x=\alpha \,T(x)\quad\text{for some } 0\le\alpha\le1\}$$ is bounded (uniformly with respect to $\alpha$), then $T$ has  at least one fixed point.
	\end{theorem}
	
	%\end{quote}
	
	Given $u^m\in\Pd_0(\T_h)$ with $u^m\ge 0$ and the unique solution $v^{m+1}$ of \eqref{eq:esquema_DG_upw_KS_truncated_v}, we define the map
	$$T\colon \Pd_0(\T_h)\times\Pd_0(\T_h)\longrightarrow\Pd_0(\T_h)\times \Pd_0(\T_h)$$ such that $T(\widehat{u},\widehat{\mu})=(u,\mu)\in\Pd_0(\T_h)\times\Pd_0(\T_h)$ is the unique solution of the linear (and decoupled) problem:
	\begin{subequations}
		\label{esquema_lineal_Leray-Schauder_DG_upw_KS}
		\begin{align}
		\label{eq:esquema_lineal_Leray-Schauder_DG_upw_KS_1}
		\frac{1}{\Delta t}\escalarL{u-u^m}{\overline{u}}&
		=-\aupw{\widehat{\mu}}{\widehat{u}_\oplus}{\overline{u}},&\forall\overline{u}\in \Pd_0(\T_h),\\
		\label{eq:esquema_lineal_Leray-Schauder_DG_upw_KS_2}
		\escalarL{\mu}{\bmu}&=k_0 \escalarL{\log((\widehat u)_\oplus+\varepsilon)}{\bmu}- k_1\escalarL{v^{m+1}}{\bmu}, &\forall\bmu\in \Pd_0(\T_h).
		\end{align}
	\end{subequations}
	
	To check that $T$ is well defined, it is straightforward to see that the solutions $u$ of \eqref{eq:esquema_lineal_Leray-Schauder_DG_upw_KS_1} and $\mu$ of \eqref{eq:esquema_lineal_Leray-Schauder_DG_upw_KS_2} are unique, which involves their existence as $\Pd_0(\T_h)$ is a finite-dimensional space.
	
	Secondly, we will check that $T$ is continuous. Let $\{\widehat{u}_j\}_{j\in\N}, \{\widehat{\mu}_j\}_{j\in\N}\subset\Pd_0(\T_h)$ be sequences such that $\lim_{j\to\infty}\widehat{u}_j=\widehat{u}$ and $\lim_{j\to\infty}\widehat{\mu}_j=\widehat{\mu}$. Taking into account that all norms are equivalent in $\Pd_0(\T_h)$ since it is a finite-dimensional space, the convergences $\widehat u_j\to \widehat u$ and $\widehat {\mu}_j\to \widehat {\mu}$ are equivalent to the elementwise convergences $(\widehat u_j)_K\to \widehat u_K$ and $(\widehat {\mu}_j)_K\to \widehat {\mu}_K$ for every $K\in\T_h$ (this may be seen, for instance, by using the norm $\norma{\cdot}_{L^\infty(\Omega)}$). Taking limits when $j\to \infty$ in \eqref{esquema_lineal_Leray-Schauder_DG_upw_KS} (with $\widehat{u}\defeq\widehat{u}_j$, $\widehat{\mu}\defeq\widehat{\mu}_j$ and $(u,\mu)\defeq T(\widehat u_j,\widehat{\mu}_j)$), using the notion of elementwise convergence and the fact that $\log((\widehat u)_\oplus+\varepsilon)$ is continuous, we get that $$\lim_{j\to \infty} T(\widehat u_j,\widehat{\mu}_j)=T(\widehat u,\widehat{\mu})=T\left(\lim_{j\to \infty}(\widehat u_j,\widehat{\mu}_j)\right),$$ hence $T$ is continuous. In addition, $T$ is compact since $\Pd_0(\T_h)$ have finite dimension.
	
	Finally, let us prove that the set $$B=\{(u,\mu)\in\Pd_0(\T_h)\times\Pd_0(\T_h)\colon (u,\mu)=\alpha T(u,\mu)\text{ for some } 0\le\alpha\le1\}$$ is bounded (independent of $\alpha$). The case $\alpha=0$ is trivial so we will assume that $\alpha\in(0,1]$.
	
	If $(u,\mu)\in B$, then $u\in\Pd_0(\T_h)$ is the solution of
	\begin{align}
	\label{eq:esquema_DG_upw_KS_Leray-Schauder_u}
	\frac{1}{\Delta t}\escalarL{u-\alpha u^m}{\overline{u}}
	=-\alpha\,\aupw{\mu}{u_\oplus}{\overline{u}},
	&&\forall\overline{u}\in \Pd_0(\T_h).
	\end{align}
	Now, testing \eqref{eq:esquema_DG_upw_KS_Leray-Schauder_u} with $\overline u=1$, we get that $$\int_\Omega u=\alpha \int_\Omega u^m ,$$ and, as $u^m\ge 0$ and it can be proved that $u\ge 0$ using the same arguments than in Theorem~\ref{thm:principio_del_maximo_DG_KS}, we get that $$\norma{u}_{L^1(\Omega)}\le \norma{u^m}_{L^1(\Omega)}.$$
	
	Moreover, since $u\ge0$, $\mu\in\Pd_0(\T_h)$ is the solution of the equation
	\begin{align}
	\label{eq:esquema_DG_upw_KS_Leray-Schauder_mu}
	\escalarL{\mu}{\bmu}=\alpha k_0 \escalarL{\log( u+\varepsilon)}{\bmu}- \alpha k_1\escalarL{v^{m+1}}{\bmu}, &&\forall\bmu\in \Pd_0(\T_h).
	\end{align}
	Hence,
	\begin{align*}
	\mu=\alpha k_0 \log( u+\varepsilon)- \alpha k_1\Pi_0 v^{m+1}, &&\text{in } \Pd_0(\T_h).
	\end{align*}
	Thus, taking into account that $u$ is bounded in $\Pd_0(\T_h)$, we conclude that $\mu$ is bounded in $\Pd_0(\T_h)$.
	
	Since $\Pd_0(\T_h)$ is a finite-dimensional space where all the norms are equivalent, we have proved that $B$ is bounded.
	
	Finally, we can apply the Leray-Schauder fixed point theorem \ref{thm:Leray-Schauder} to prove the existence of a fixed point of \eqref{eq:esquema_lineal_Leray-Schauder_DG_upw_KS_1}--\eqref{eq:esquema_lineal_Leray-Schauder_DG_upw_KS_2} and, consequently, the existence of a solution $(u^{m+1},\mu^{m+1})$ of
	\eqref{eq:esquema_DG_upw_KS_truncated_u}--\eqref{eq:esquema_DG_upw_KS_truncated_mu}.
\end{proof}

Since every solution of \eqref{esquema_DG_upw_KS_truncated} is positive according to Theorem~\ref{thm:principio_del_maximo_DG_KS}, the schemes \eqref{esquema_DG_upw_KS_truncated} and \eqref{esquema_DG_upw_KS_non_truncated} are equivalent in the sense that any solution of \eqref{esquema_DG_upw_KS_truncated} is solution of \eqref{esquema_DG_upw_KS_non_truncated} and vice versa. Therefore, using Propositions~\ref{prop:existencia_solucion_KS_DG-UPW_v} and~\ref{prop:existencia_solucion_KS_DG-UPW_u} the following result holds.
\begin{corollary}
	\label{cor:esquema_DG_upw_KS_no_parte_pos}
	There is at least one solution of the decoupled non-truncated scheme \eqref{esquema_DG_upw_KS_non_truncated}. Moreover, $v^{m+1}$ is unique.
\end{corollary}

\begin{remark}
	Obtaining the nonnegative solutions of \eqref{esquema_DG_upw_KS_non_truncated} can be enforced by solving the scheme \eqref{esquema_DG_upw_KS_truncated} including the cut-off operator \eqref{eq:esquema_DG_upw_KS_truncated_u}. In practice, the same solution was found in our numerical experiments using either the auxiliary truncated scheme \eqref{esquema_DG_upw_KS_truncated} or the non-truncated scheme \eqref{esquema_DG_upw_KS_non_truncated} without explicitly imposing the nonnegativity restriction $u^{m+1}\ge 0$ (see Remark~\ref{rmk:cut-off_operator}).
\end{remark}

\begin{remark}
	Showing uniqueness of solution of \eqref{esquema_DG_upw_KS_non_truncated} is not straightforward and it might require using inverse inequalities that would probably involve some kind of restriction on the time step and mesh size, and this is beyond the scope of this work.
\end{remark}

\begin{theorem}
	\label{thm:energia_esquema_KS}
	Any solution
	of the scheme \eqref{esquema_DG_upw_KS_non_truncated} satisfies the following \textbf{discrete energy law} at the time step $m+1$:
	\begin{multline}
	\label{eq:energy_law}
	\delta_t E_\varepsilon(u^{m+1}, v^{m+1})
	+ \Delta t \frac{k_1 k_3}{2k_4}\int_\Omega (\delta_t v^{m+1})^2+ \Delta t \frac{k_1k_2}{2k_4}\int_\Omega\vert\delta_t \nabla v^{m+1}\vert^2
	\\
	+ \tau \frac{k_1}{k_4} \int_\Omega (\delta_t v^{m+1})^2
	+\aupw{\mu^{m+1}}{u^{m+1}}{\mu^{m+1}}\le
	0,
	\end{multline}
	where
	\begin{align}
	\label{energia_eps_KS}
	&E_\varepsilon(u,v)\defeq\\\nonumber&\int_\Omega \big(k_0(u+\varepsilon)\log(u+\varepsilon)-k_1uv+\frac{k_1k_2}{2k_4}\vert\nabla v\vert ^2+\frac{k_1k_3}{2k_4}v^2\big).
	\end{align}
\end{theorem}
\begin{proof}
	
	Take $\bu=\mu^{m+1}$, $\bmu=\delta_t u^{m+1}$, $\bv=(k_1/k_4)\delta_t v^{m+1}$ in \eqref{esquema_DG_upw_KS_non_truncated} and consider the equalities
	\begin{align*}
	\delta_t(u^{m+1}v^{m+1})&=u^m\delta_t(v^{m+1})+\delta_t(u^{m+1}) v^{m+1},\\
	\delta_t(v^{m+1})v^{m+1}&=\frac{1}{2}\delta_t(v^{m+1})^2+\frac{\Delta t}{2}(\delta_t v^{m+1})^2.
	\end{align*}
	Then, adding the resulting expressions for \eqref{eq:esquema_DG_upw_KS_non_truncated_1} and \eqref{eq:esquema_DG_upw_KS_non_truncated_3} and substracting \eqref{eq:esquema_DG_upw_KS_non_truncated_2}, we obtain
	\begin{align}
	0&=\aupw{\mu^{m+1}}{u^{m+1}}{\mu^{m+1}}+k_0\int_\Omega\delta_t (u^{m+1})\log(u^{m+1}+\varepsilon)\nonumber\\ &\quad -k_1\int_\Omega \delta_t (u^{m+1})v^{m+1}
	-k_1\int_\Omega u^m\delta_t(v^{m+1})+\frac{k_1}{k_4}\int_\Omega (\delta_t v^{m+1})^2\nonumber\\&\quad+\frac{k_1k_2}{k_4}\int_\Omega\nabla v^{m+1}\cdot\nabla(\delta_t v^{m+1})+\frac{k_1k_3}{k_4}\int_\Omega v^{m+1}\delta_t (v ^{m+1})\nonumber\\
	&=\aupw{\mu^{m+1}}{u^{m+1}}{\mu^{m+1}}+k_0\int_\Omega\delta_t (u^{m+1})\log(u^{m+1}+\varepsilon)\nonumber\\
	&\quad-k_1\delta_t\int_\Omega u^{m+1}v^{m+1}+\frac{k_1k_2}{2k_4}\delta_t\int_\Omega\vert\nabla v^{m+1}\vert^2+\frac{\Delta tk_1k_2}{2k_4}\int_\Omega\vert\delta_t \nabla v^{m+1}\vert^2\nonumber\\
	\label{thm:discrete_energy_1}
	&\quad+\frac{k_1k_3}{2k_4}\delta_t\int_\Omega (v^{m+1})^2+\left(\frac{k_1}{k_4}+\frac{\Delta tk_1 k_3}{2k_4}\right)\int_\Omega (\delta_t v^{m+1})^2.
	\end{align}
	
	Now, using that $\delta_t (u^{m+1}) F'(u^{m+1}) \ge \delta_t (F(u^{m+1}))$ for $F'(u^{m+1})=\log(u^{m+1}+\varepsilon)$ (owing to the fact that $F(u)$ is convex) we have that
	\begin{align*}
	\delta_t(u^{m+1})\log(u^{m+1}+\varepsilon)\ge\delta_t\left((u^{m+1}+\varepsilon)\log(u^{m+1}+\varepsilon)\right)-\delta_t(u^{m+1}+\varepsilon).
	\end{align*}
	Hence, using Proposition~\ref{prop:discrete_mass_conservation},
	\begin{equation}
	\label{thm:discrete_energy_2}
	\int_\Omega\delta_t(u^{m+1})\log(u^{m+1}+\varepsilon)\ge
	\delta_t\left(\int_\Omega\left((u^{m+1}+\varepsilon)\log(u^{m+1}+\varepsilon)\right) \right).
	\end{equation}
	
	Thus, taking into account Theorems~\ref{thm:discrete_energy_1} and~\ref{thm:discrete_energy_2}, we obtain the discrete energy law \eqref{eq:energy_law}.
\end{proof}

\begin{corollary}
	Given a solution of the scheme \eqref{esquema_DG_upw_KS_non_truncated}, the upwind bilinear form
  defined in \eqref{def:aupw} satisfies
	$$
	\aupw{\mu^{m+1}}{u^{m+1}}{\mu^{m+1}}\ge0.
	$$
	
	In consequence, the scheme \eqref{esquema_DG_upw_KS_non_truncated} is unconditionally energy stable with respect to the approximated energy $E_\varepsilon$, that is $$E_\varepsilon(u^{m+1},v^{m+1})\le E_\varepsilon(u^{m},v^{m}).$$
\end{corollary}
\begin{proof}
	Since we know that the discrete energy satisfies \eqref{eq:energy_law}, it suffices to prove that $\aupw{\mu^{m+1}}{u^{m+1}}{\mu^{m+1}}\ge0$ to show $\delta_t E_\varepsilon(u^{m+1},v^{m+1})\le 0$.
	
	Now, take $\bu=\mu^{m+1}$ and use the definition \eqref{def:aupw_baricenter} of the upwind bilinear form to get the following:
	\begin{align*}
	&\aupw{\mu^{m+1}}{(u^{m+1})_\oplus}{\mu^{m+1}}=
	\\&=\sum_{e\in\E_h^i,e=K\cap L}\frac{1}{\mathcal{D}_e(\T_h)}\int_e\left(\left(\salto{\mu^{m+1}}\right)_\oplus\uK^{m+1}-(\salto{\mu^{m+1}})_\ominus\uL^{m+1}\right)\salto{\mu^{m+1}}\\&=\sum_{e\in\E_h^i,e=K\cap L}\frac{1}{\mathcal{D}_e(\T_h)}\int_e\left(\Big(\salto{\mu^{m+1}}\Big)_\oplus^2\uK^{m+1}+\Big(\salto{\mu^{m+1}}\Big)_\ominus^2\uL^{m+1}\right)
	%\\&
	\ge 0.
	\end{align*}
\end{proof}

\begin{remark}
	Notice that the energy stability of the scheme \eqref{esquema_DG_upw_KS_non_truncated} is obtained thanks to the approximation of the flux $-\nabla\mu$ made in the upwind bilinear form $\aupw{\cdot}{\cdot}{\cdot}$. In addition, the approximation $-\nabland\mu$ on the edges $e\in\E_h^{\text i}$ requires the assumption of the Hypothesis~\ref{hyp:mesh_n} for the mesh $\T_h$ as discussed in the Remark~\ref{rmk:approx_grad_mu}.
\end{remark}

%--------------------------------------
\section{Numerical experiments}
%--------------------------------------
\label{sec:numer-experiments}

In this section we show some numerical tests whose results are according to the results shown above for the scheme \eqref{esquema_DG_upw_KS_non_truncated}. For these tests we consider the parameters $k_i=1$ for $i\in\{0,1,\ldots,4\}$, $\tau=1$, $\varepsilon=10^{-10}$ and the domain $\Omega=[-1/2,1/2]\times[-1/2,1/2]$ unless otherwise specified. Also, the mesh~1 in Figure~\ref{fig:representation_D_bar} is used to discretize the domain.

In the test \ref{test1} we reproduce the first numerical experiment shown in the paper \cite{chertock_second-order_2008} by A. Chertock and A. Kurganov. In this paper, they use a scheme that preserves the positivity of both variables $u$ and $v$, although they do not show any energy related result. Hence, in our case, we can improve the results shown in the aforementioned paper assuring that our scheme preserves the energy law of the continuous Keller-Segel model.

Then, in the test \ref{test2} we simulate the qualitative behaviour of the solution in the numerical experiment made by N. Saito in \cite{saito_conservative_2007}. However, since the initial conditions used for the experiment in \cite{saito_conservative_2007} are not specified we cannot reproduce the exact same test shown in this paper. In the case of our numerical test, the qualitative behaviour of the solution is similar to the one in \cite{saito_conservative_2007} until the mesh is refined enough so that we capture the blow-up phenomenon.

Finally, in the test \ref{test3} we reproduce the qualitative behaviour of the results in \cite{andreianov_finite_2011, chamoun_monotone_2014, tyson_fractional_2000} for different variations of chemotaxis equations and in \cite{fatkullin_study_2013} for the Keller-Segel equations, where pattern formations with multiple peaks are shown.

\begin{remark}
	The scheme \eqref{esquema_DG_upw_KS_non_truncated} preserves the positivity and conserves the mass of $u^{m+1}$ which implies $u^{m+1}\in L^1(\Omega)$. Therefore, we cannot expect an actual blow-up in the discrete case as it occurs in the continuous model. However, we  observe in the numerical tests how the mass accumulates in some elements of $\T_h$ leading to the formation of peaks.
	
	In fact, we are able to capture peaks that reach values up to the order of $10^{7}$ using the approximation shown in \eqref{esquema_DG_upw_KS_non_truncated}. These kinds of numerical results are not usual in the literature due to the difficulties when approximating the steep gradients that this process involved.
\end{remark}

The numerical results shown in this paper have been obtained using the Python library \textit{FEniCS}, \cite{Fenics:AlnaesEtAl:2015}. In order to improve the efficiency of the code, these have been run in parallel using several CPUs.

For the sake of a better visualization of the results, a $\Pc_1$-projection of $u$ is represented in 3D using \textit{Paraview}, \cite{paraview}.

\begin{remark}
	\label{rmk:cut-off_operator}
	All the tests were carried out using both the non-truncated equation \eqref{eq:esquema_DG_upw_KS_non_truncated_2} without the restriction $u^{m+1}\ge 0$ and the truncated version \eqref{eq:esquema_DG_upw_KS_truncated_u} to enforce the nonnegativity. The approximations of $u$ and $v$ obtained in every case using both versions of the scheme were identical in all the degrees of freedom.
\end{remark}

\subsection{One bulge of cells}
\label{test1}
First, we reproduce the results shown in \cite{chertock_second-order_2008}. For this purpose, we consider the radially symmetric initial conditions
\begin{equation*}
u_0 = 1000 e^{-100 (x^2 + y^2)}, \quad
v_0 = 500e^{-50 (x^2 + y^2)},
\end{equation*}
which are plotted in Figure~\ref{fig:ic_chertock}.

\begin{figure}
	\centering
	\begin{subfigure}{0.49\textwidth}
		\centering
		\boldmath{$u_0$}
		
		\includegraphics[scale=0.11]{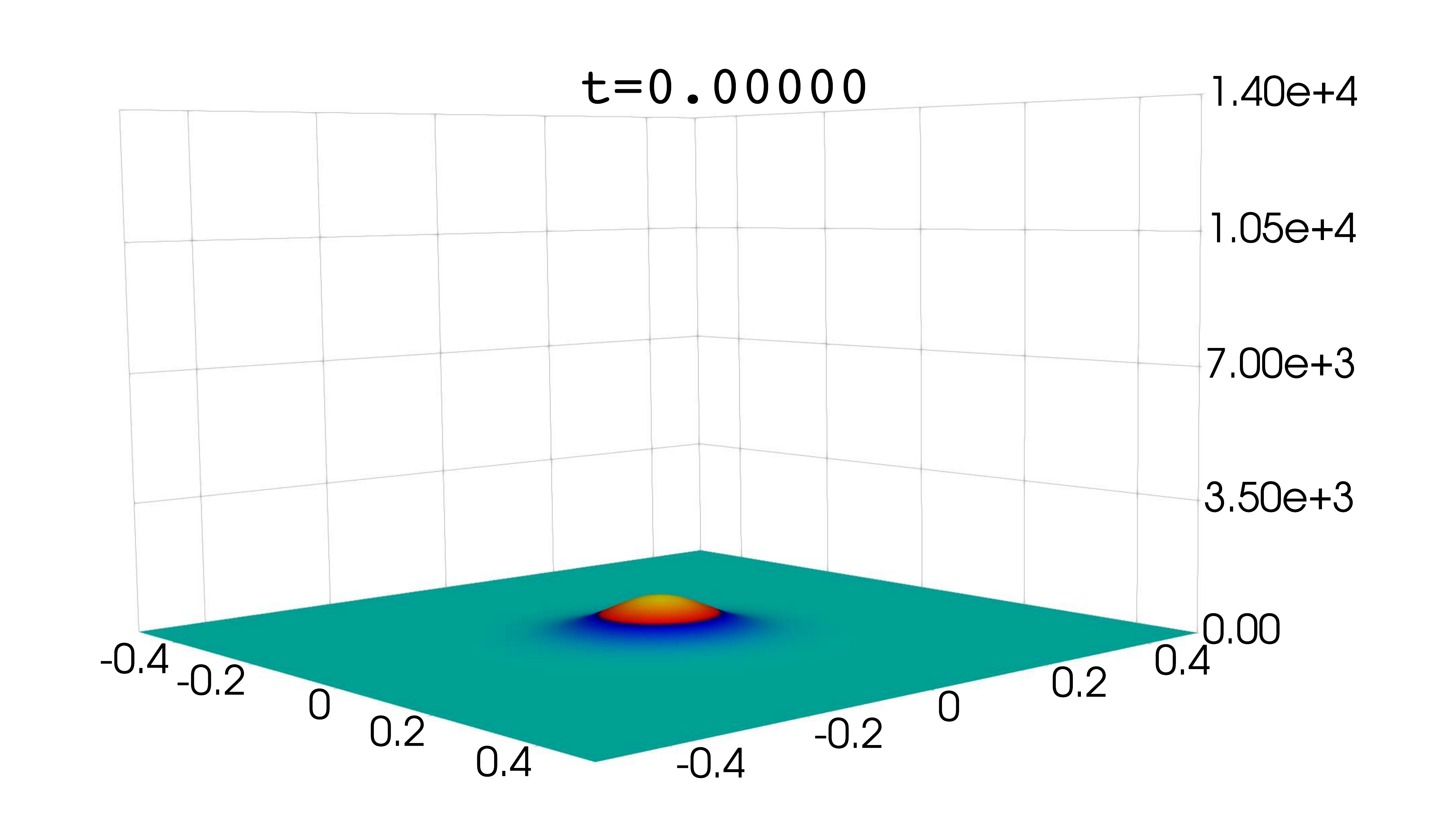}
	\end{subfigure}
	\begin{subfigure}{0.49\textwidth}
		\centering
		\boldmath{$v_0$}
		
		\includegraphics[scale=0.11]{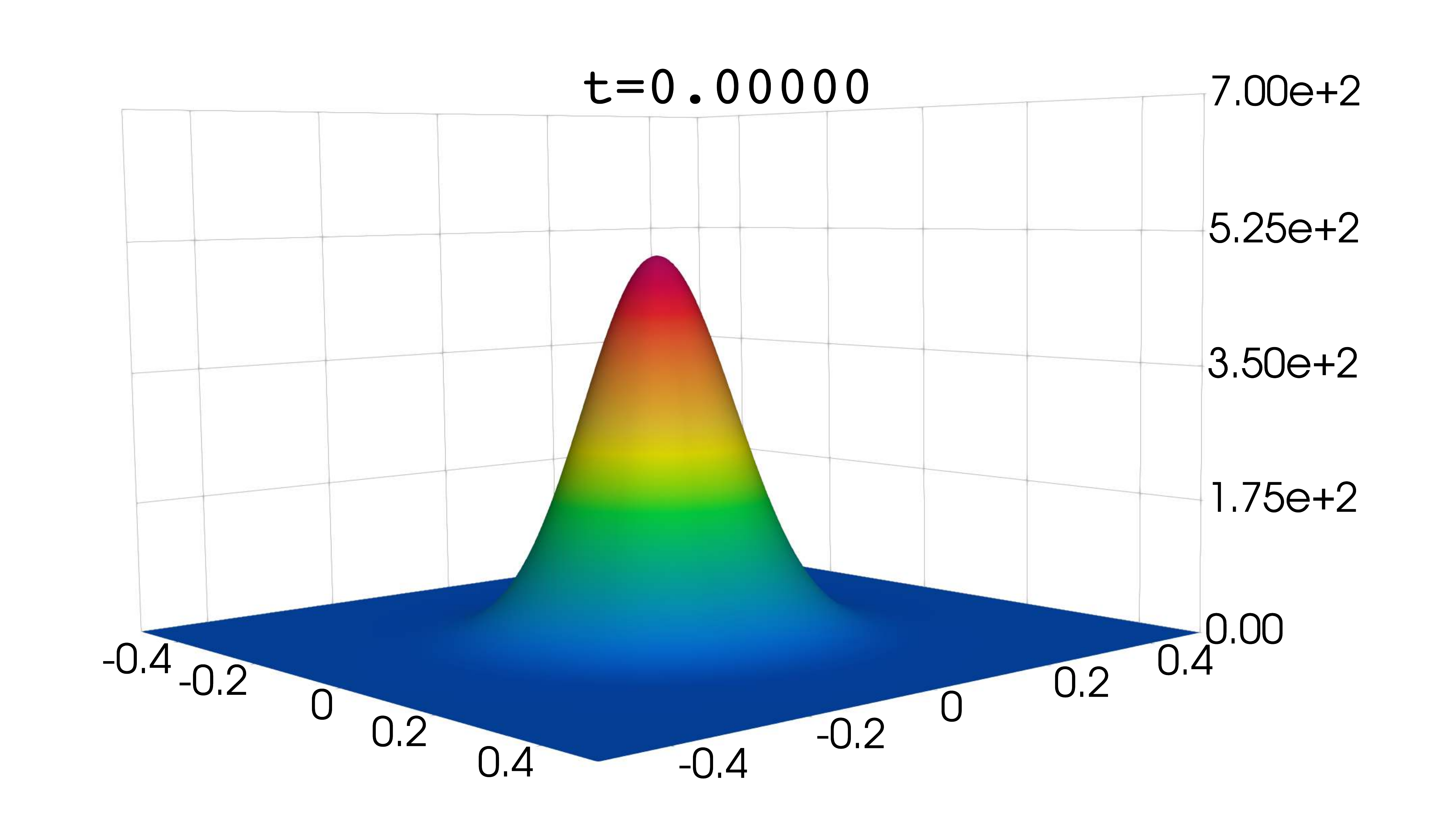}
	\end{subfigure}
	\caption{Initial conditions for blow-up as in \cite{chertock_second-order_2008} (different scales are used for $u$ and $v$).}
	\label{fig:ic_chertock}
\end{figure}

As stated in \cite{chertock_second-order_2008}, the $u$ and $v$ components of the solution are expected to blow up in a finite time due to the initial conditions chosen. The result of the test using the scheme \eqref{esquema_DG_upw_KS_non_truncated} with $h\approx 1.41\cdot 10^{-3}$ and $\Delta t=10^{-6}$ is shown in Figures~\ref{fig:u_chertock} and \ref{fig:v_chertock}. In fact, we observe a blow-up phenomenon for a certain finite time in the range conjectured by A. Chertock and A. Kurganov in \cite{chertock_second-order_2008},  $t^*\in(4.4\cdot 10^{-5},10^{-4})$, as our discrete approximation reaches values of order $10^6$ in this time interval.

\begin{figure}
	\centering
	\boldmath{$u$}
	
	\begin{subfigure}{0.49\textwidth}
		\centering
		\includegraphics[scale=0.11]{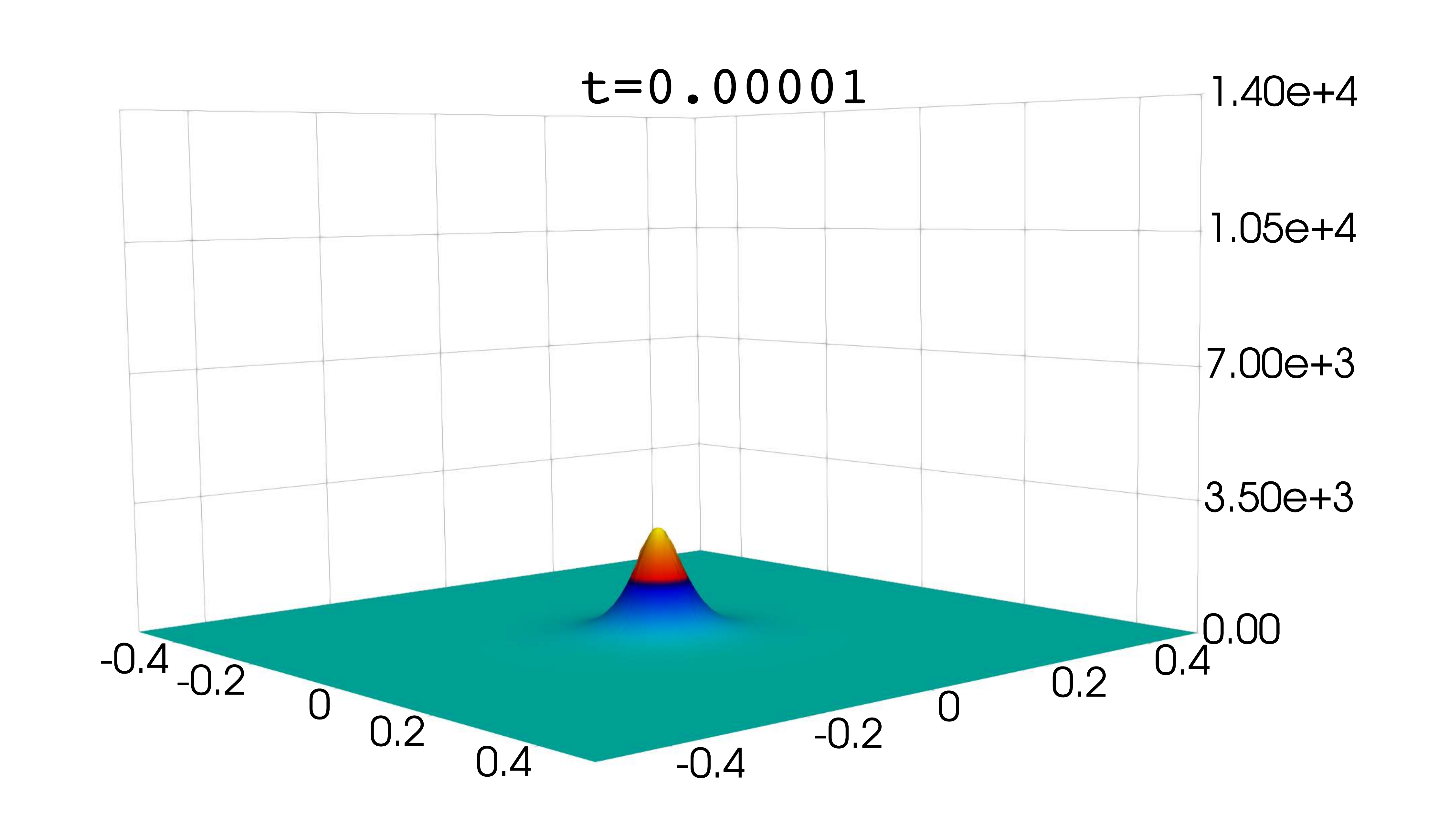}
	\end{subfigure}
	\begin{subfigure}{0.49\textwidth}
		\centering
		\includegraphics[scale=0.11]{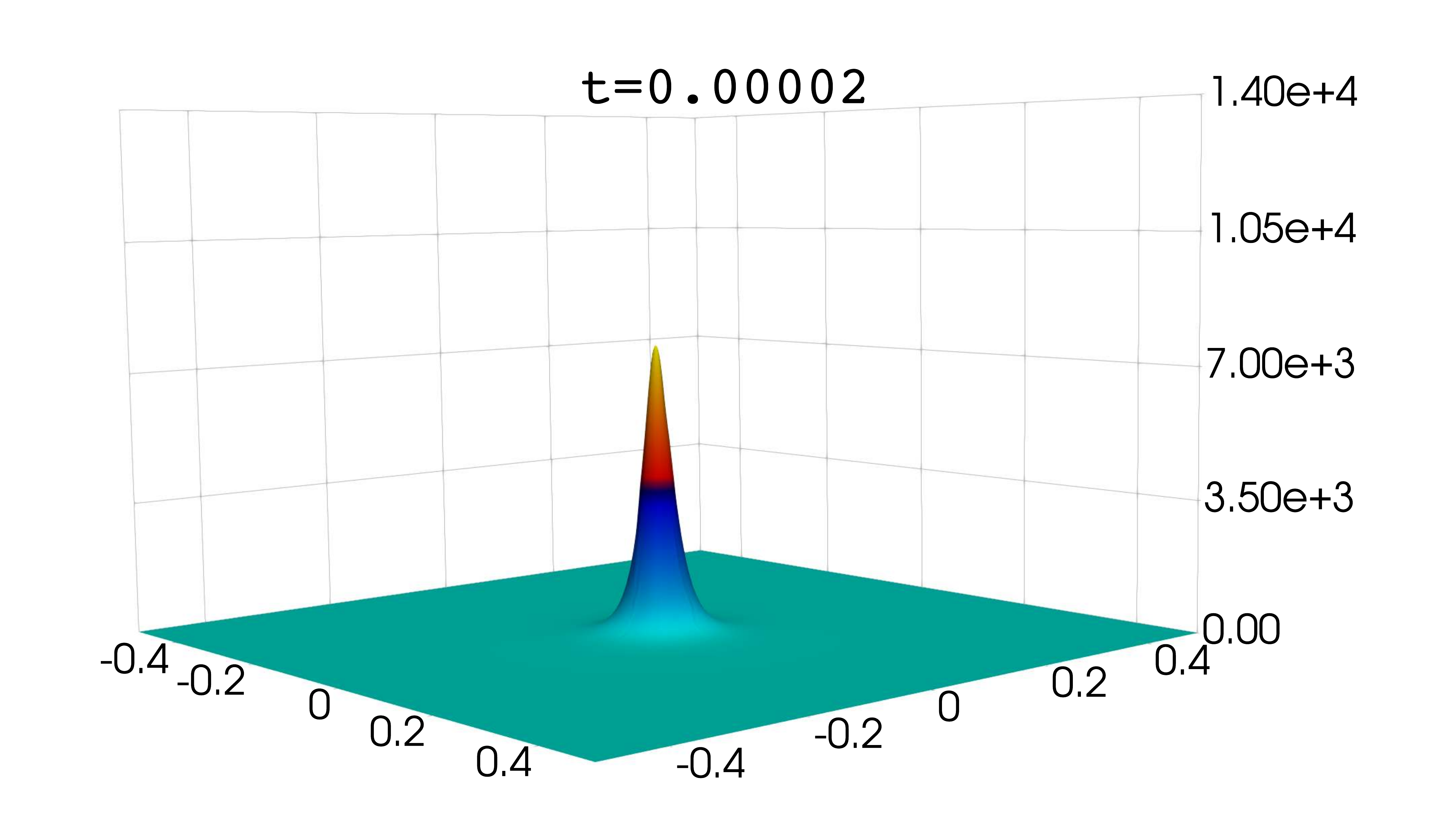}
	\end{subfigure}
	\begin{subfigure}{0.49\textwidth}
		\centering
		\includegraphics[scale=0.11]{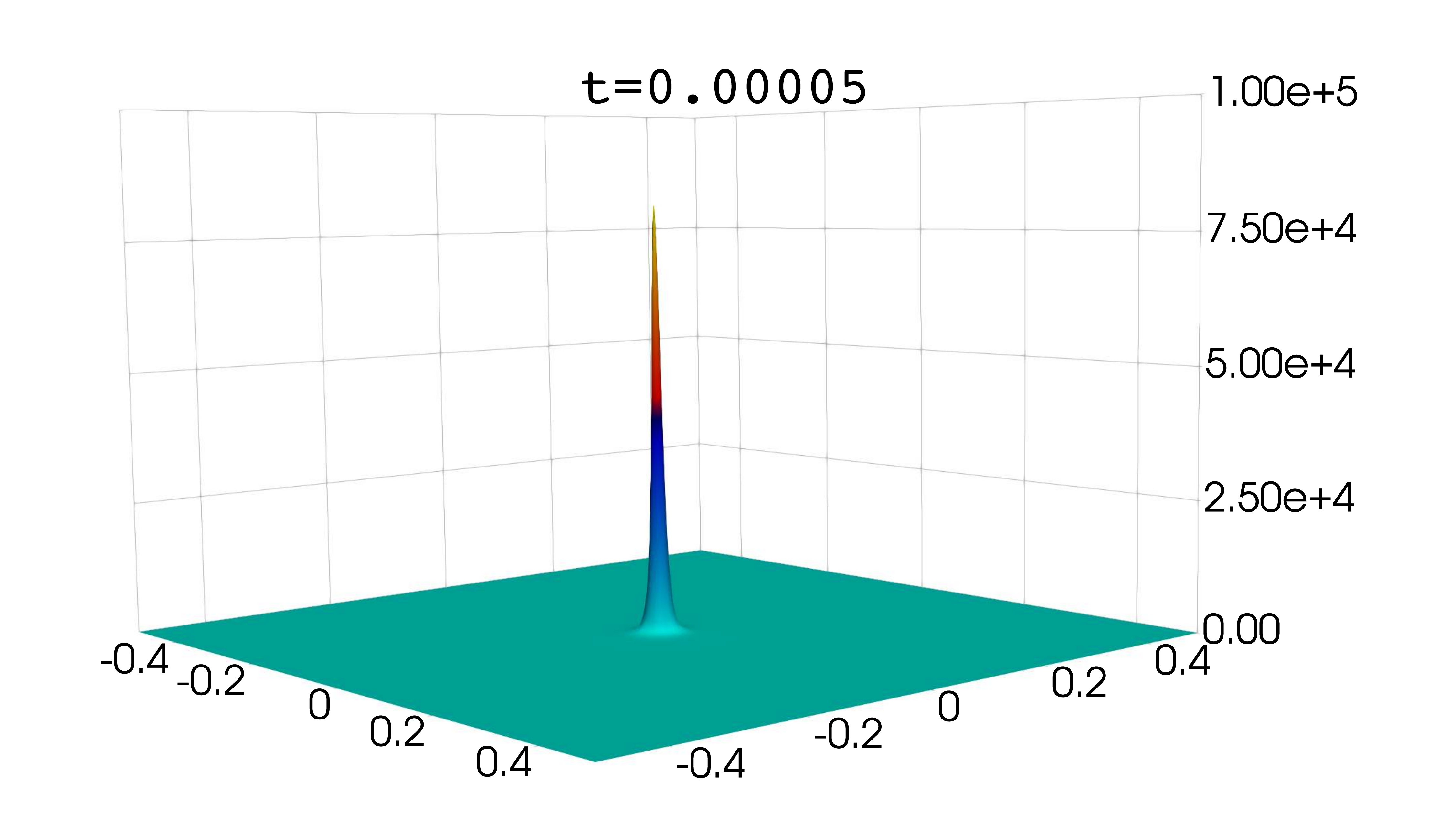}
	\end{subfigure}
	\begin{subfigure}{0.49\textwidth}
		\centering
		\includegraphics[scale=0.11]{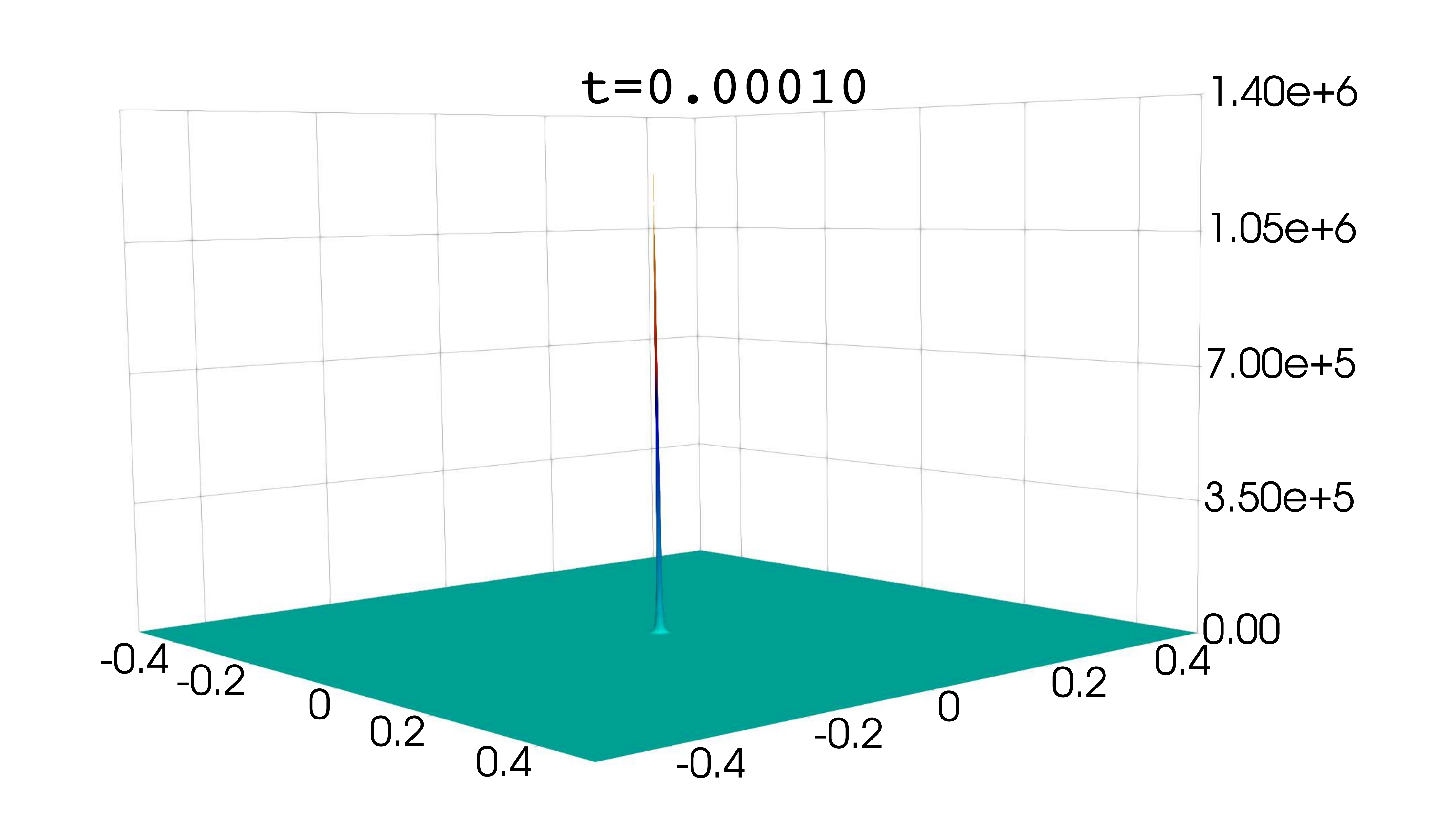}
	\end{subfigure}
	\caption{Blow-up of $u$ as in \cite{chertock_second-order_2008}.}
	\label{fig:u_chertock}
\end{figure}

\begin{figure}
	\centering
	\boldmath{$v$}
	
	\begin{subfigure}{0.49\textwidth}
		\centering
		\includegraphics[scale=0.11]{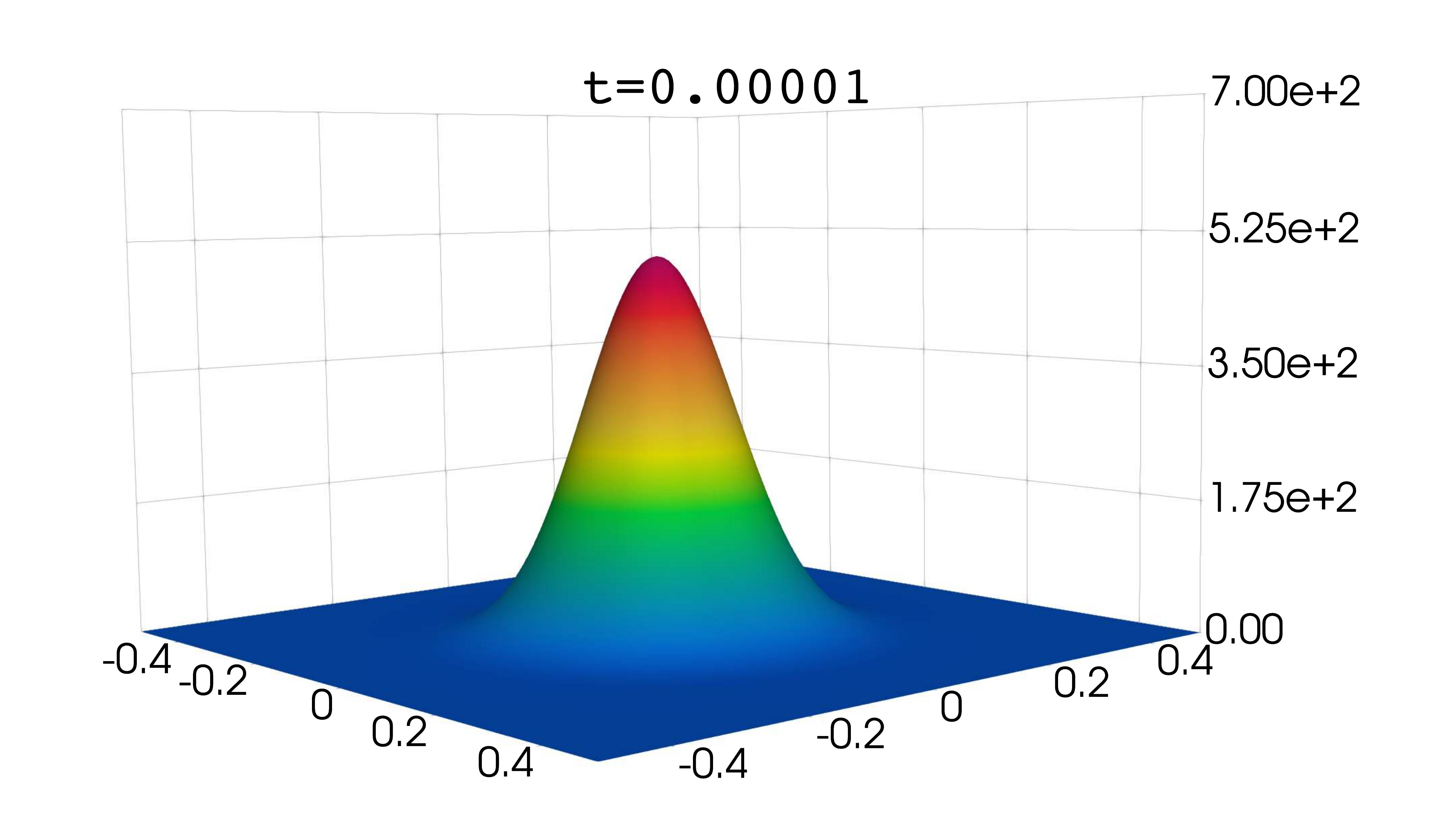}
	\end{subfigure}
	\begin{subfigure}{0.49\textwidth}
		\centering
		\includegraphics[scale=0.11]{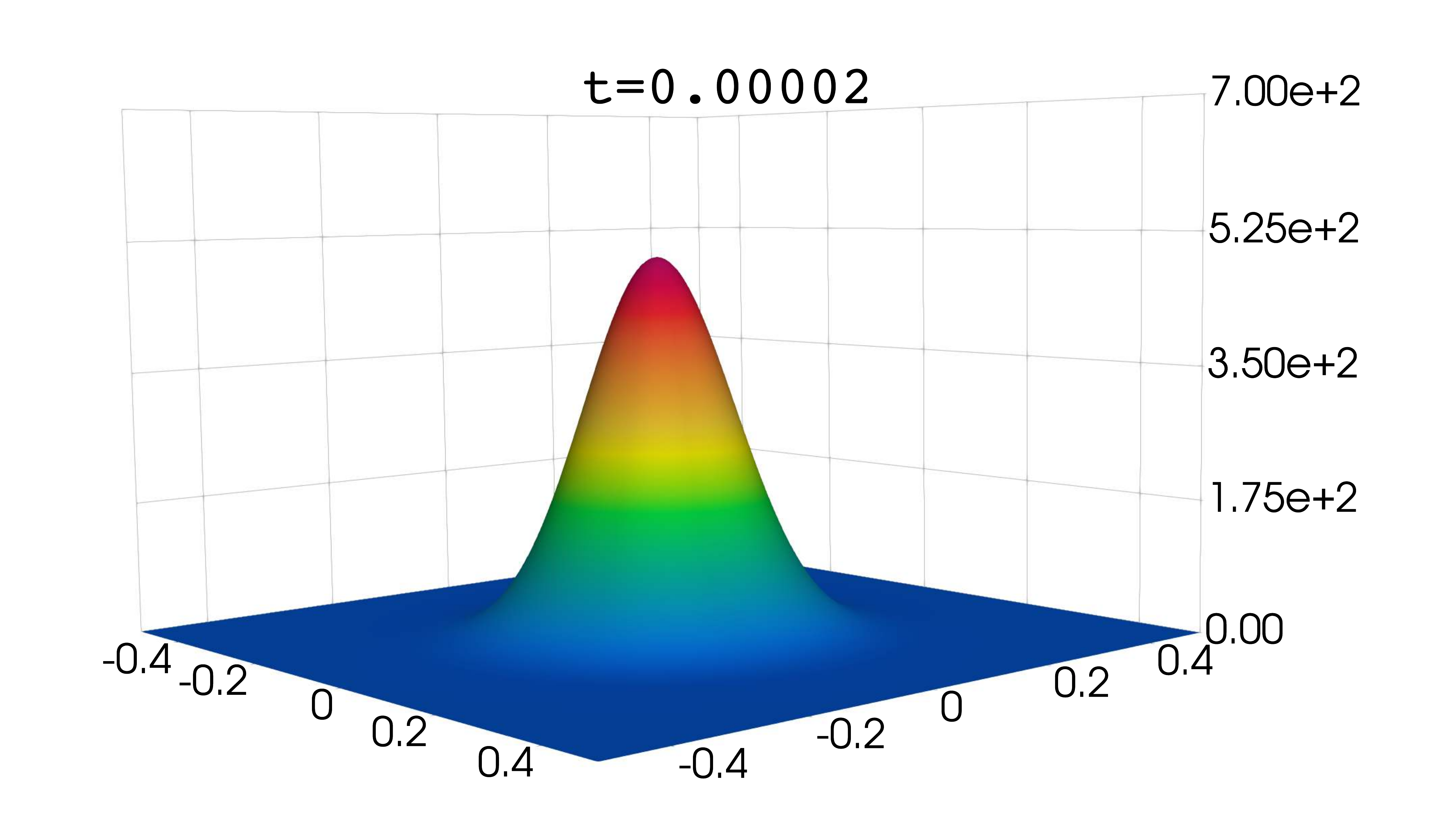}
	\end{subfigure}
	\begin{subfigure}{0.49\textwidth}
		\centering
		\includegraphics[scale=0.11]{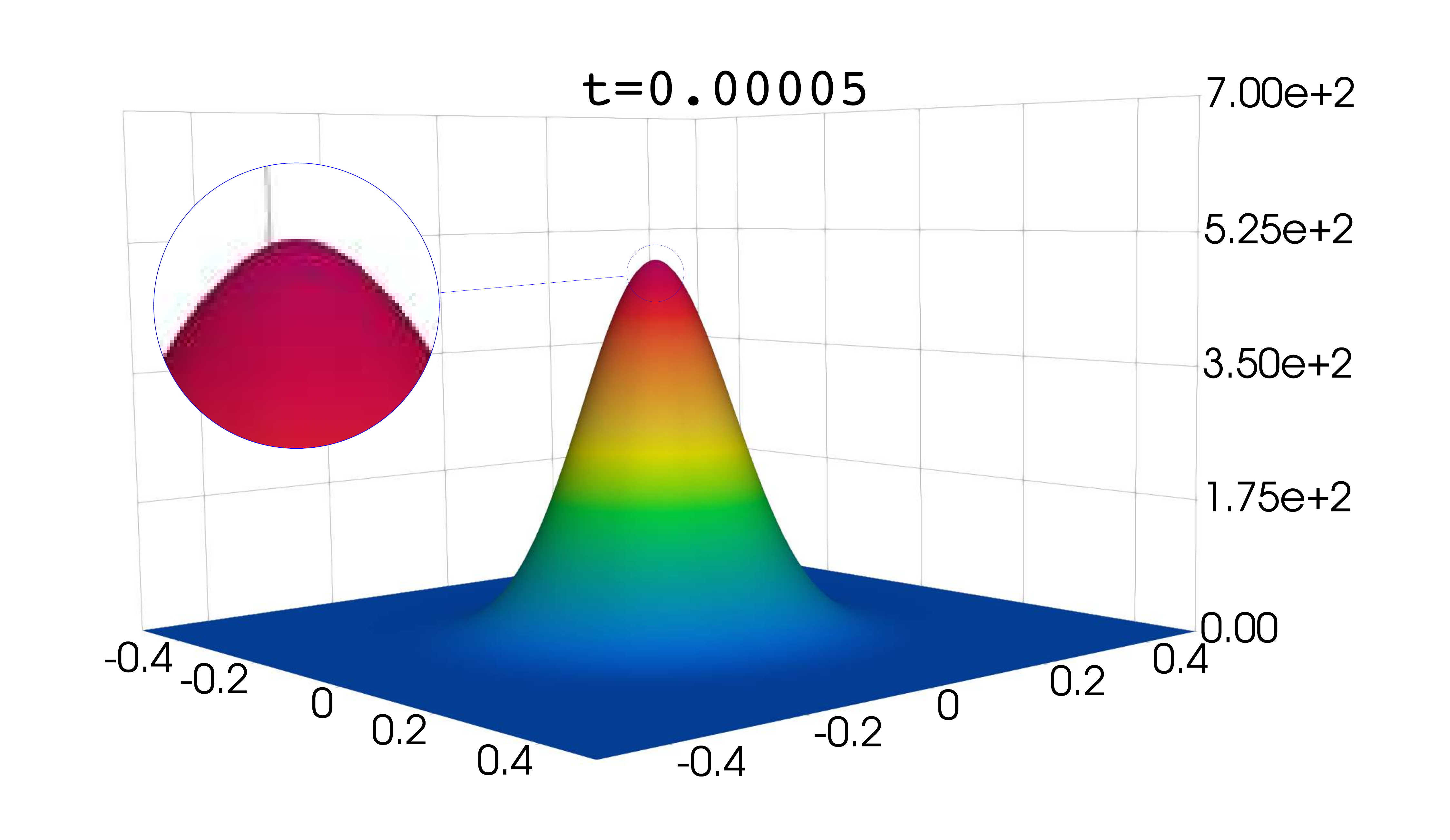}
	\end{subfigure}
	\begin{subfigure}{0.49\textwidth}
		\centering
		\includegraphics[scale=0.11]{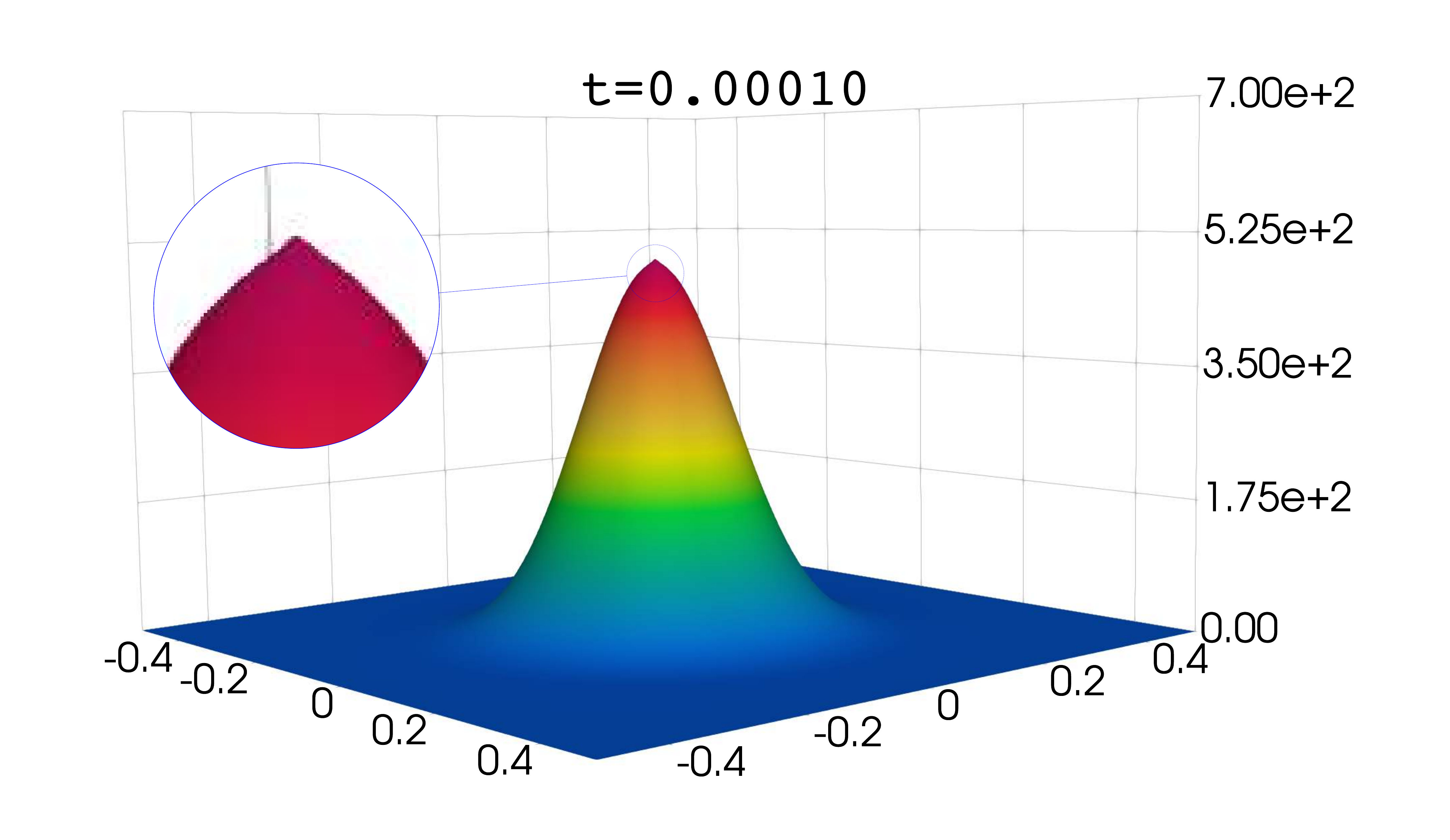}
	\end{subfigure}
	\caption{Aggregation of $v$ in the test in \cite{chertock_second-order_2008}.}
	\label{fig:v_chertock}
\end{figure}

Moreover, the positivity is preserved for both $u$ and $v$ as stated in Theorem~\ref{thm:principio_del_maximo_DG_KS} and, unlike the scheme presented in \cite{chertock_second-order_2008}, we are certain that the discrete energy decreases (both for $E(\cdot, \cdot)$ and $E(\cdot, \cdot)$) using the scheme \eqref{esquema_DG_upw_KS_non_truncated} as proved in Theorem~\ref{thm:energia_esquema_KS}. See Figures~\ref{fig:min-max_chertock} and~\ref{fig:energy_chertock}.

	\begin{remark}
		This test have been computed with greater and lower values of $\varepsilon$ including the limiting case $\varepsilon=0$. The difference in norms $L^2$ and $L^\infty$ are shown in Table~\ref{table:test_1_eps}. From a qualitative point of view, the solutions are indistinguishable.

		In this case, the numerical approximation works with $\varepsilon=0$ since the minimum of $u$ remains strictly positive and does not tend to $0$, it takes values around $10^{-19}$ during all the iterations computed. Below, in Remark~\ref{rmk:sin-cos}, we show a different test where we do have to take $\varepsilon>0$ to ensure convergence of the scheme.
	\end{remark}

\begin{table}
	\centering
	% \scriptsize
	\begin{tabular}{||c||c|c||c|c||}
		\hline
		\multirow{2}*{$\varepsilon$} & \multicolumn{2}{|c||}{$\norma{\cdot}_{L^2}$} & \multicolumn{2}{|c||}{$\norma{\cdot}_{L^\infty}$}\\
		\cline{2-5}
		& $u$ & $v$ & $u$ & $v$\\
		\hline\hline
		$10^{-6}$ & $4.81\cdot 10^{-8}$ & $1.01\cdot 10^{-12}$ & $3.52\cdot 10^{-6}$ & $2.90\cdot 10^{-11}$ \\
		\hline
		$10^{-10}$ & $1.10\cdot 10^{-10}$ & $3.87\cdot 10^{-13}$ & $1.71\cdot 10^{-8}$ & $2.73\cdot 10^{-12}$ \\
		\hline
		$10^{-14}$ & $8.75\cdot 10^{-11}$ & $6.85\cdot 10^{-14}$ & $1.59\cdot 10^{-8}$ & $8.53\cdot 10^{-13}$ \\
		\hline
	\end{tabular}
	\caption{Difference between approximations of the test in subsection~\ref{test1} at $t=5\cdot 10^{-5}$ with respect to the solution with $\varepsilon=0$.}
	\label{table:test_1_eps}
\end{table}

An accuracy test in space (with $\varepsilon=10^{-10}$) has been also carried out where the solution obtained with $h\approx 7.071\cdot 10^{-4}$ and $\Delta t=10^{-6}$ has been taken as reference solution. The results shown in Tables~\ref{table:test_1_convergence_u} and~\ref{table:test_1_convergence_v} suggest first order of convergence in space in norm $L^2$ both for $u$ and $v$.

It is remarkable to notice that, although the $L^2$ errors of the approximation of $u$ may seem huge at first, particularly as it approaches the blow-up time, they are not that big in relative terms. As it can be observed in Figure~\ref{fig:u_chertock}, the maximum value reached by $u$ is around $10^3$ bigger than the $L^2$ errors shown in Table~\ref{table:test_1_convergence_u} at each time step. These errors will tend to vanish as the mesh is refined so that the spiky bulge in the middle of the domain is more accurately approximated.

Also, we would like to emphasize the difficulty of achieving such results as obtaining a reference solution in a blow-up situation where the exact solution tends to degenerate and huge gradients appear require a significant computational effort. In this regard, the reference solution has been computed in parallel using a domain decomposition technique.

\begin{table}
	\centering
	% \scriptsize
	\begin{tabular}{||c|c|c|c|c|c|c|c||}
		\hline
		\multirow{2}{*}{$t$} & $h\approx 1.41\cdot 10^{-2}$ & \multicolumn{2}{c|}{$5h/7\approx 1.01\cdot 10^{-2}$} & \multicolumn{2}{c|}{$5h/9\approx 7.86\cdot 10^{-3}$} & \multicolumn{2}{c|}{$5h/11\approx 6.43\cdot 10^{-3}$} \\
		\cline{2-8}
		&Error &Error & Order &  Error & Order &  Error & Order \\
		\hline
		\hline
		$10^{-5}$  & $7.01$ & $5.08$ & $1.00$  &  $4.13$ & $0.83$ & $3.41$ & $0.95$ \\
		\hline
		\hline
        $2\cdot10^{-5}$  & $2.38\cdot 10$ & $1.70\cdot 10$ & $0.99$ & $1.42\cdot 10$ & $0.71$ & $1.09\cdot 10$ & $1.33$ \\
		\hline\hline
		$5\cdot10^{-5}$  & $3.15\cdot 10^2$ & $2.31\cdot 10^{2}$ & $0.92$ & $1.84\cdot 10^2$ & $0.92$ & $1.28\cdot 10^2$ & $1.81$ \\
		\hline
	\end{tabular}
	\captionof{table}{Accuracy test in norm $L^2$ for $u$ (test in subsection~\ref{test1}).}
	\label{table:test_1_convergence_u}
\end{table}

\begin{table}
	\centering
	% \scriptsize
	\begin{tabular}{||c|c|c|c|c|c|c|c||}
		\hline
		\multirow{2}{*}{$t$} & $h\approx 1.41\cdot 10^{-2}$ & \multicolumn{2}{c|}{$5h/7\approx 1.01\cdot 10^{-2}$} & \multicolumn{2}{c|}{$5h/9\approx 7.86\cdot 10^{-3}$} & \multicolumn{2}{c|}{$5h/11\approx 6.43\cdot 10^{-3}$} \\
		\cline{2-8}
		&Error &Error & Order &  Error & Order &  Error & Order \\
		\hline
		\hline
		$10^{-5}$  & $1.43\cdot 10^{-2}$ & $1.04\cdot 10^{-2}$ & $0.95$ & $7.36\cdot 10^{-3}$& $1.36$ &$5.63\cdot 10^{-3}$ & $1.33$ \\
		\hline
		\hline
        $2\cdot10^{-5}$  & $2.04\cdot 10^{-2}$ & $1.27\cdot 10^{-2}$ & $1.40$ & $8.82\cdot 10^{-3}$ & $1.46$ & $6.76\cdot 10^{-3}$ & $1.32$ \\
		\hline\hline
		$5\cdot10^{-5}$  & $2.54\cdot 10^{-2}$ & $1.51\cdot 10^{-2}$ & $1.55$ & $1.27\cdot 10^{-2}$ & $0.68$ & $9.94\cdot 10^{-3}$ & $1.22$\\
		\hline
	\end{tabular}
	\captionof{table}{Accuracy test in norm $L^2$ for $v$ (test in subsection~\ref{test1}).}
	\label{table:test_1_convergence_v}
\end{table}

\begin{figure}
	\centering
	\begin{subfigure}{0.49\textwidth}
		\centering
		\boldmath{$u$}
		
		\includegraphics[scale=0.5]{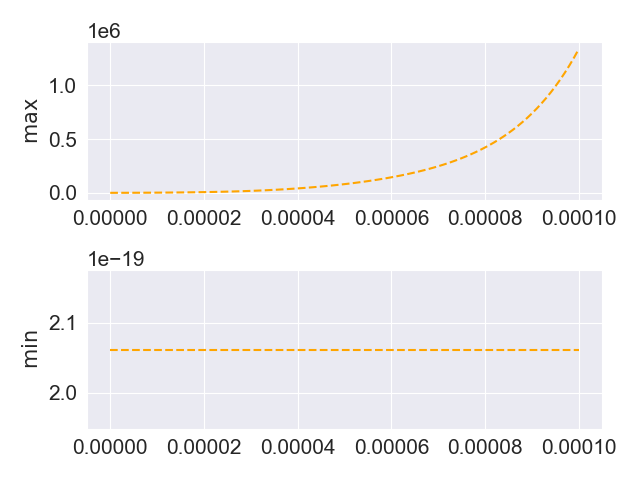}
	\end{subfigure}
	\begin{subfigure}{0.49\textwidth}
		\centering
		\boldmath{$v$}
		
		\includegraphics[scale=0.5]{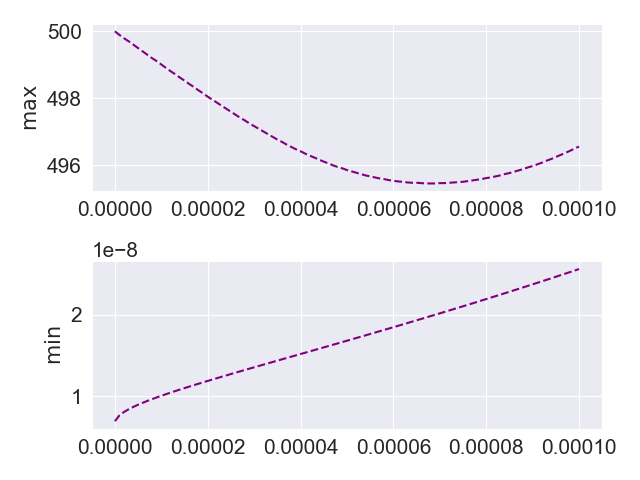}
	\end{subfigure}
	\caption{Minimum and maximum of $u$ and $v$ over time in the case shown in \cite{chertock_second-order_2008}.}
	\label{fig:min-max_chertock}
\end{figure}

\begin{figure}
	\centering
	\includegraphics[scale=0.5]{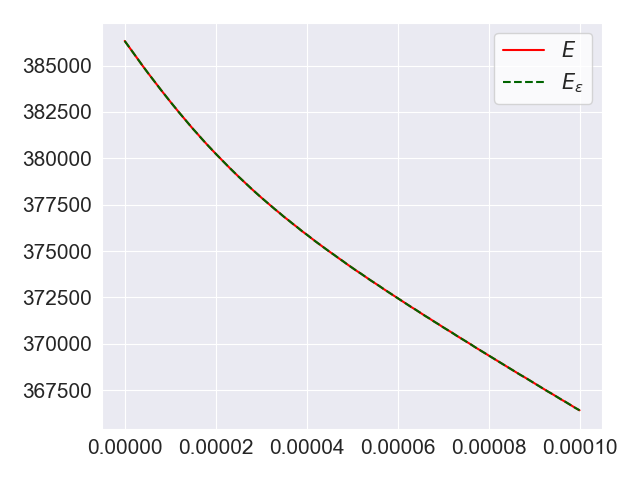}
	\caption{Discrete energy over time in the case shown in \cite{chertock_second-order_2008}.}
	\label{fig:energy_chertock}
\end{figure}

\subsection{Three bulges of cells}
\label{test2}
Now, we show the results for a similar test to the one that appears in \cite{saito_conservative_2007}. In this case, we take the parabolic-elliptic case ($\tau=0$) so that the characteristic speed of $v$ is much faster than the characteristic speed of $u$. Moreover, we consider the initial condition
\begin{equation*}
u_0 = 900 e^{-100((x-0.2)^2 + y^2)} + 800 e^{-100(x^2 + (y-0.2)^2)} + 1000 e^{-100 ((x-0.3)^2 + (y-0.3)^2)},
\end{equation*}
which is plotted in Figure~\ref{fig:ic_saito}. As stated in \cite{saito_conservative_2007} and the references therein, since $\norma{u_0}_{L^1(\Omega)}>8\pi$, the solution is expected to blow-up in finite time.

\begin{figure}
	\centering
	\boldmath{$u_0$}
	
	\includegraphics[scale=0.11]{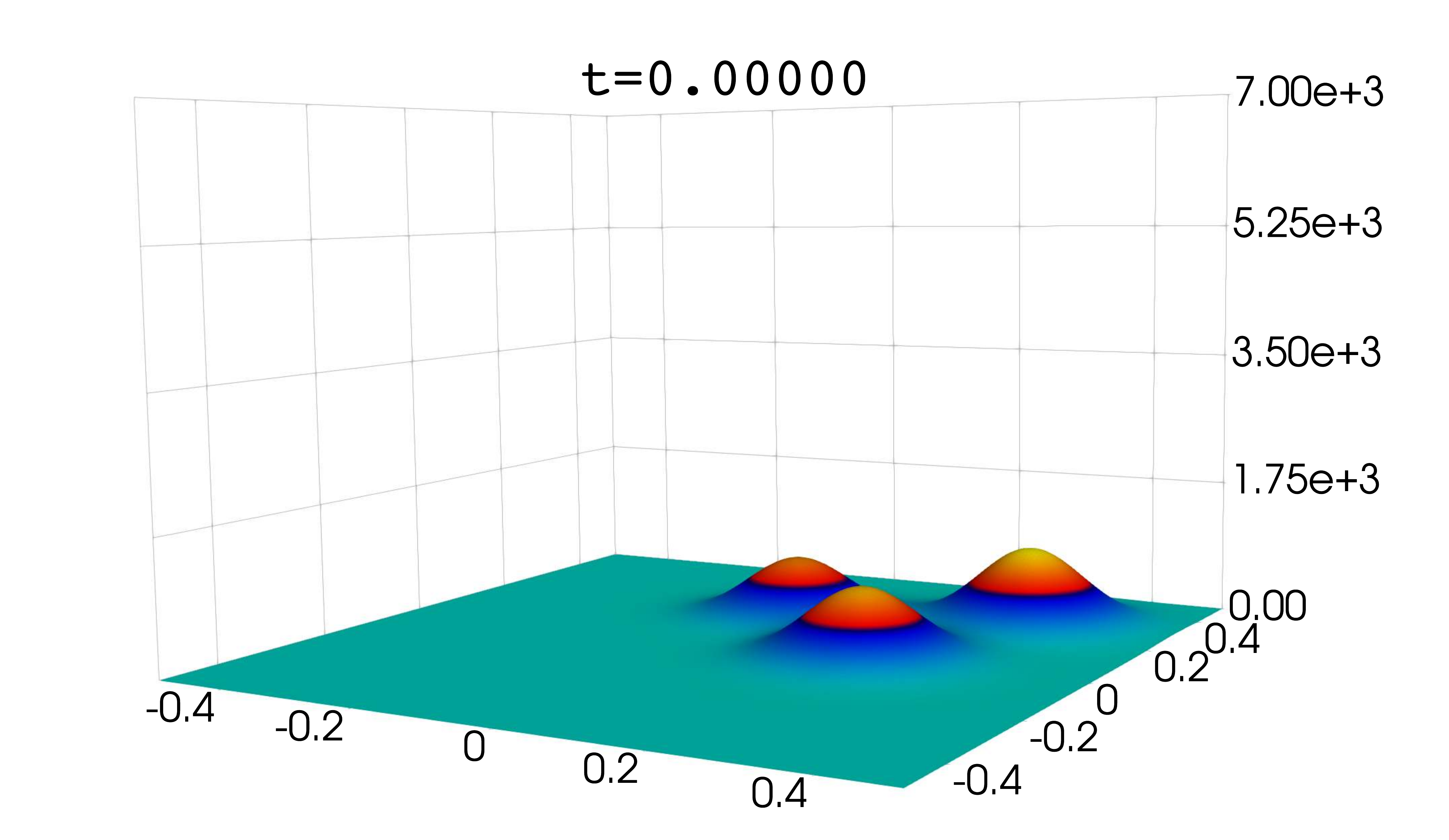}
	\caption{Initial condition with three cell bulges (similar to the one in \cite{chertock_second-order_2008}).}
	\label{fig:ic_saito}
\end{figure}

In Figures~\ref{fig:u_saito_1} and \ref{fig:v_saito_1} we can observe the result of the test with $h\approx2.83\cdot 10^{-2}$ and $\Delta t=10^{-5}$. In this case, the qualitative behavior of the solution is similar to the one shown in \cite{saito_conservative_2007}, with the peak of cells moving towards a corner of the domain. However, the qualitative behavior of the solution is different if we take $h\approx7.07\cdot 10^{-3}$ and $\Delta t=10^{-5}$. Now, as represented in Figures~\ref{fig:u_saito_2} and \ref{fig:v_saito_2}, a blow-up phenomenon seems to occur in finite time and the peak of cells remains motionless far away from the corners of the domain.

\begin{figure}
	\centering
	\boldmath{$u$}
	
	\begin{subfigure}{0.49\textwidth}
		\centering
		\includegraphics[scale=0.11]{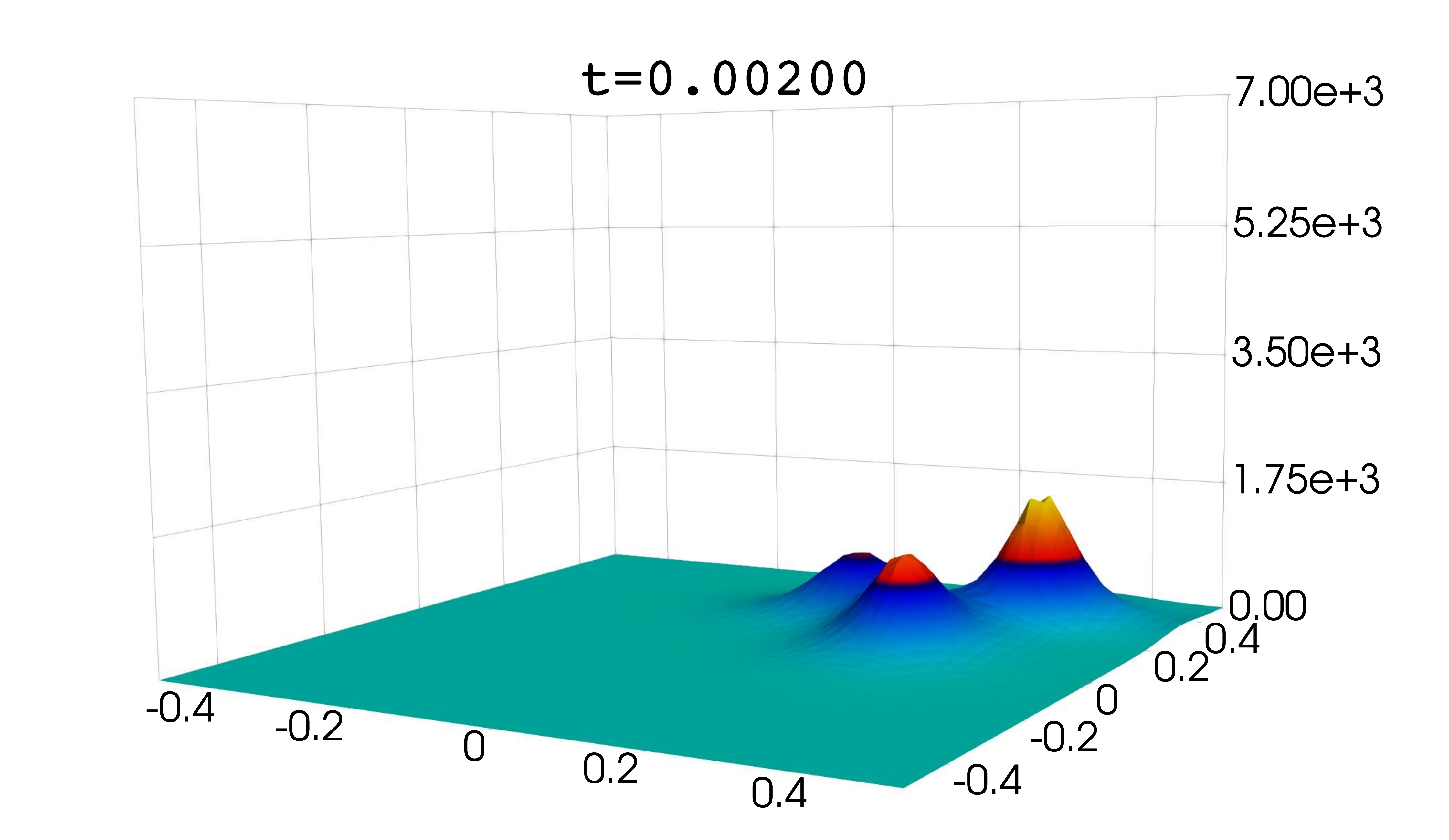}
	\end{subfigure}
	\begin{subfigure}{0.49\textwidth}
		\centering
		\includegraphics[scale=0.11]{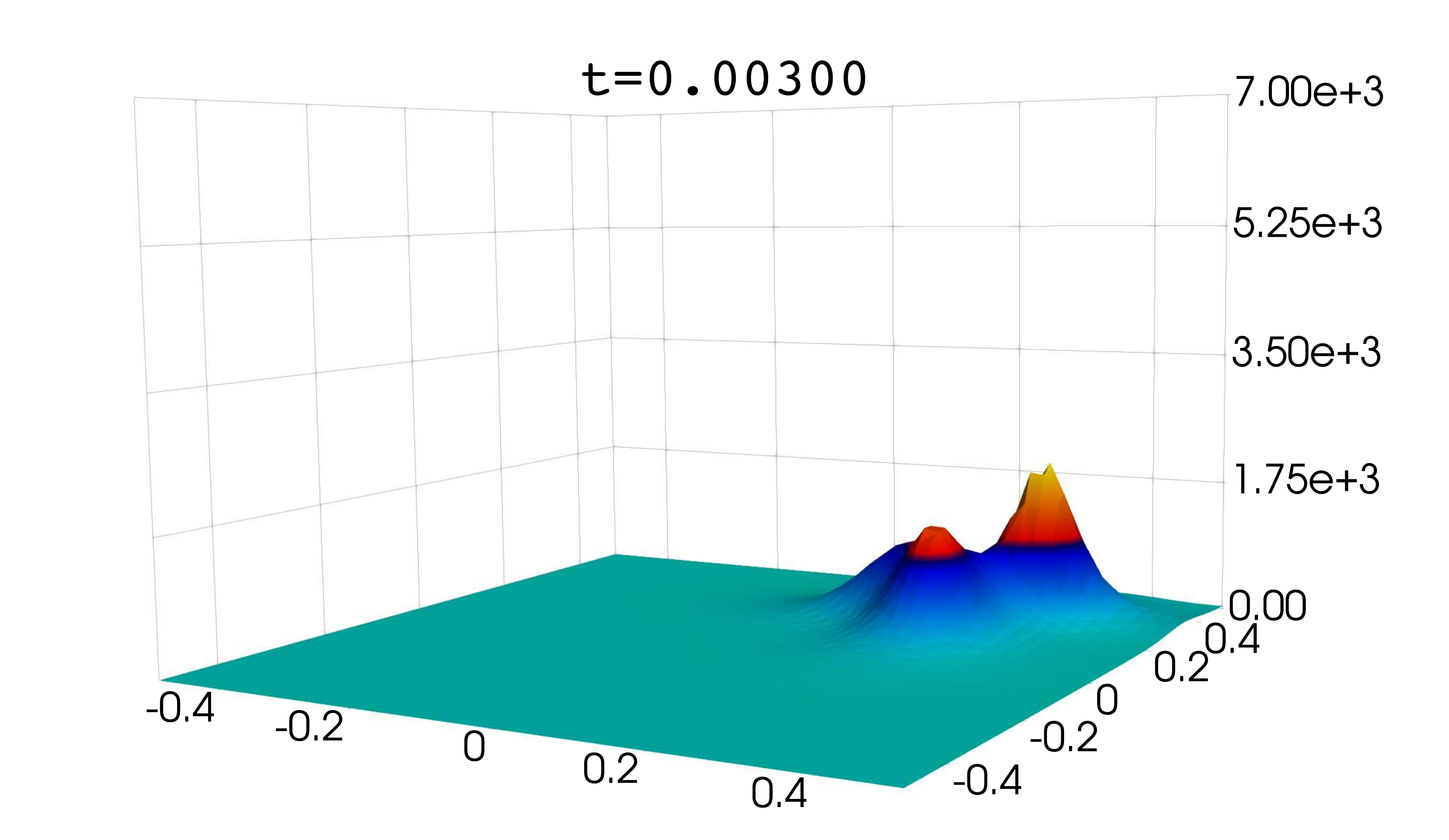}
	\end{subfigure}
	\begin{subfigure}{0.49\textwidth}
		\centering
		\includegraphics[scale=0.11]{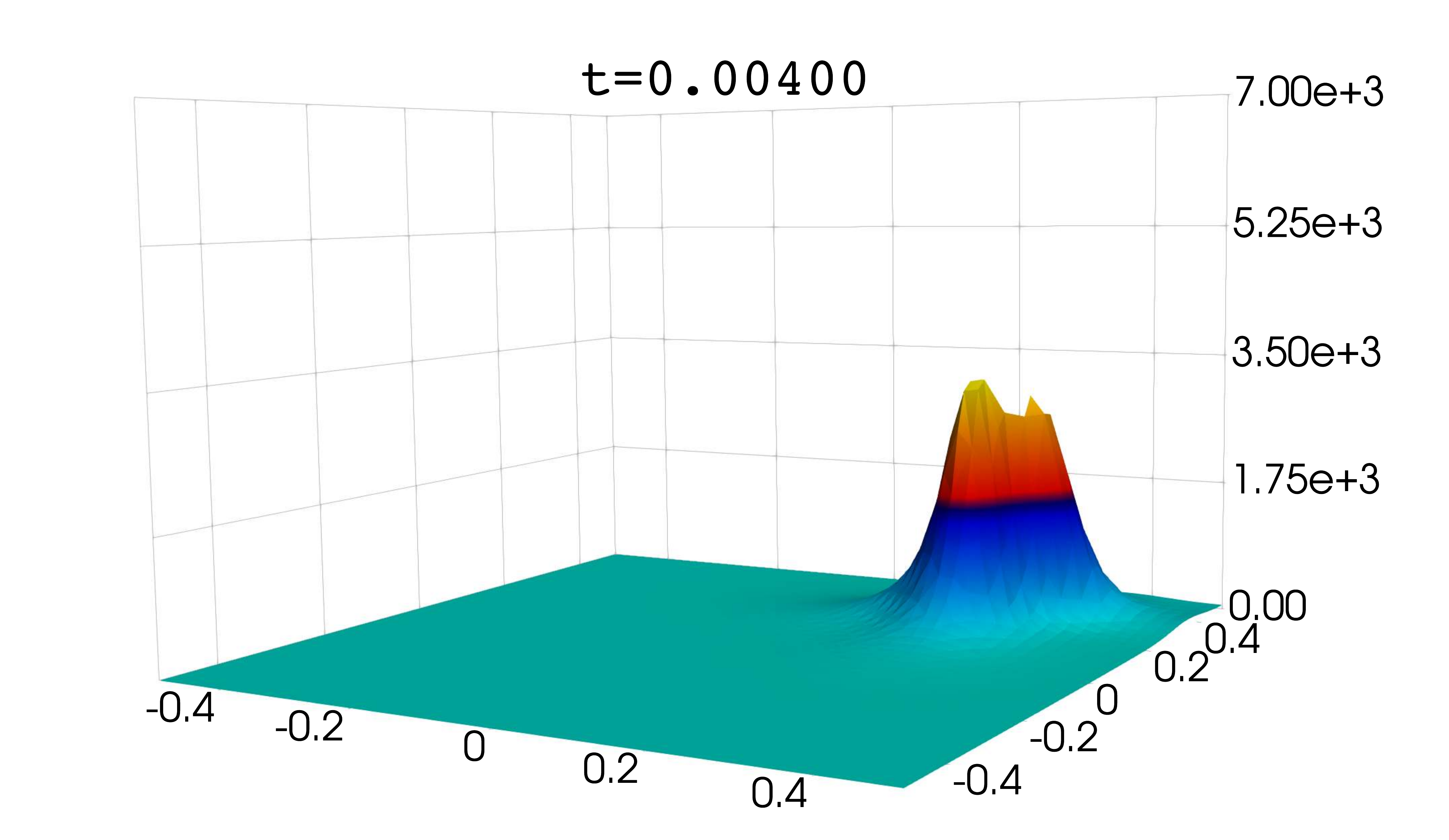}
	\end{subfigure}
	\begin{subfigure}{0.49\textwidth}
		\centering
		\includegraphics[scale=0.11]{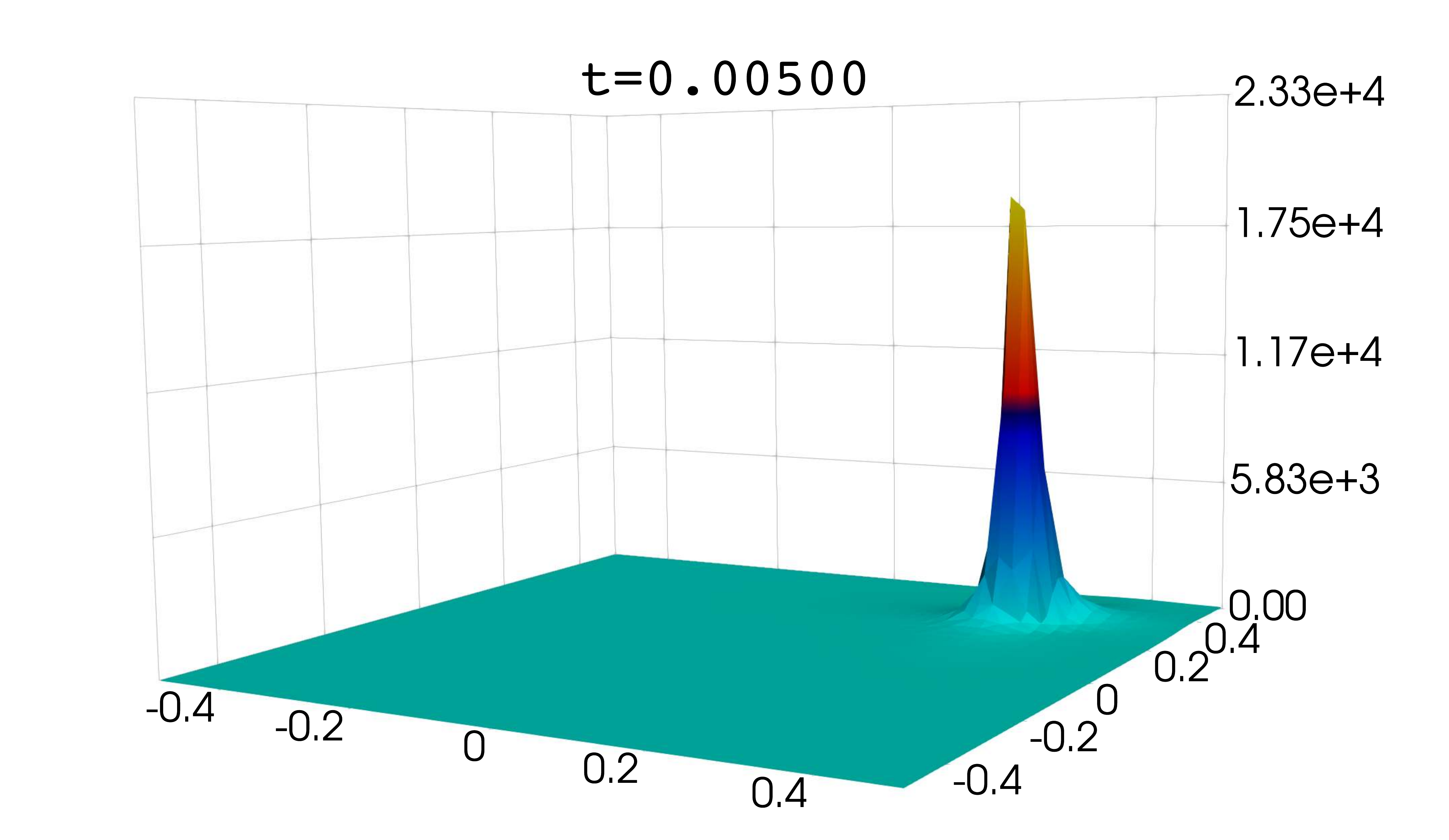}
	\end{subfigure}
	\begin{subfigure}{0.49\textwidth}
		\centering
		\includegraphics[scale=0.11]{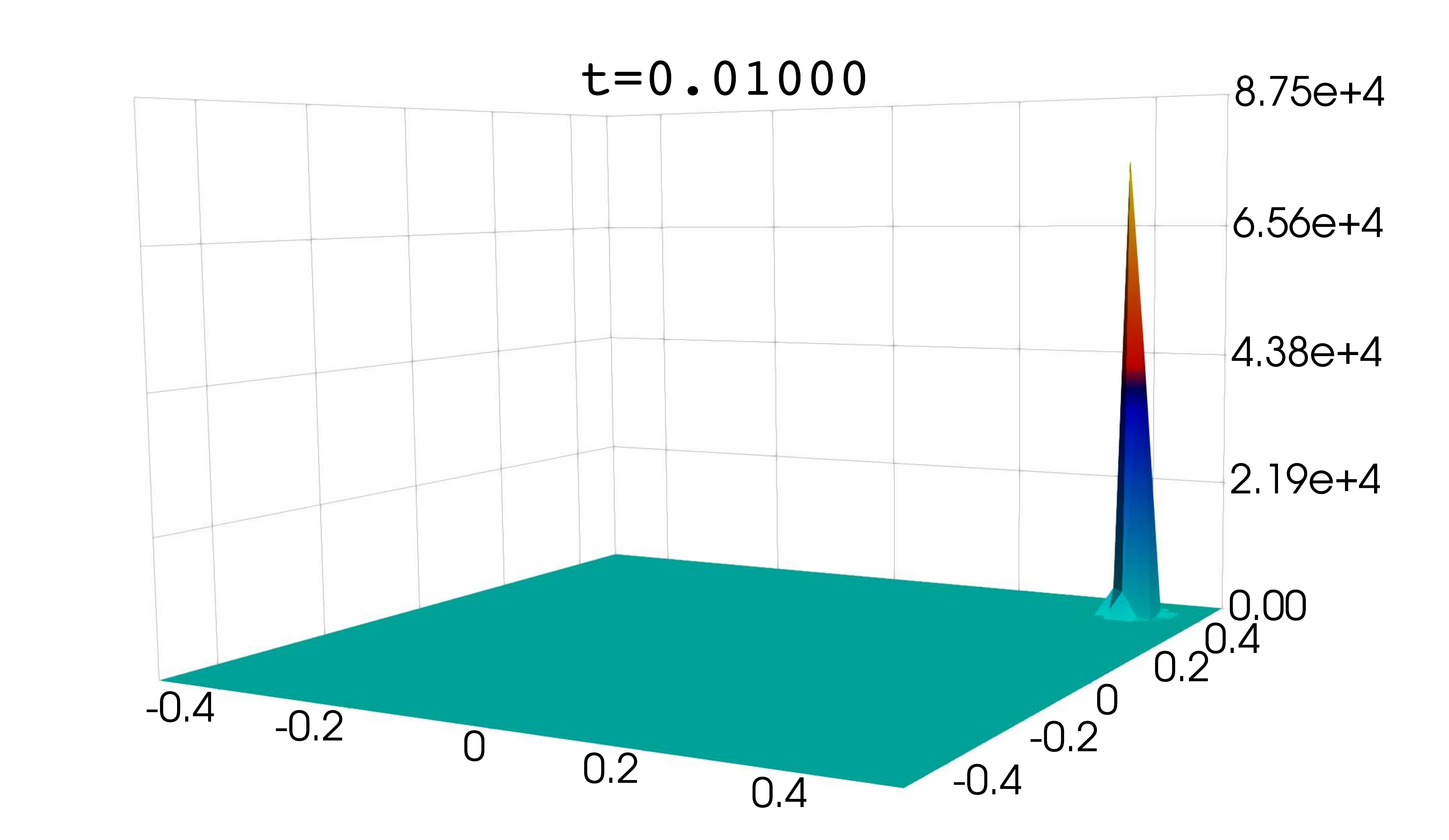}
	\end{subfigure}
	\begin{subfigure}{0.49\textwidth}
		\centering
		\includegraphics[scale=0.11]{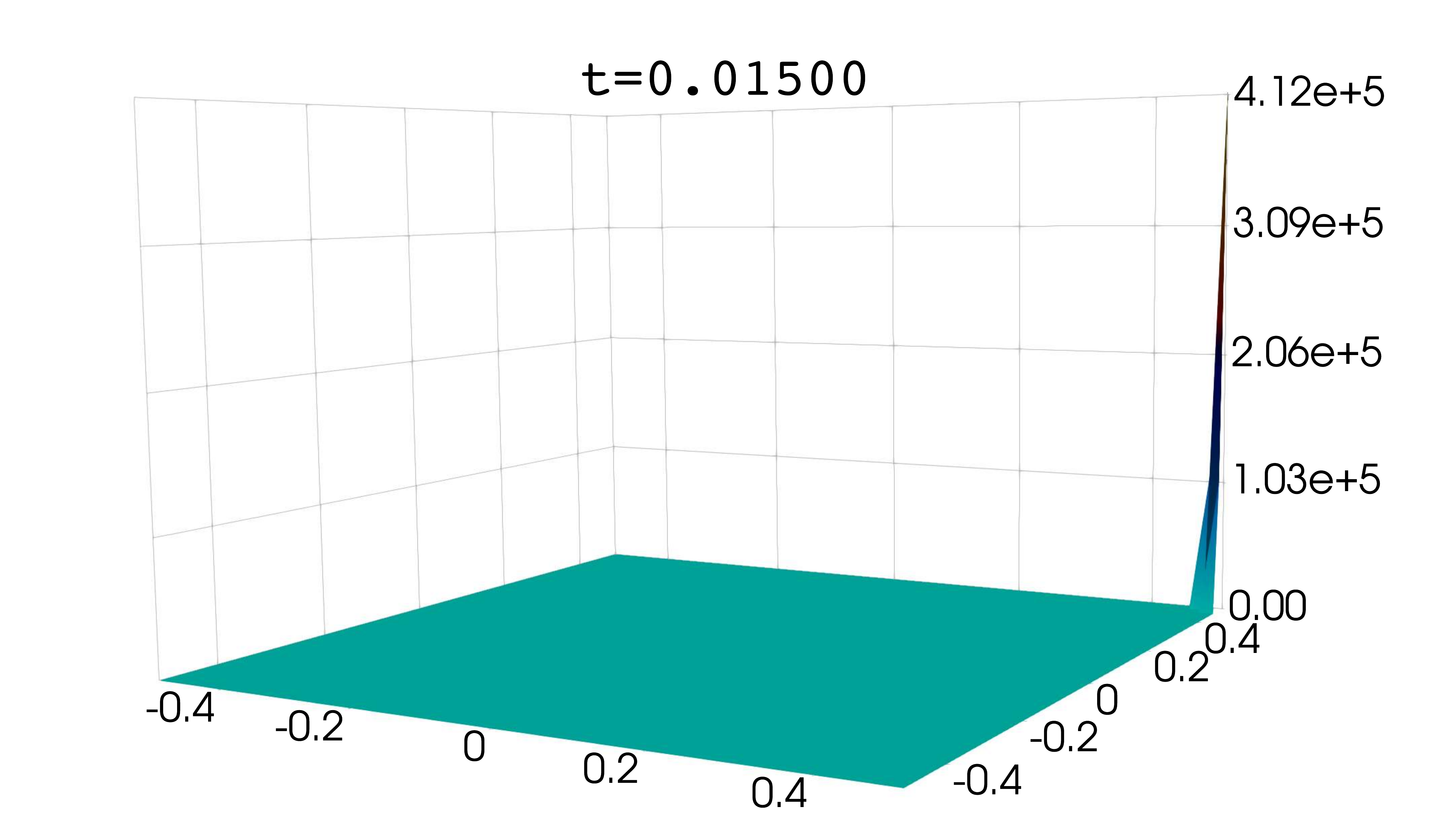}
	\end{subfigure}
	\caption{Aggregation of three cell bulges with $h\approx2.83\cdot 10^{-2}$.}
	\label{fig:u_saito_1}
\end{figure}

\begin{figure}
	\centering
	\boldmath{$v$}
	
	\begin{subfigure}{0.49\textwidth}
		\centering
		\includegraphics[scale=0.11]{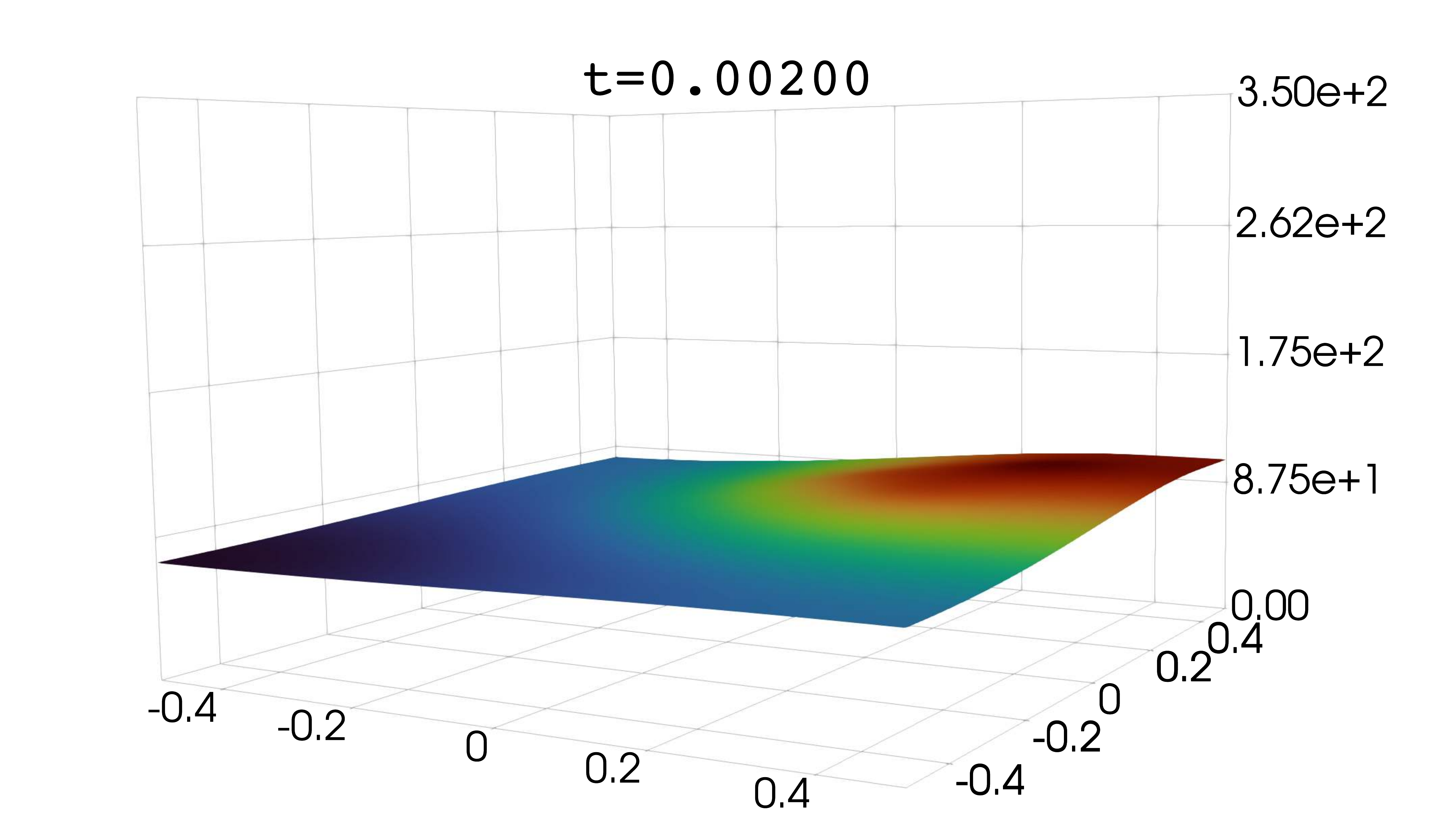}
	\end{subfigure}
	\begin{subfigure}{0.49\textwidth}
		\centering
		\includegraphics[scale=0.11]{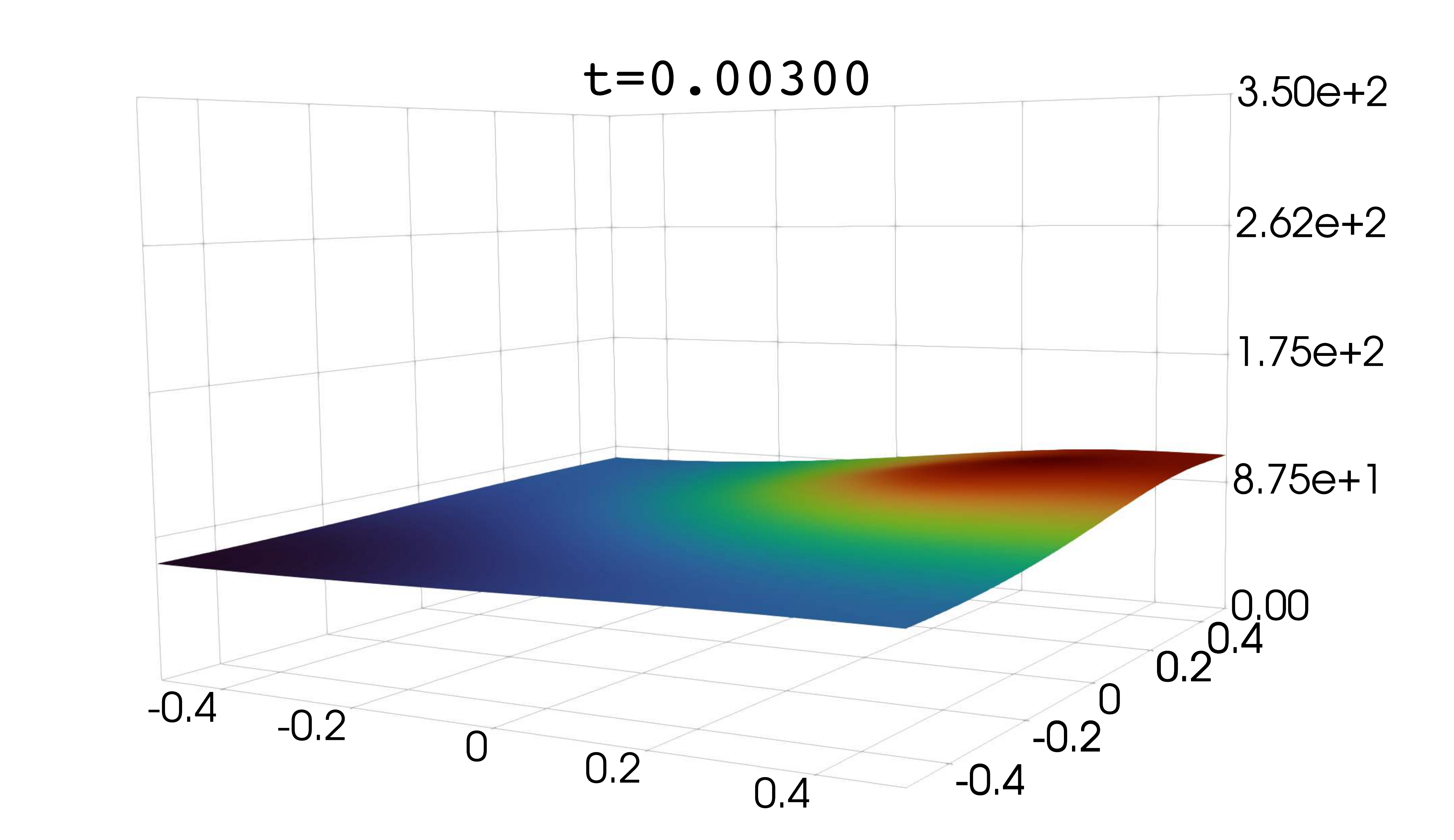}
	\end{subfigure}
	\begin{subfigure}{0.49\textwidth}
		\centering
		\includegraphics[scale=0.11]{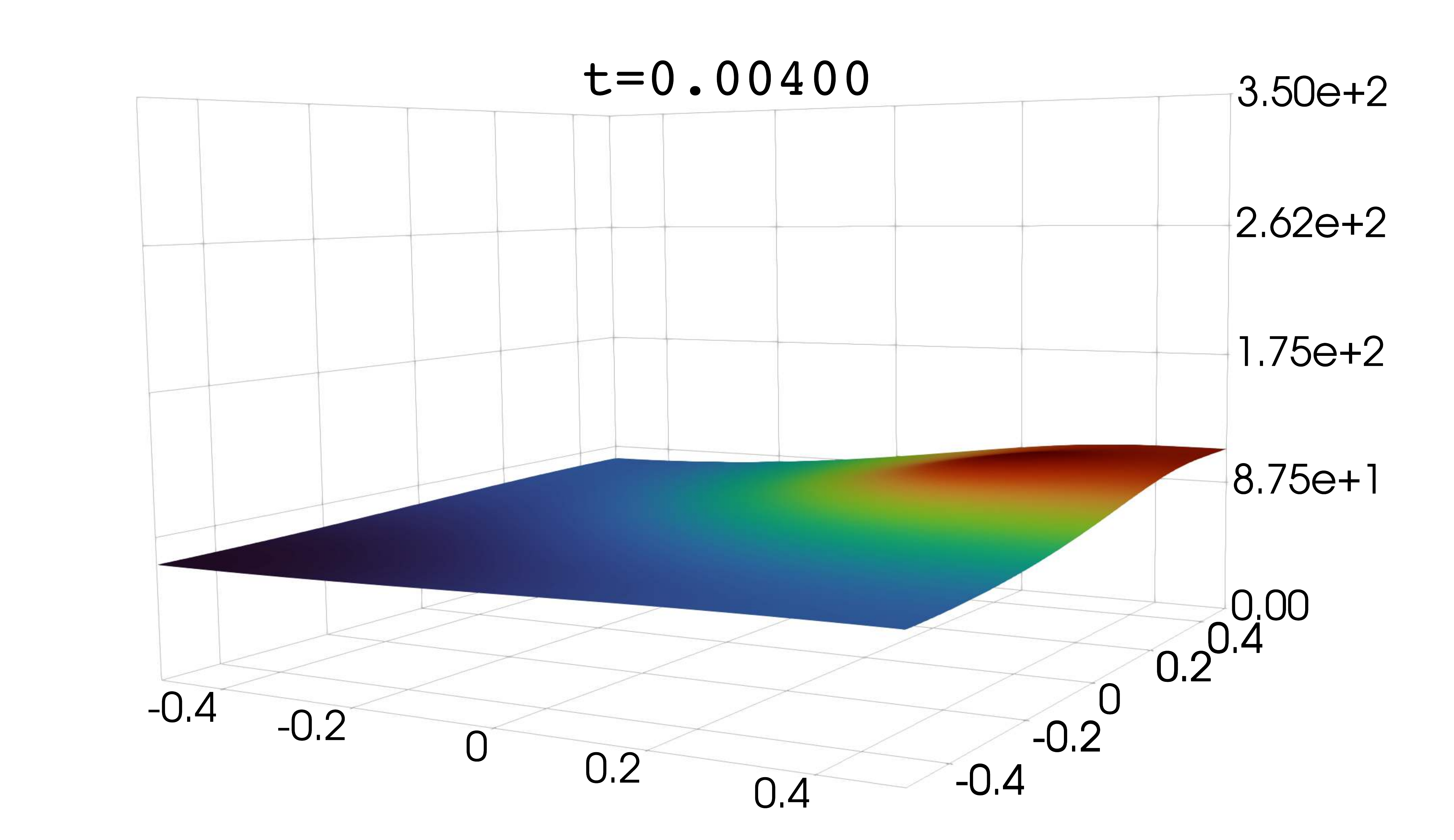}
	\end{subfigure}
	\begin{subfigure}{0.49\textwidth}
		\centering
		\includegraphics[scale=0.11]{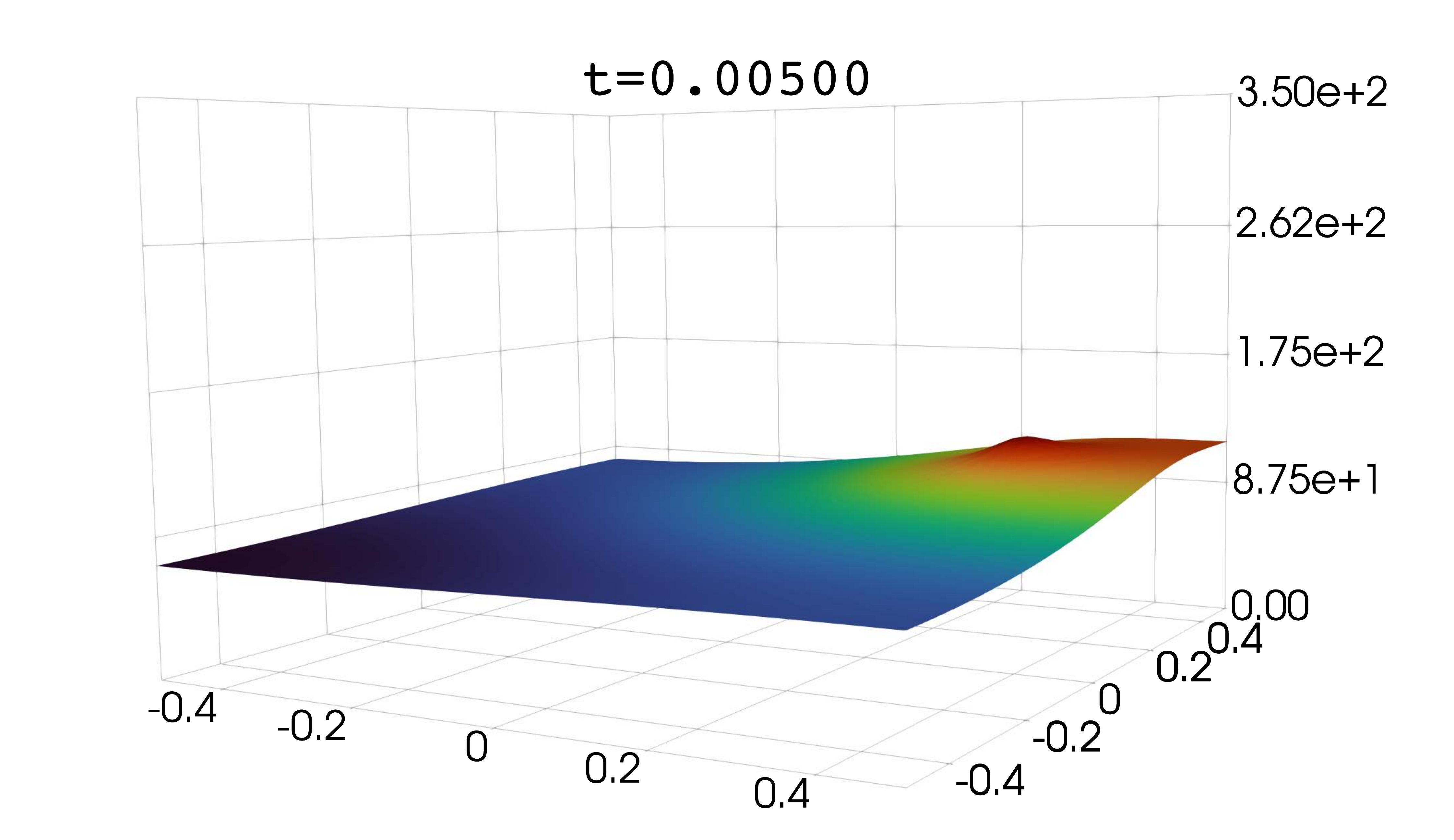}
	\end{subfigure}
	\begin{subfigure}{0.49\textwidth}
		\centering
		\includegraphics[scale=0.11]{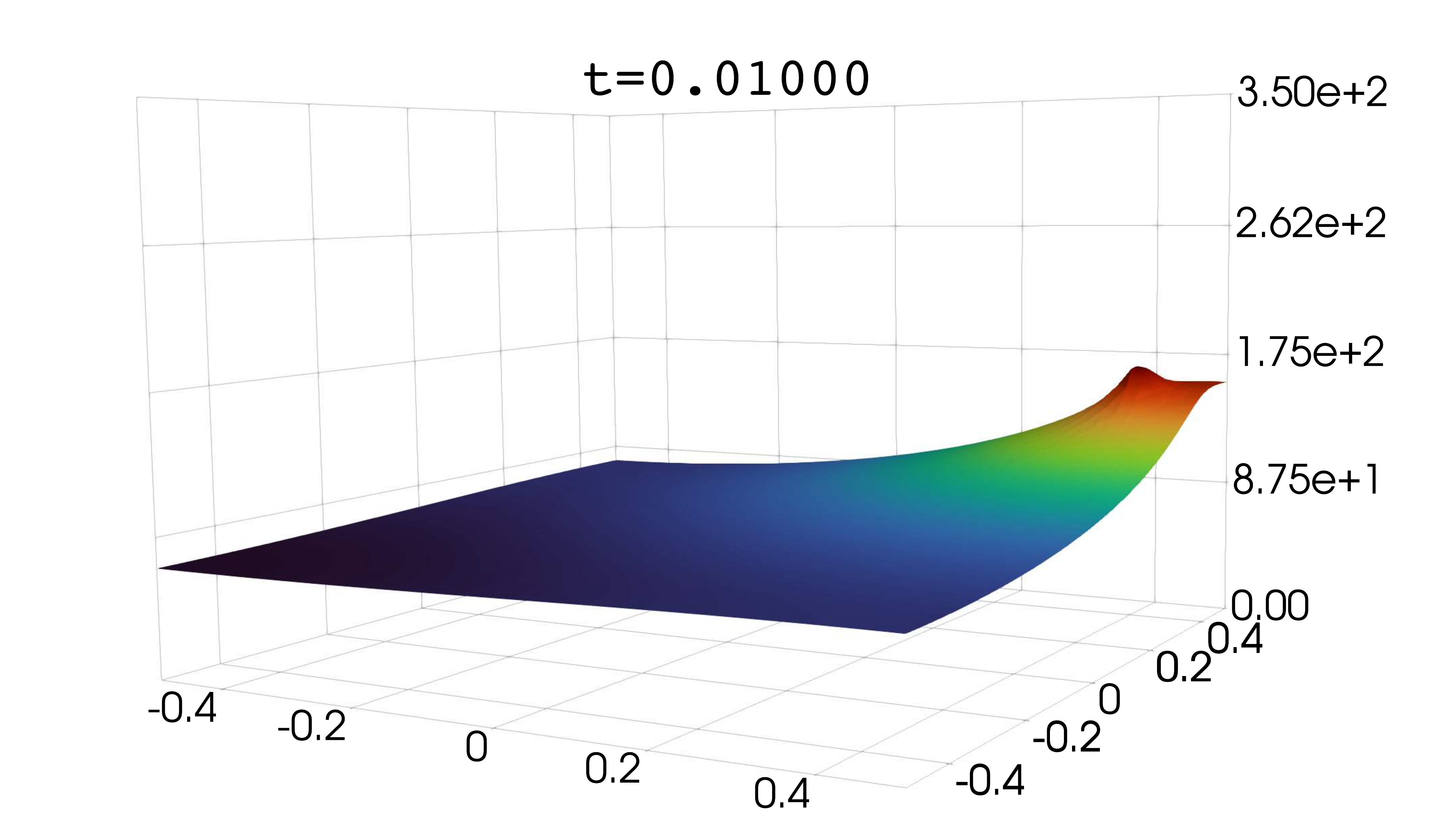}
	\end{subfigure}
	\begin{subfigure}{0.49\textwidth}
		\centering
		\includegraphics[scale=0.11]{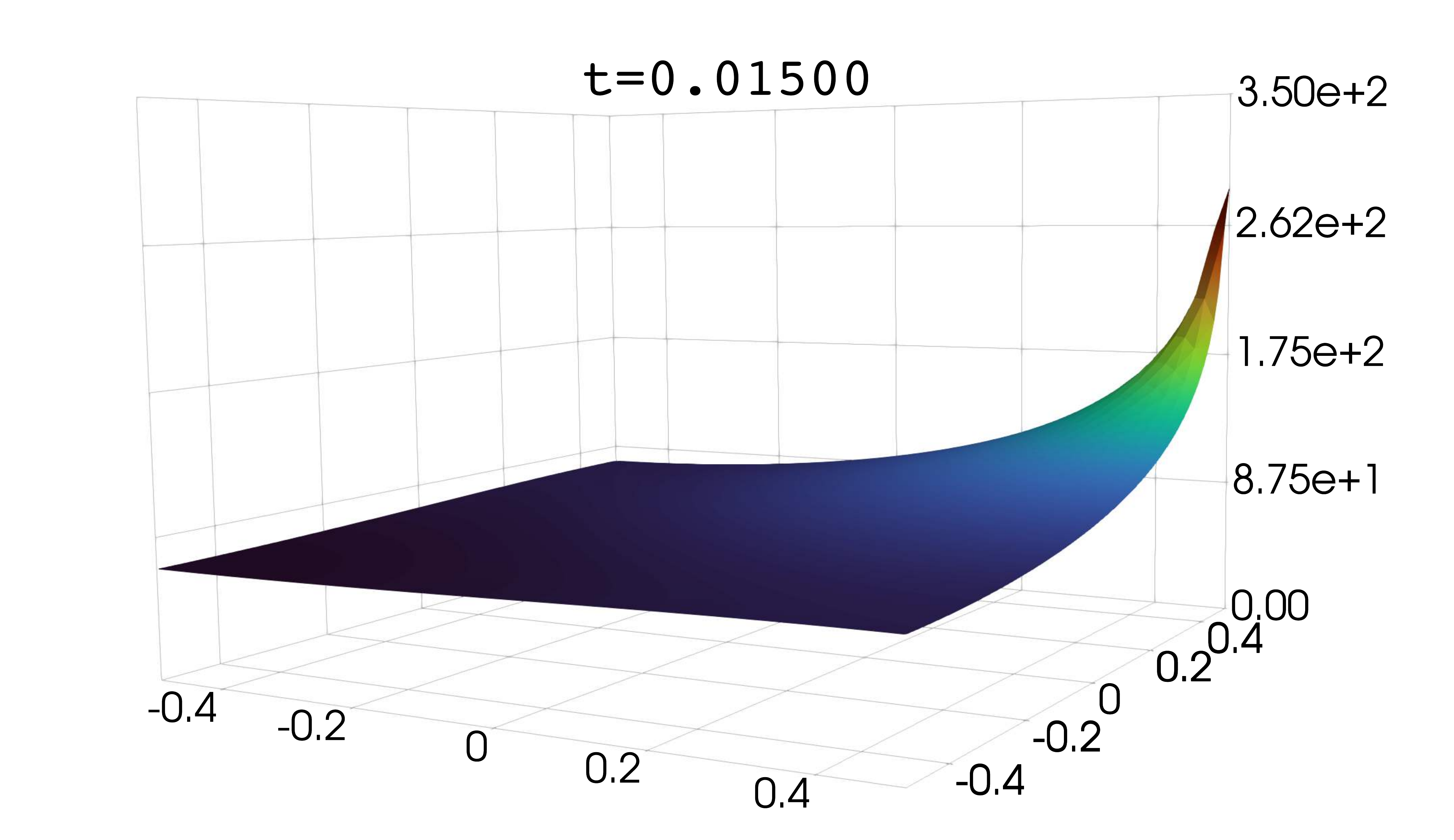}
	\end{subfigure}
	\caption{Chemoattractant in the case of three cell bulges with $h\approx2.83\cdot 10^{-2}$.}
	\label{fig:v_saito_1}
\end{figure}

\begin{figure}
	\centering
	\boldmath{$u$}
	
	\begin{subfigure}{0.49\textwidth}
		\centering
		\includegraphics[scale=0.11]{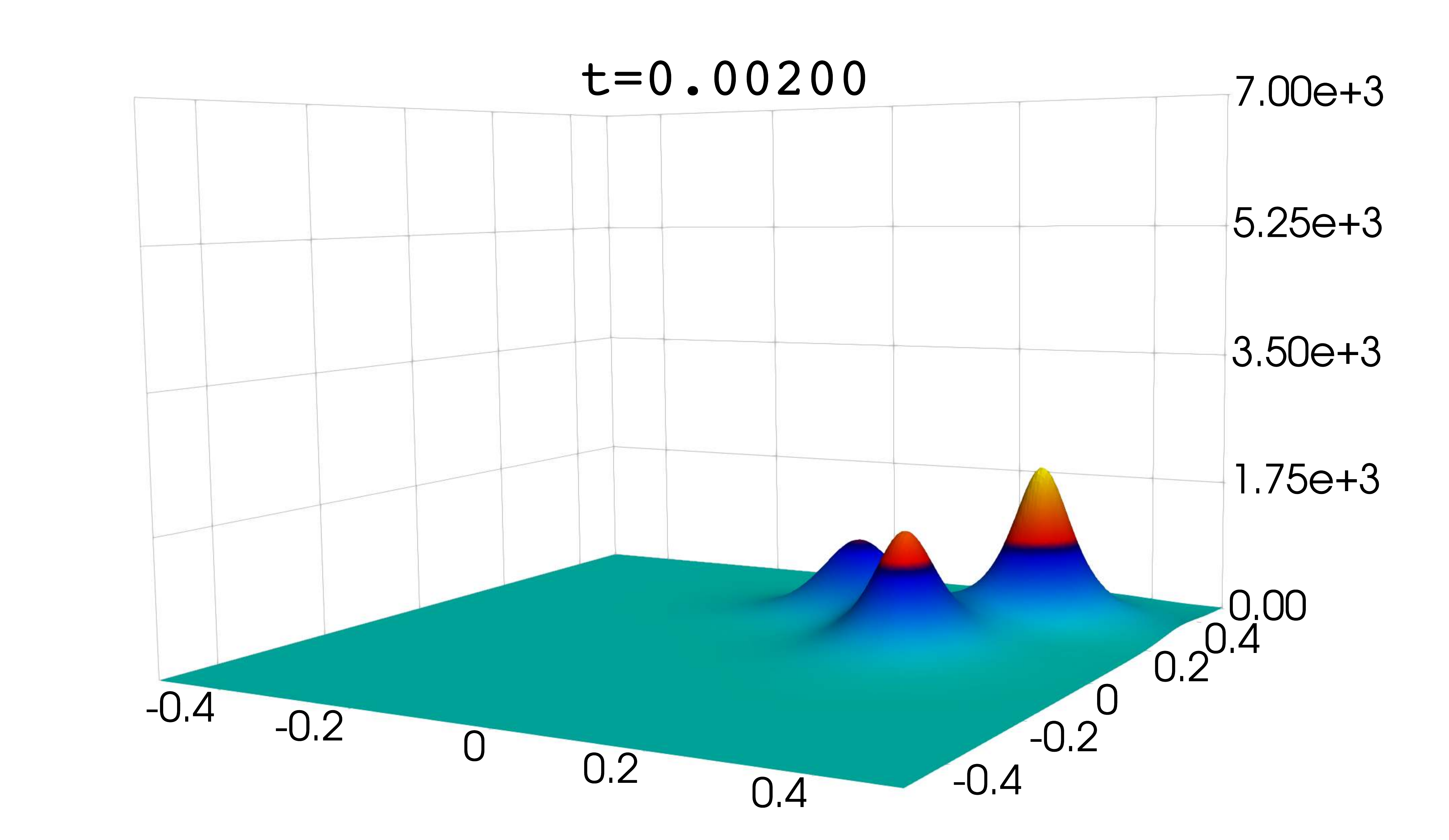}
	\end{subfigure}
	\begin{subfigure}{0.49\textwidth}
		\centering
		\includegraphics[scale=0.11]{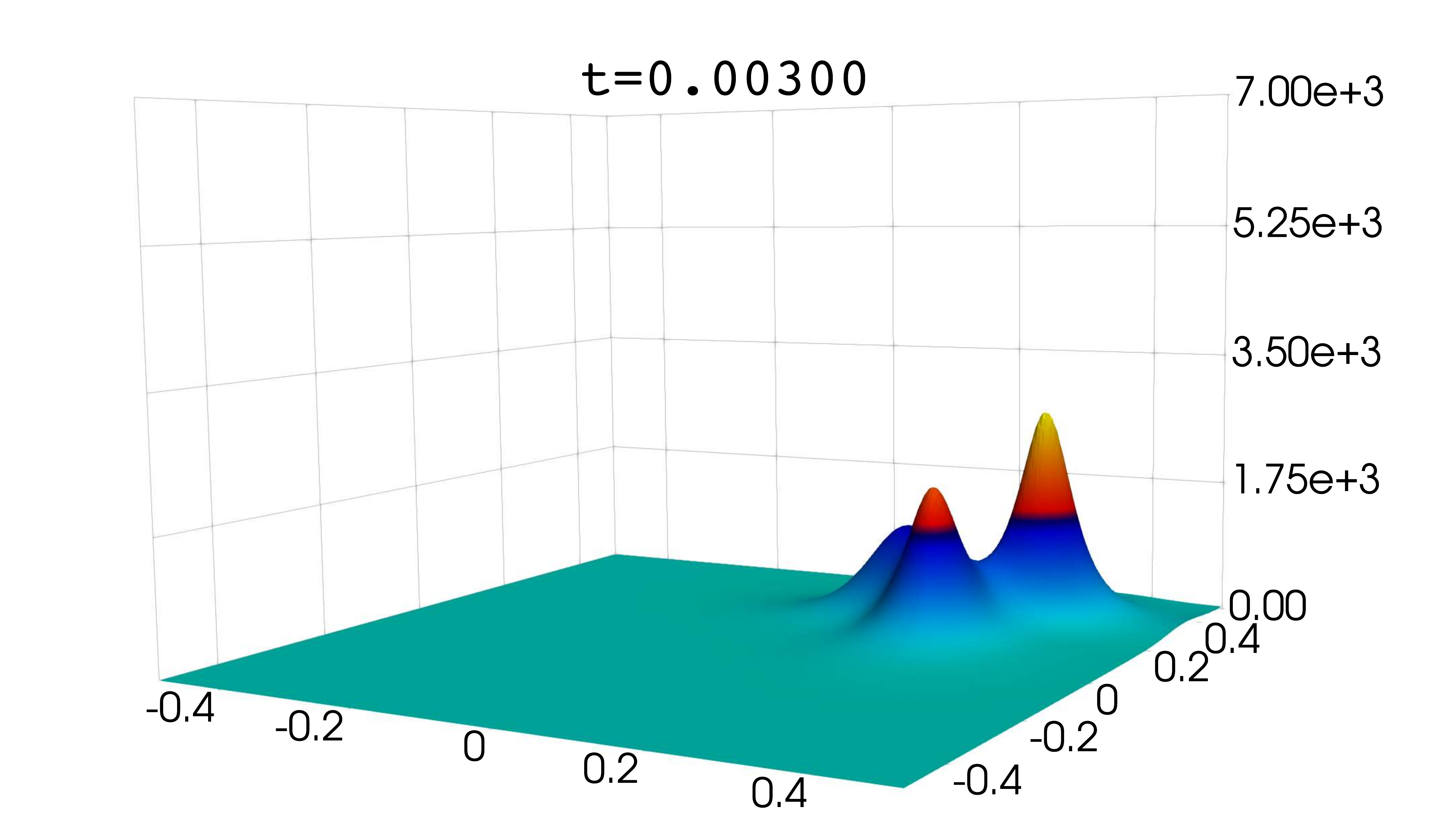}
	\end{subfigure}
	\begin{subfigure}{0.49\textwidth}
		\centering
		\includegraphics[scale=0.11]{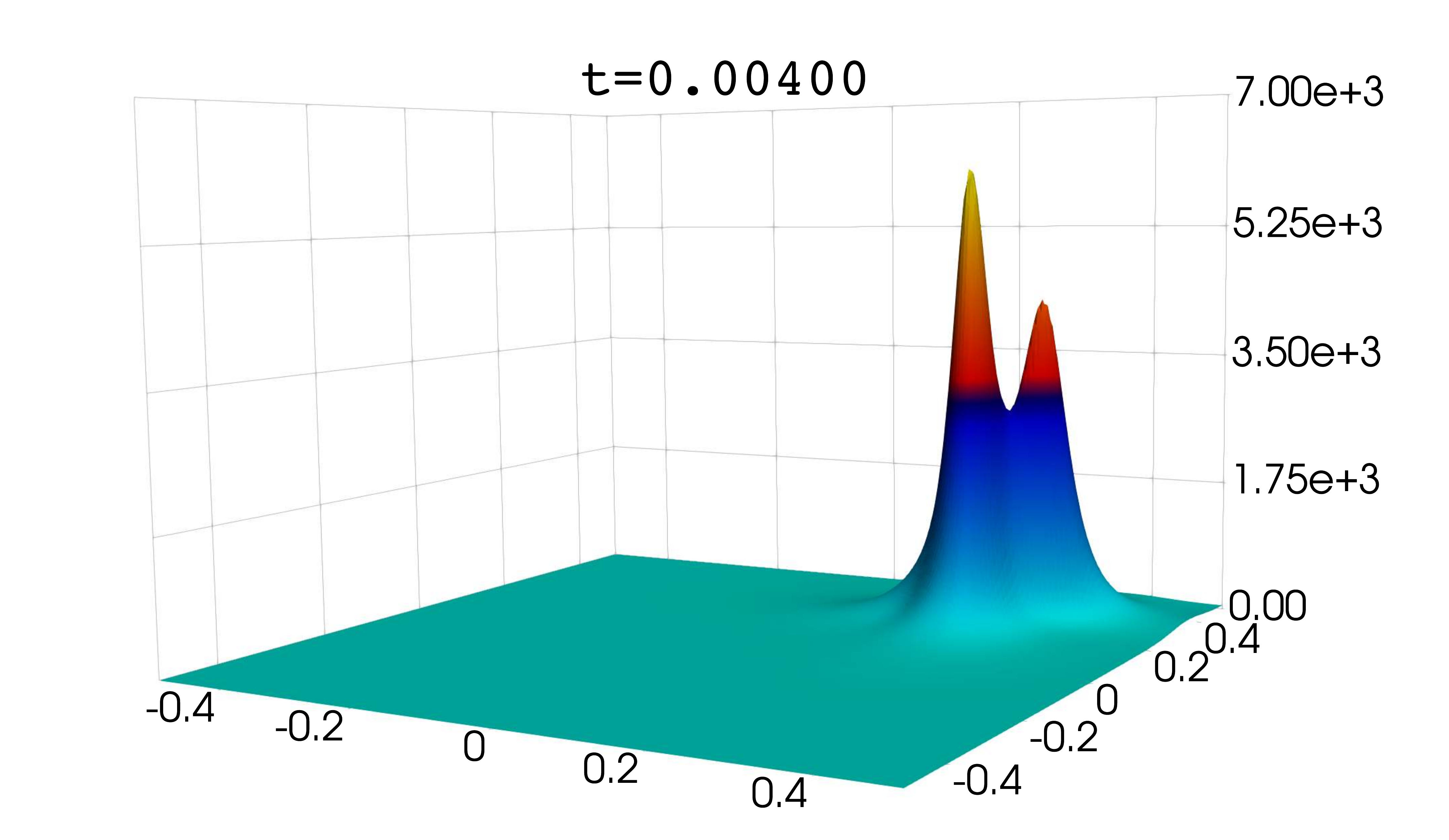}
	\end{subfigure}
	\begin{subfigure}{0.49\textwidth}
		\centering
		\includegraphics[scale=0.11]{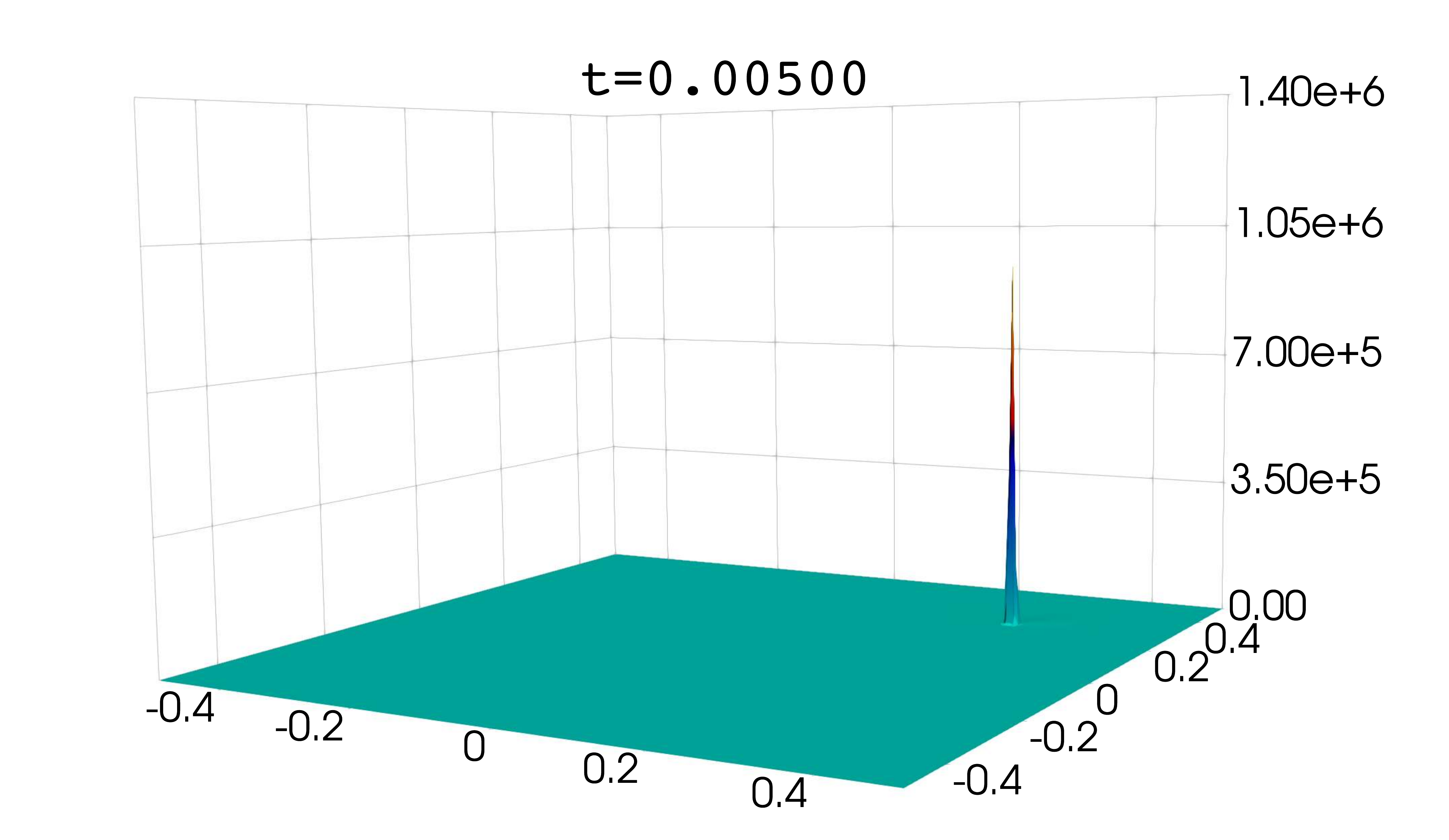}
	\end{subfigure}
	\begin{subfigure}{0.49\textwidth}
		\centering
		\includegraphics[scale=0.11]{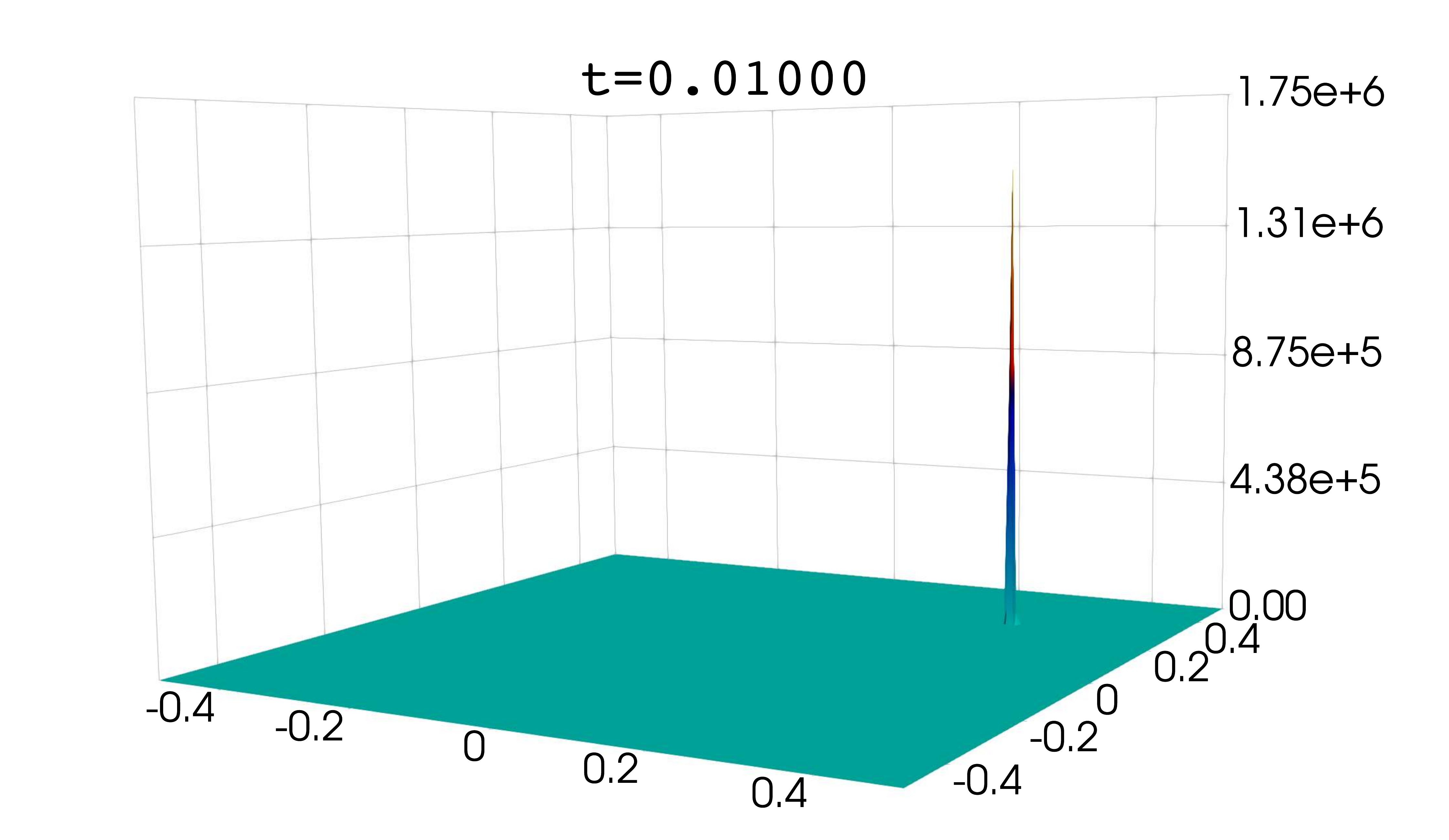}
	\end{subfigure}
	\begin{subfigure}{0.49\textwidth}
		\centering
		\includegraphics[scale=0.11]{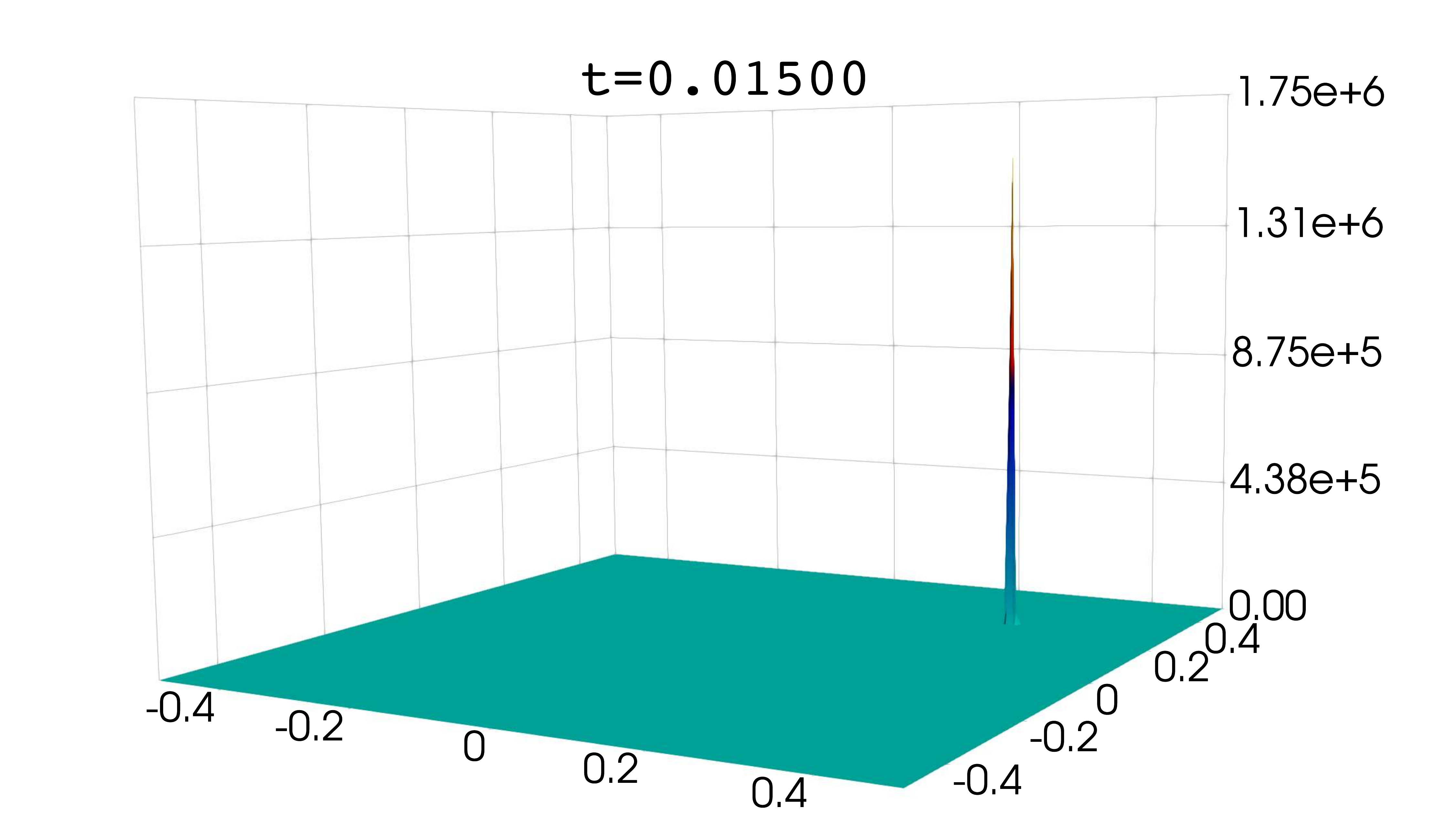}
	\end{subfigure}
	\caption{Aggregation of three cell bulges with $h\approx7.07\cdot 10^{-3}$.}
	\label{fig:u_saito_2}
\end{figure}

\begin{figure}
	\centering
	\boldmath{$v$}
	
	\begin{subfigure}{0.49\textwidth}
		\centering
		\includegraphics[scale=0.11]{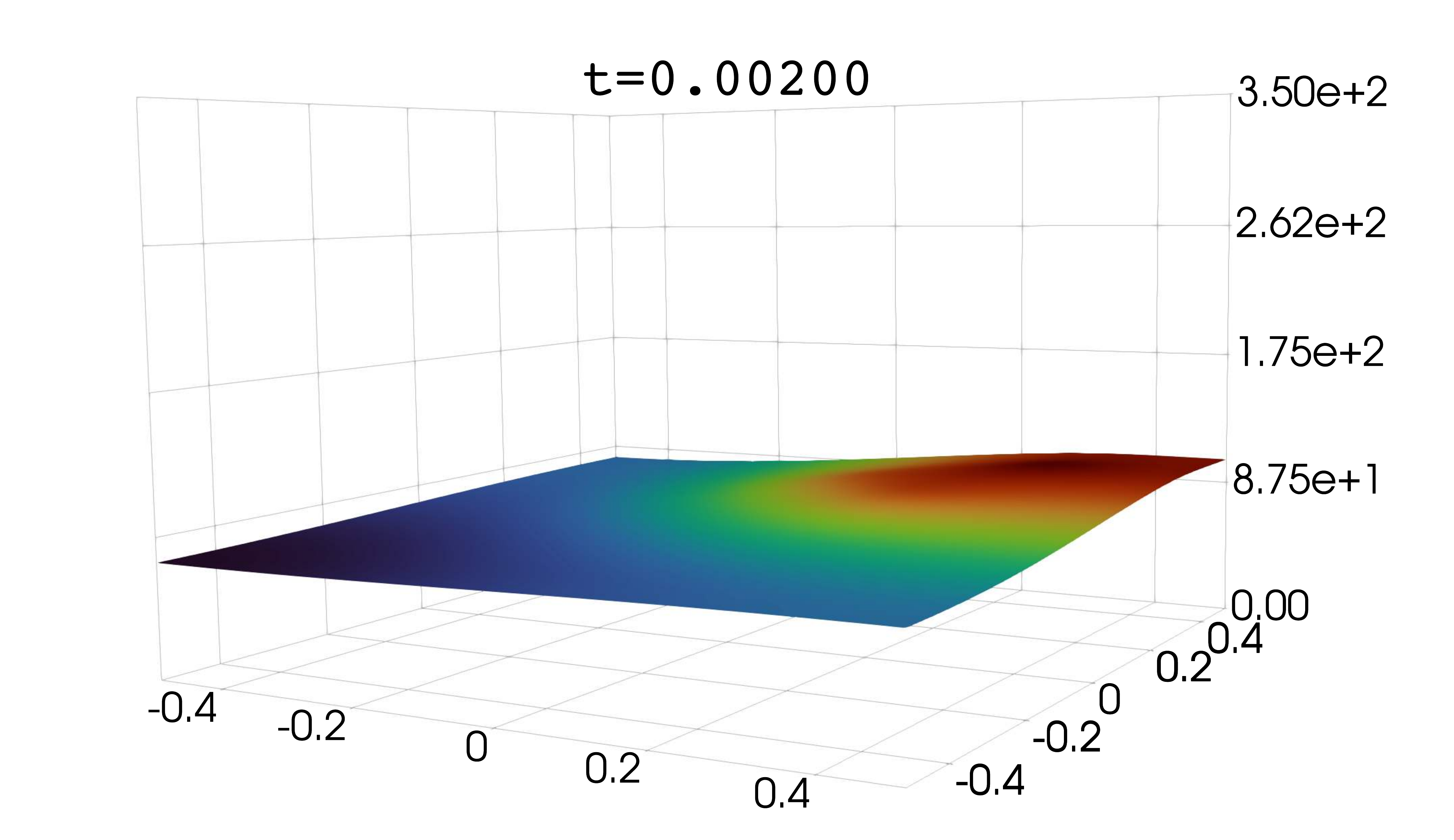}
	\end{subfigure}
	\begin{subfigure}{0.49\textwidth}
		\centering
		\includegraphics[scale=0.11]{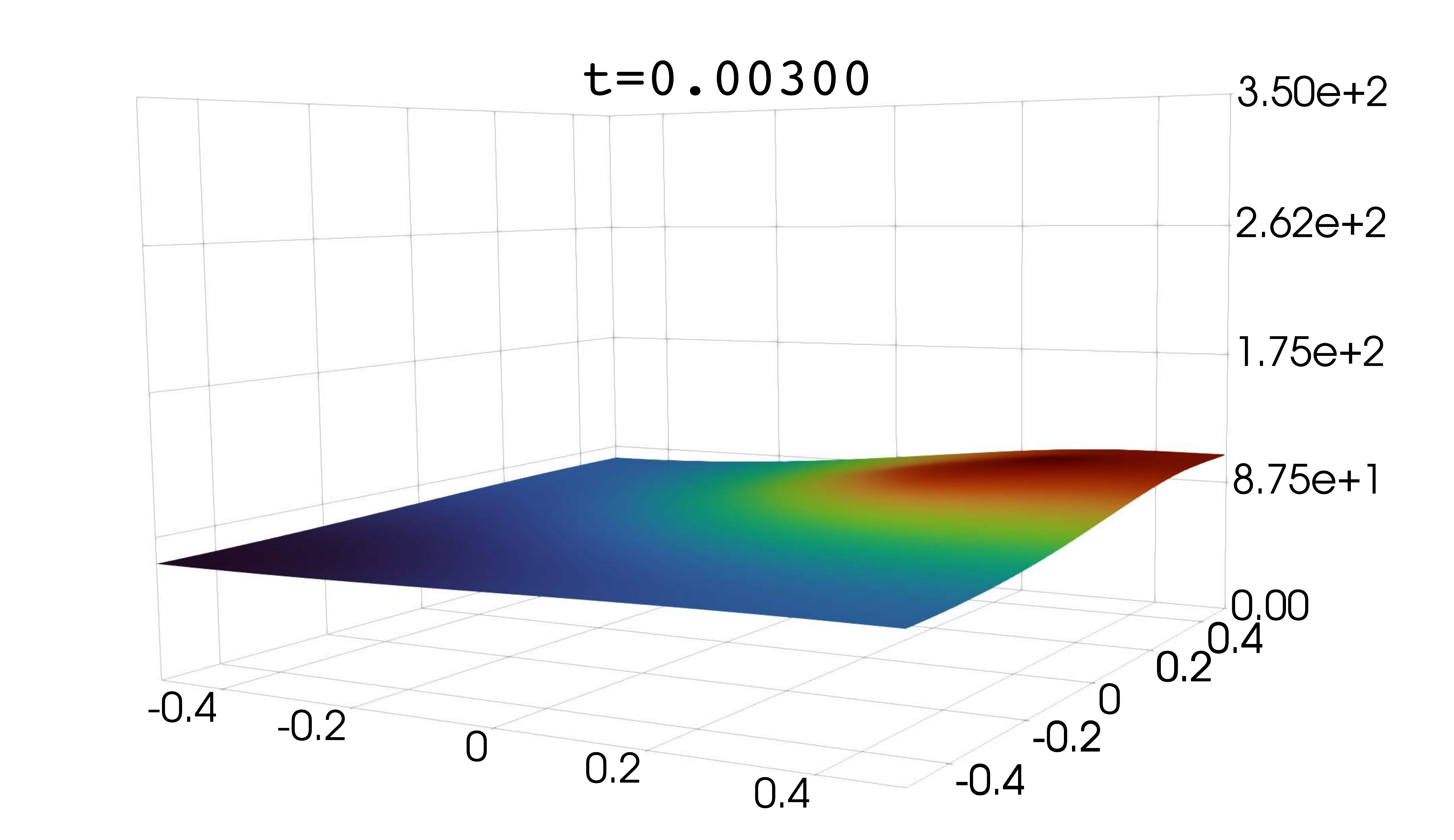}
	\end{subfigure}
	\begin{subfigure}{0.49\textwidth}
		\centering
		\includegraphics[scale=0.11]{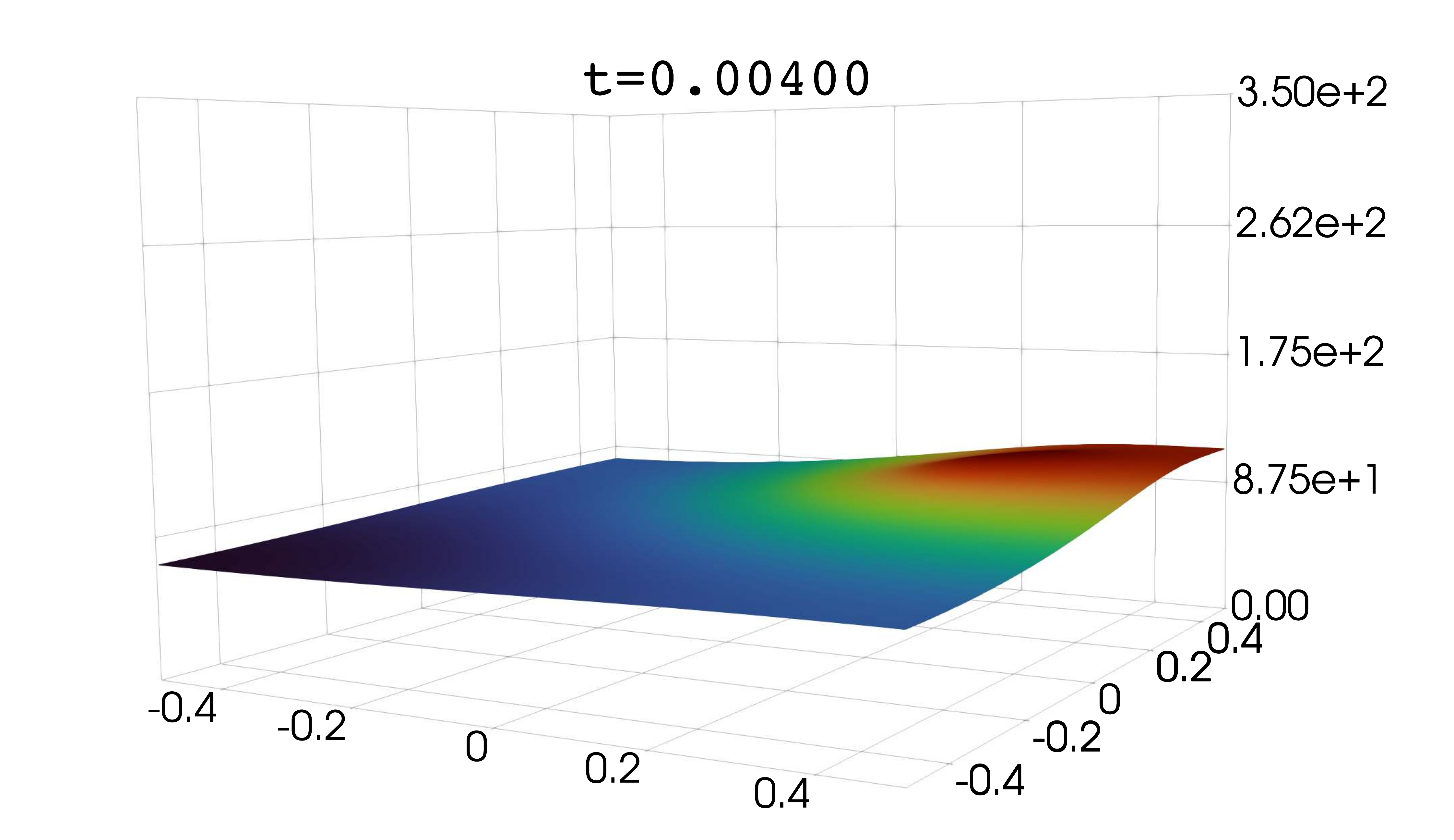}
	\end{subfigure}
	\begin{subfigure}{0.49\textwidth}
		\centering
		\includegraphics[scale=0.11]{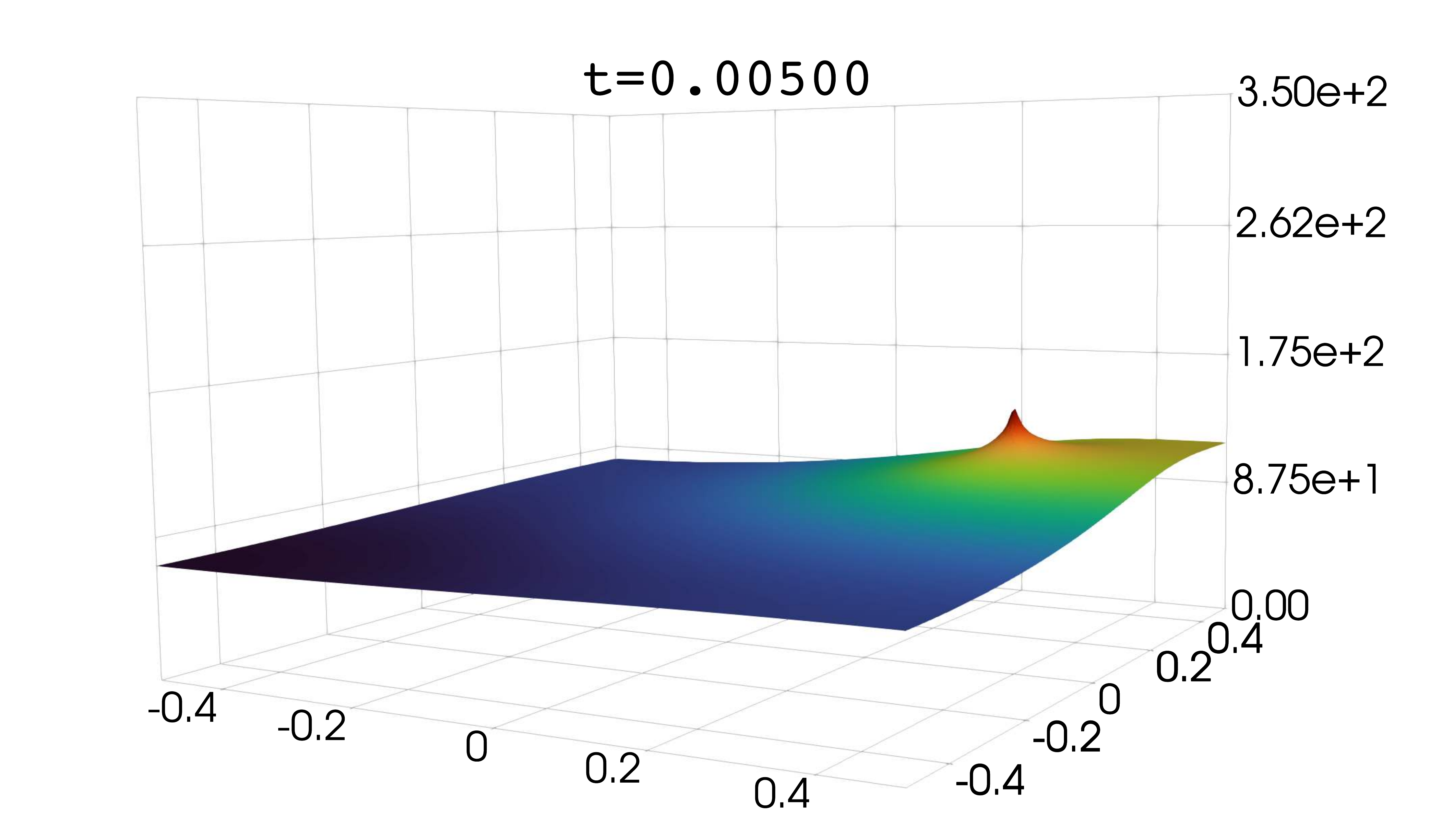}
	\end{subfigure}
	\begin{subfigure}{0.49\textwidth}
		\centering
		\includegraphics[scale=0.11]{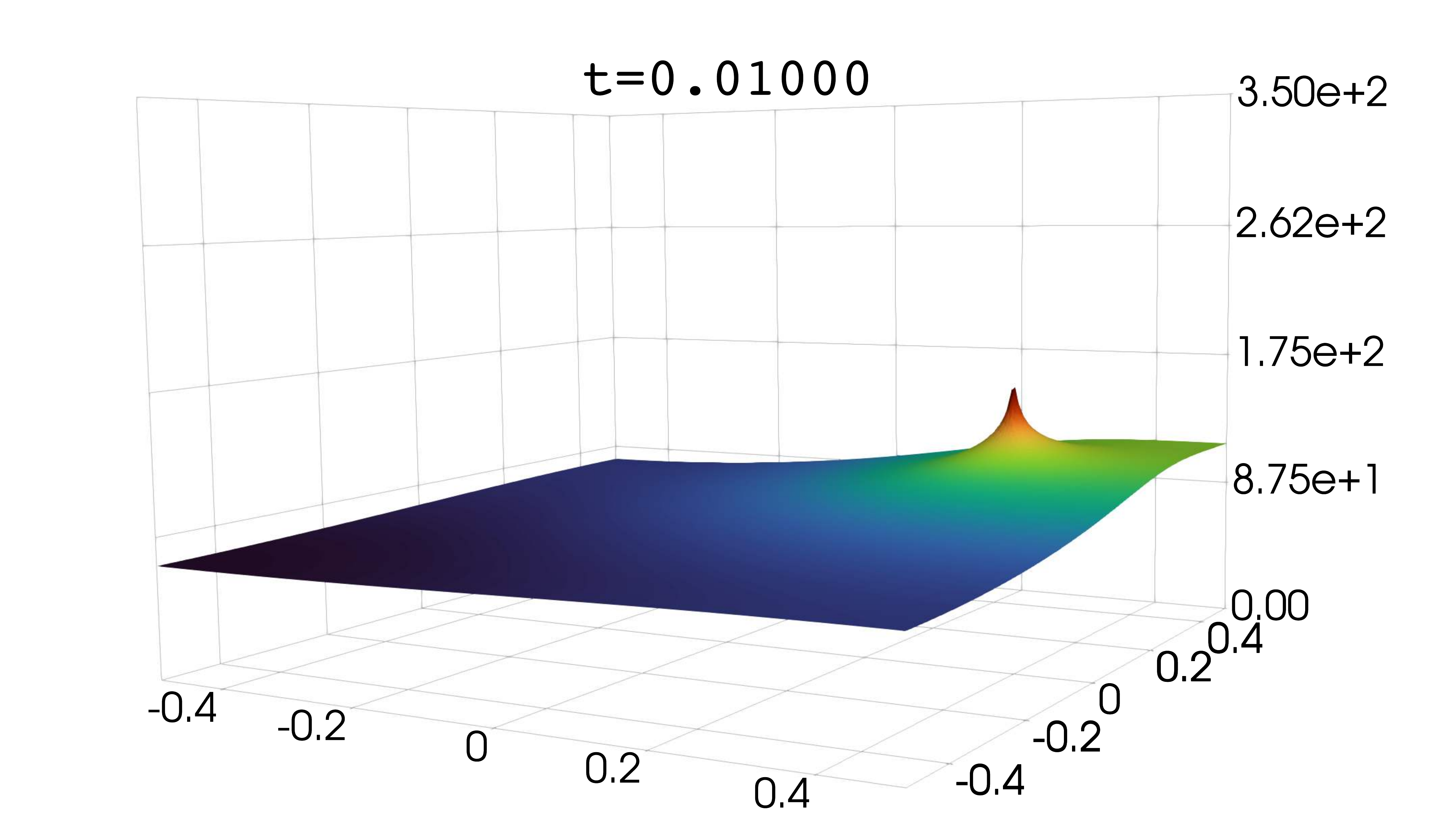}
	\end{subfigure}
	\begin{subfigure}{0.49\textwidth}
		\centering
		\includegraphics[scale=0.11]{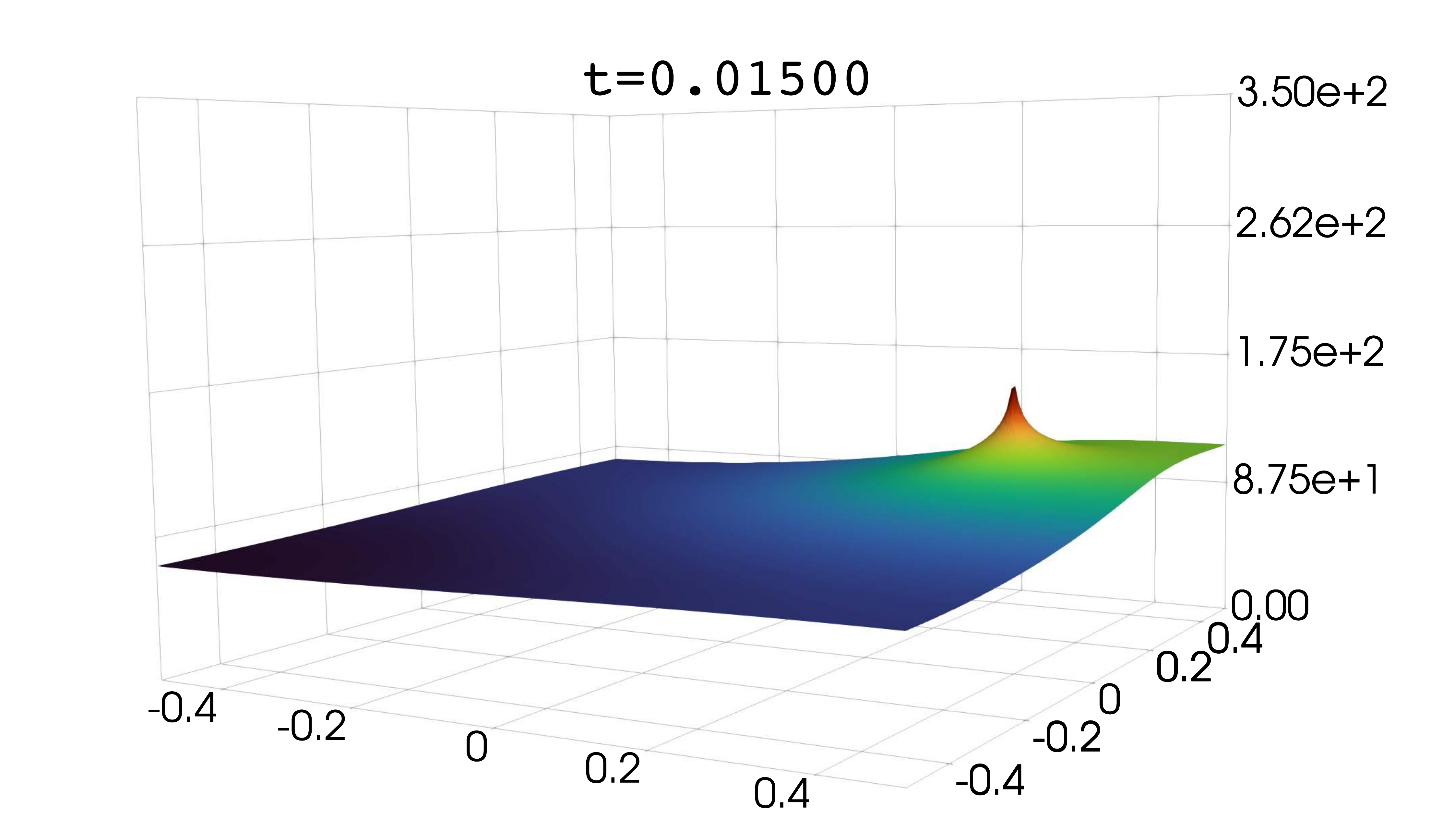}
	\end{subfigure}
	\caption{Chemoattractant in the case of three cell bulges with $h\approx7.07\cdot 10^{-3}$.}
	\label{fig:v_saito_2}
\end{figure}

In both cases, the positivity is preserved and the energy decreases in the discrete case. See Figures~\ref{fig:min-max_saito_1} and~\ref{fig:energy_saito} (left) for the case $h\approx2.83\cdot 10^{-2}$ and Figures~\ref{fig:min-max_saito_2} and~\ref{fig:energy_saito} (right) for the case $h\approx7.07\cdot 10^{-3}$.

\begin{figure}
	\centering
	\begin{subfigure}{0.49\textwidth}
		\centering
		\boldmath{$u$}
		
		\includegraphics[scale=0.5]{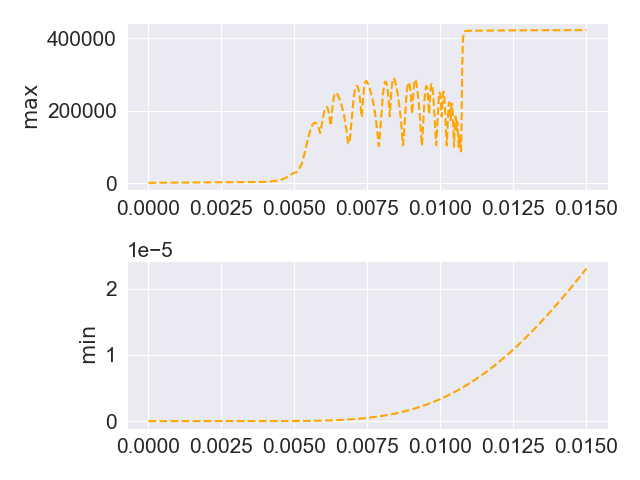}
	\end{subfigure}
	\begin{subfigure}{0.49\textwidth}
		\centering
		\boldmath{$v$}
		
		\includegraphics[scale=0.5]{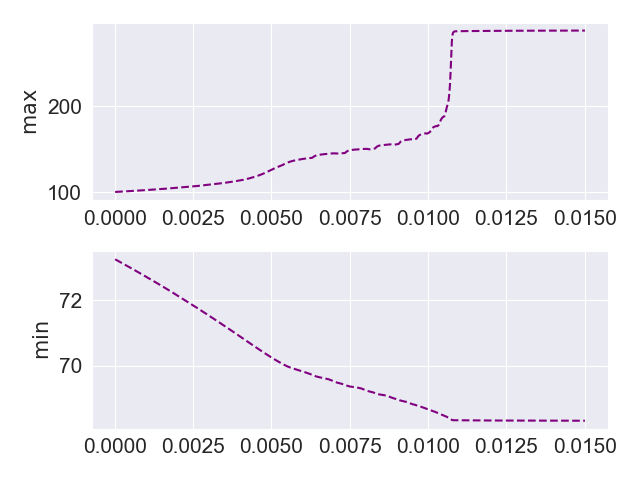}
	\end{subfigure}
	\caption{Minimum and maximum of $u$ and $v$ in the case of aggregation of three cell bulges with $h\approx2.83\cdot 10^{-2}$.}
	\label{fig:min-max_saito_1}
\end{figure}

\begin{figure}
	\centering
	\begin{subfigure}{0.49\textwidth}
		\centering
		\boldmath{$u$}
		
		\includegraphics[scale=0.5]{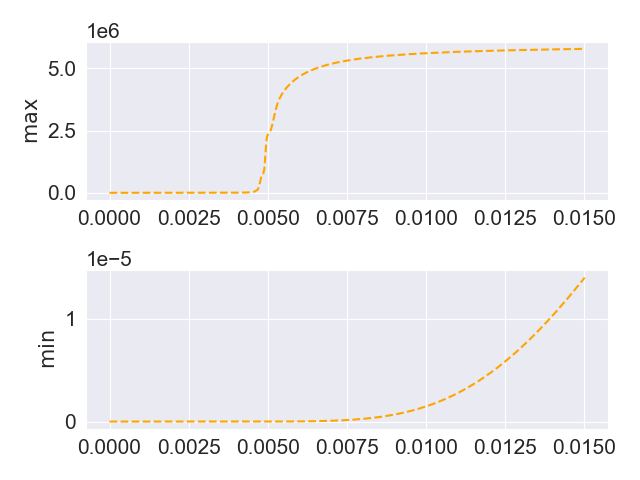}
	\end{subfigure}
	\begin{subfigure}{0.49\textwidth}
		\centering
		\boldmath{$v$}
		
		\includegraphics[scale=0.5]{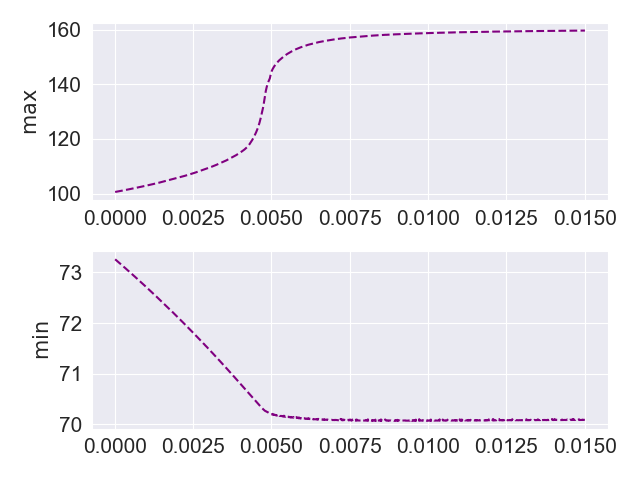}
	\end{subfigure}
	\caption{Minimum and maximum of $u$ and $v$ in the case of aggregation of three cell bulges with $h\approx7.07\cdot 10^{-3}$.}
	\label{fig:min-max_saito_2}
\end{figure}

\begin{figure}
	\centering
	\begin{subfigure}{0.49\textwidth}
		\centering
		\includegraphics[scale=0.5]{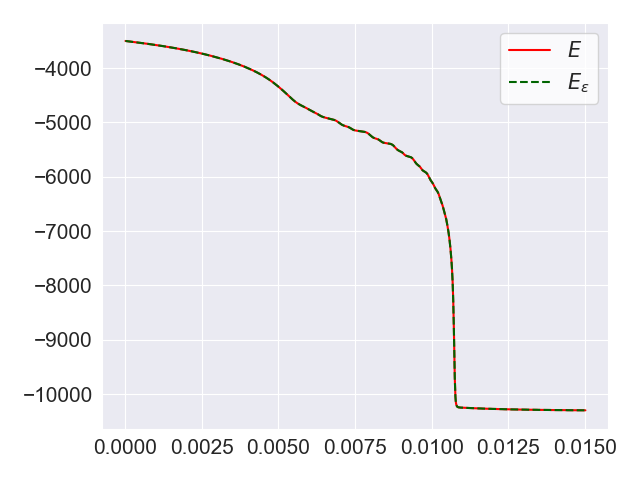}
	\end{subfigure}
	\begin{subfigure}{0.49\textwidth}
		\centering
		\includegraphics[scale=0.5]{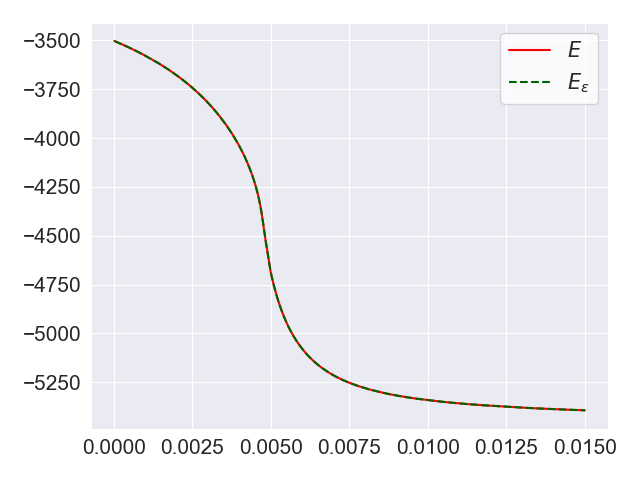}
	\end{subfigure}
	\caption{Discrete energy over time in the case of aggregation of three cell bulges. On the left, $h\approx2.83\cdot 10^{-2}$. On the right, $h\approx7.07\cdot 10^{-3}$.}
	\label{fig:energy_saito}
\end{figure}

\subsection{Pattern formation with multiple peaks}
\label{test3}
Finally, we show the results for a test in which
we obtain a numerical solution describing a pattern with multiple peaks as it occurs, for instance, in the cases that appear in \cite{andreianov_finite_2011, chamoun_monotone_2014, tyson_fractional_2000} for different variations of chemotaxis equations and in \cite{fatkullin_study_2013} for the Keller-Segel equations. For this purpose, we consider the initial conditions
\begin{equation*}
u_0 = 1000 (\cos(2\pi x)\cos(2\pi y) + 1), \quad
v_0 = 500 (\sin(3\pi x)\sin(3\pi y) + 1),
\end{equation*}
which are plotted in Figure~\ref{fig:ic_sin-cos}.

\begin{figure}
	\centering
	\begin{subfigure}{0.49\textwidth}
		\centering
		\boldmath{$u_0$}
		
		\includegraphics[scale=0.11]{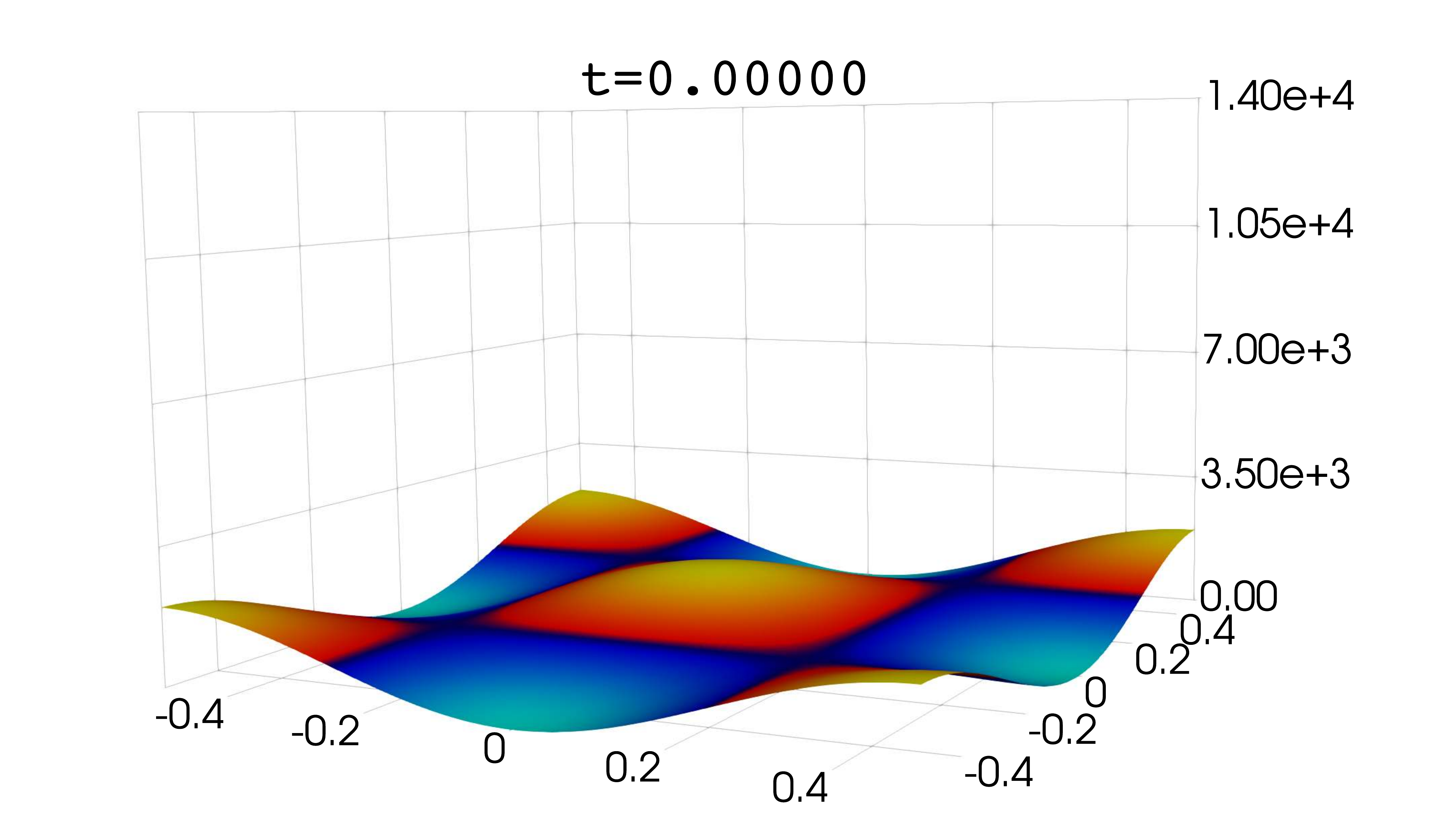}
	\end{subfigure}
	\begin{subfigure}{0.49\textwidth}
		\centering
		\boldmath{$v_0$}
		
		\includegraphics[scale=0.11]{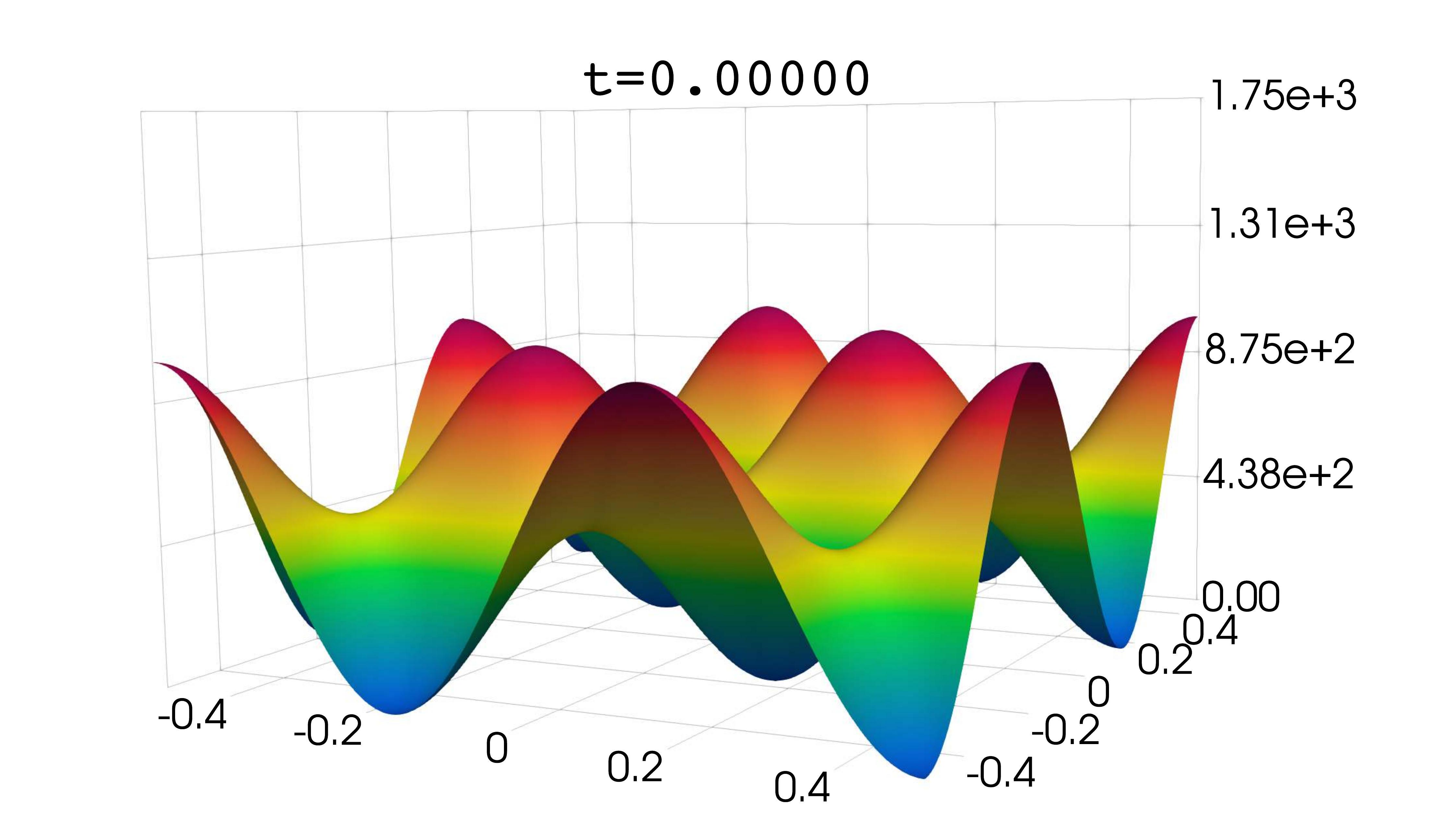}
	\end{subfigure}
	\caption{Initial conditions for the pattern formation with multiple peaks (different scales are used for $u$ and $v$).}
	\label{fig:ic_sin-cos}
\end{figure}

In Figures~\ref{fig:u_sin-cos} and \ref{fig:v_sin-cos} we can observe the result of the test with $h\approx3.54\cdot 10^{-3}$ and $\Delta t=10^{-7}$. Notice that we obtain 8 peaks of cells that reach very high values (up to values of order $10^{7}$), which may be due to a blow-up phenomenon occurring at a certain finite time $t^*$ close to $10^{-4}$.

\begin{figure}
	\centering
	\boldmath{$u$}
	
	\begin{subfigure}{0.49\textwidth}
		\centering
		\includegraphics[scale=0.11]{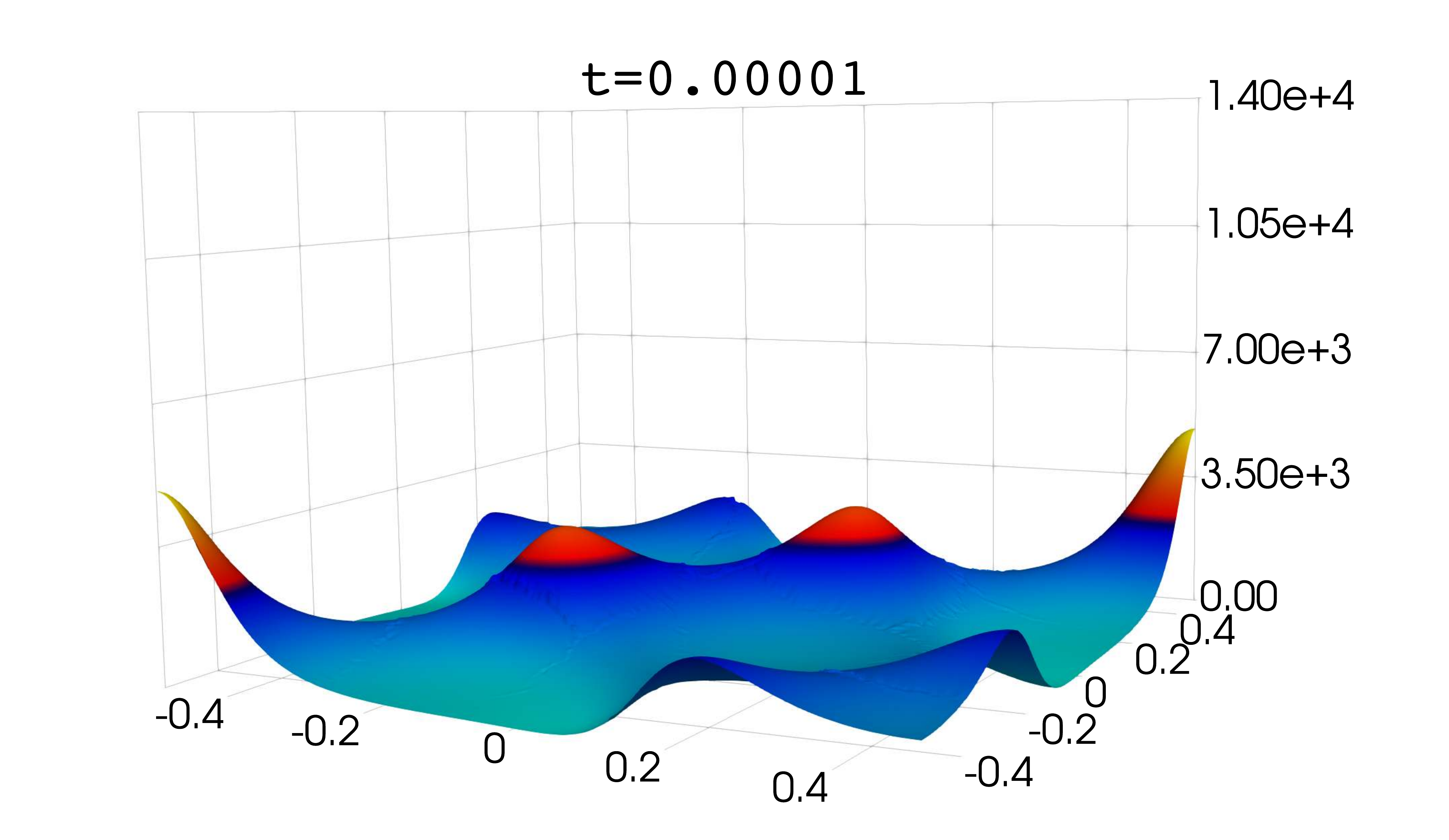}
	\end{subfigure}
	\begin{subfigure}{0.49\textwidth}
		\centering
		\includegraphics[scale=0.11]{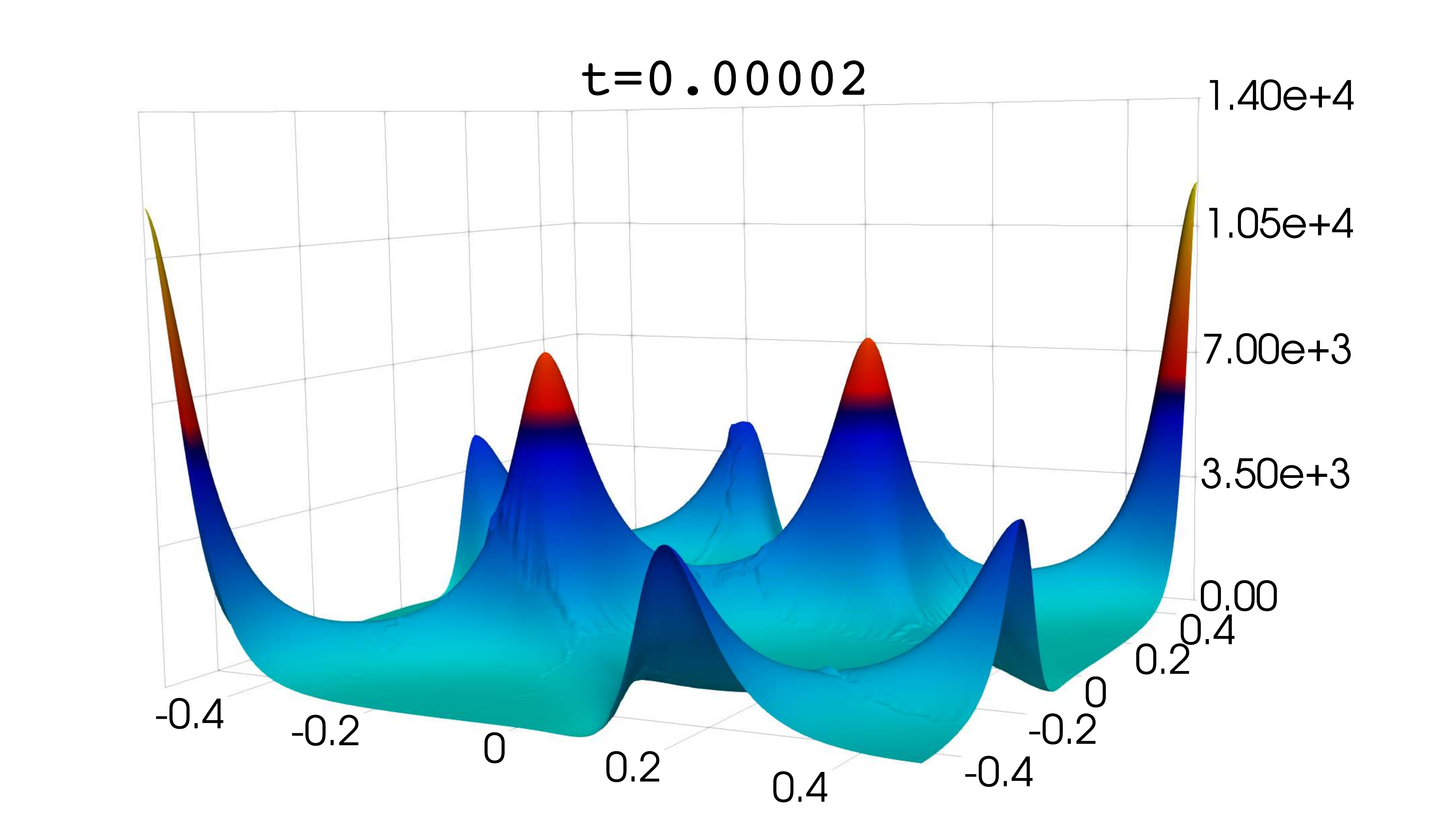}
	\end{subfigure}
	\begin{subfigure}{0.49\textwidth}
		\centering
		\includegraphics[scale=0.11]{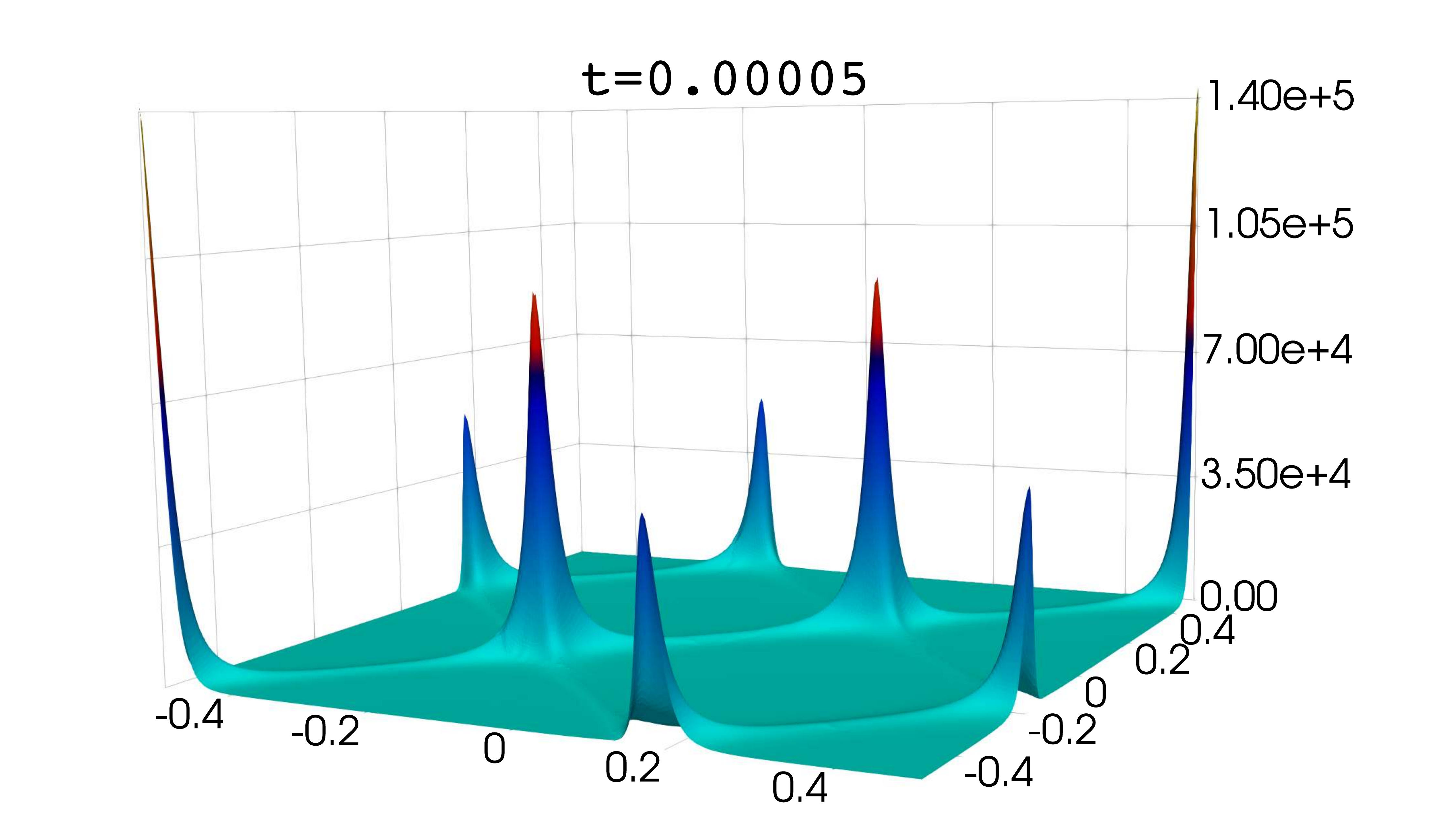}
	\end{subfigure}
	\begin{subfigure}{0.49\textwidth}
		\centering
		\includegraphics[scale=0.11]{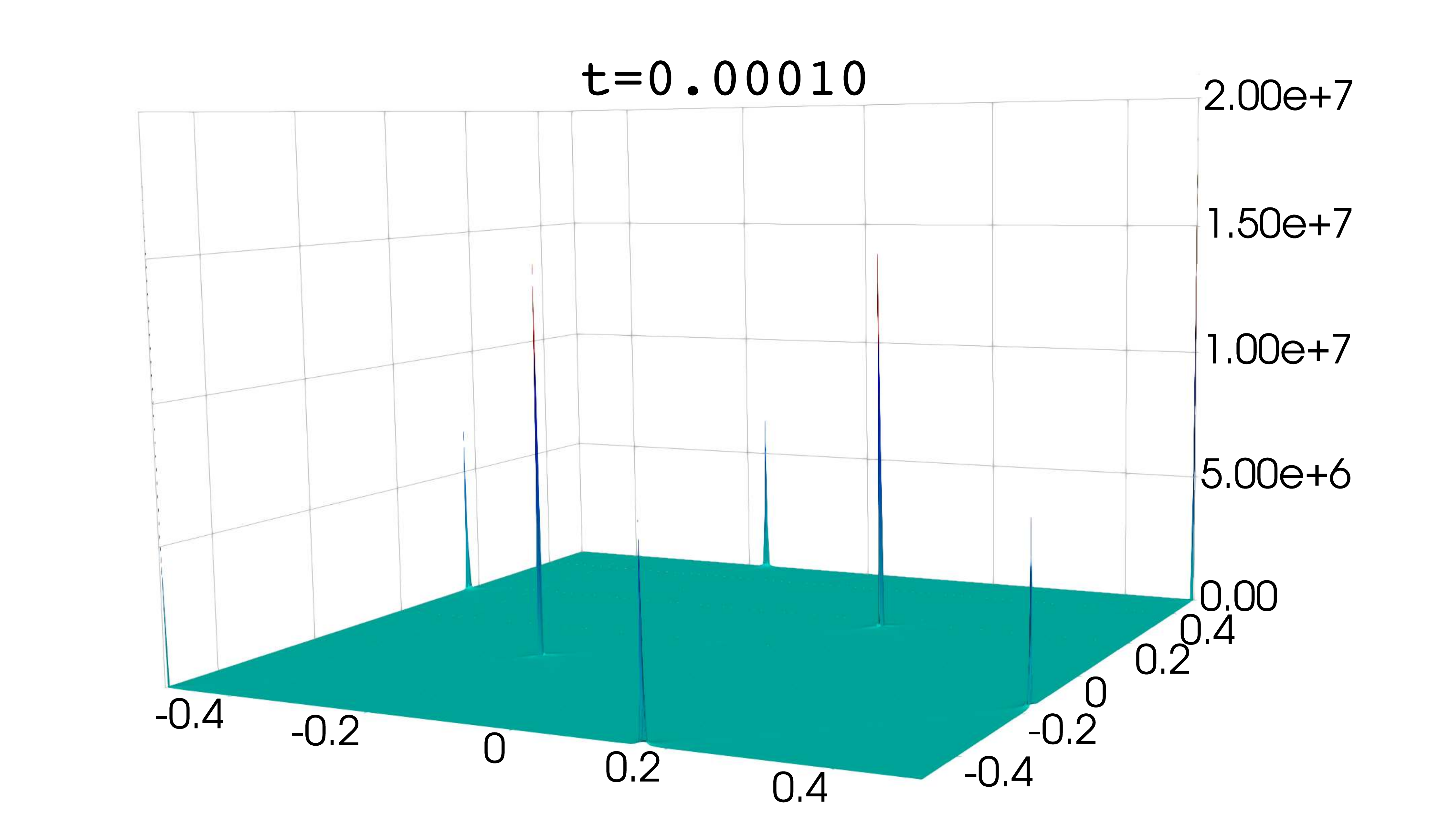}
	\end{subfigure}
	\caption{Pattern formation of $u$ with multiple peaks.}
	\label{fig:u_sin-cos}
\end{figure}

\begin{figure}
	\centering
	\boldmath{$v$}
	
	\begin{subfigure}{0.49\textwidth}
		\centering
		\includegraphics[scale=0.11]{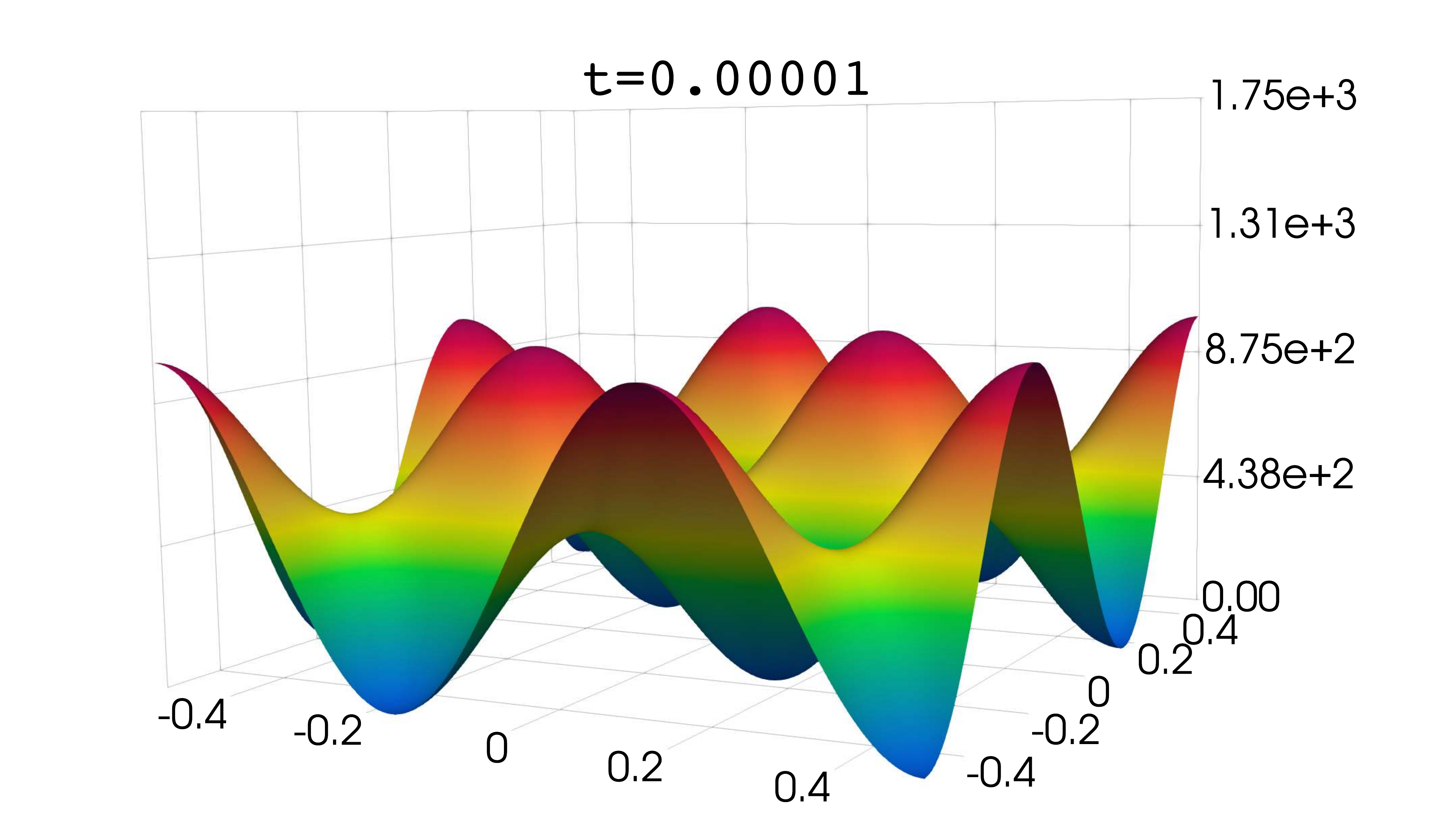}
	\end{subfigure}
	\begin{subfigure}{0.49\textwidth}
		\centering
		\includegraphics[scale=0.11]{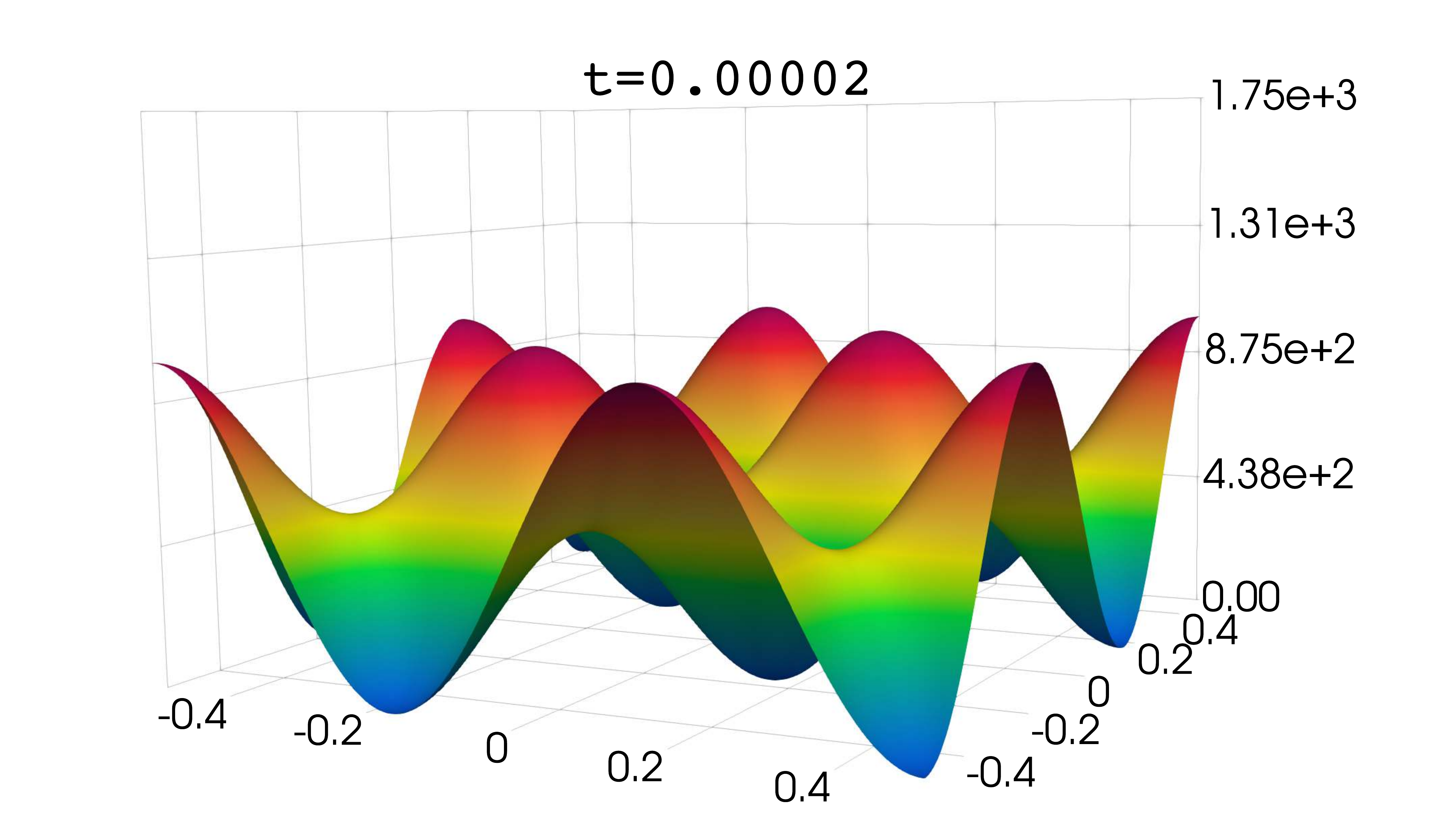}
	\end{subfigure}
	\begin{subfigure}{0.49\textwidth}
		\centering
		\includegraphics[scale=0.11]{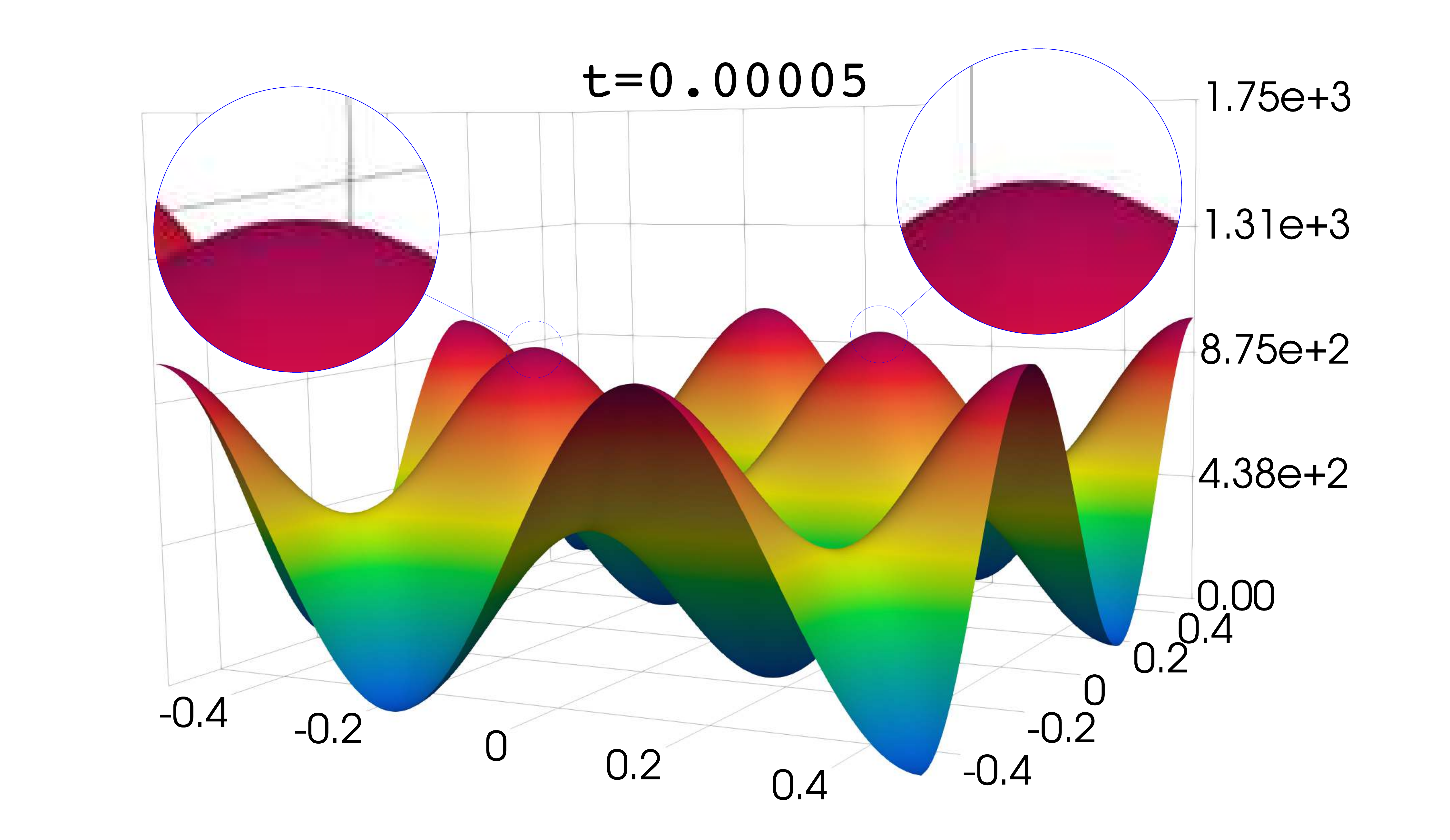}
	\end{subfigure}
	\begin{subfigure}{0.49\textwidth}
		\centering
		\includegraphics[scale=0.11]{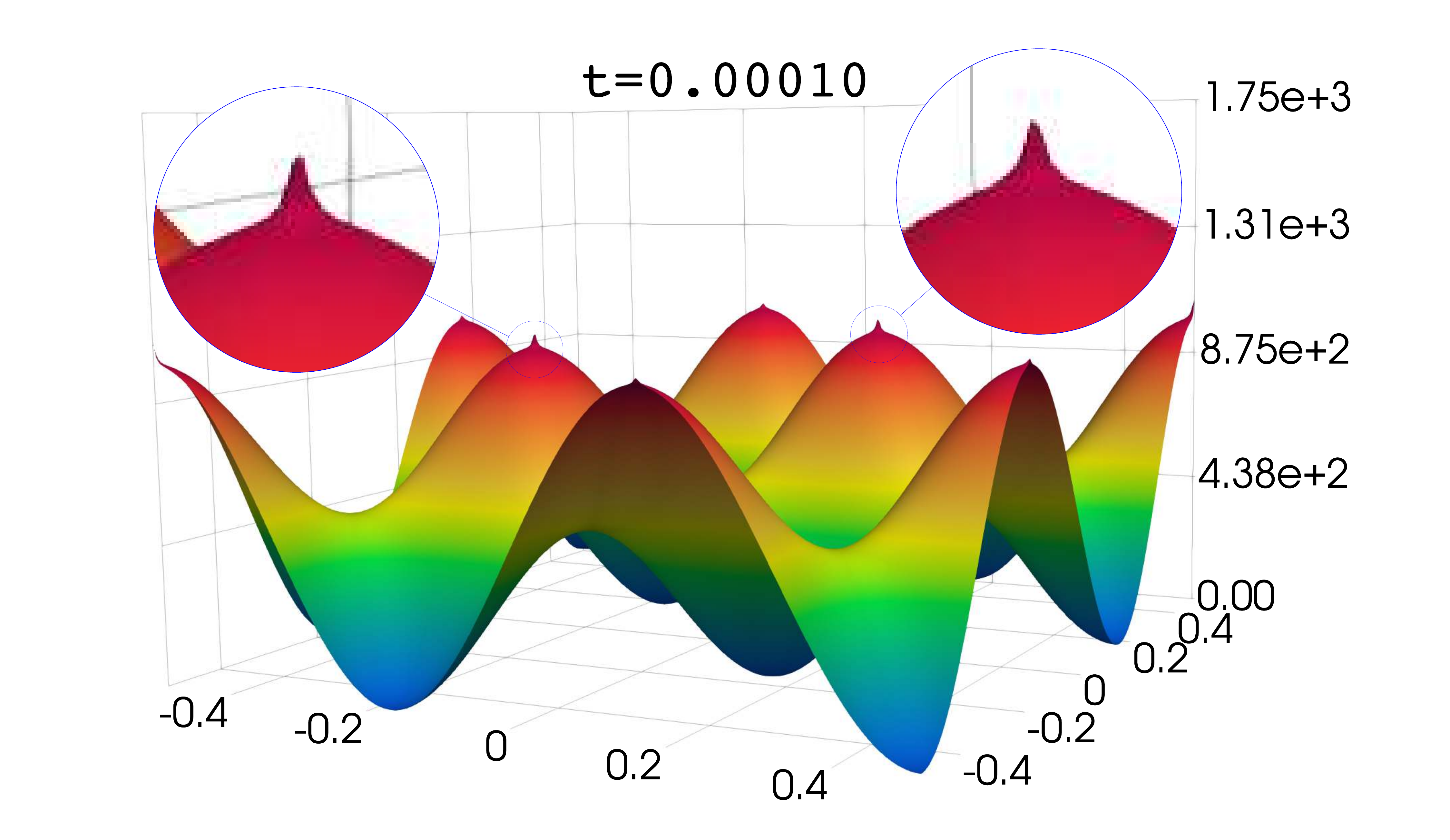}
	\end{subfigure}
	\caption{Pattern formation of $v$ with multiple peaks.}
	\label{fig:v_sin-cos}
\end{figure}

Again, the positivity is preserved as shown in Figure~\ref{fig:min-max_sin-cos} and the energy decreases in the discrete case as in Figure~\ref{fig:energy_sin-cos}.

\begin{figure}
	\centering
	\begin{subfigure}{0.49\textwidth}
		\centering
		\boldmath{$u$}
		
		\includegraphics[scale=0.5]{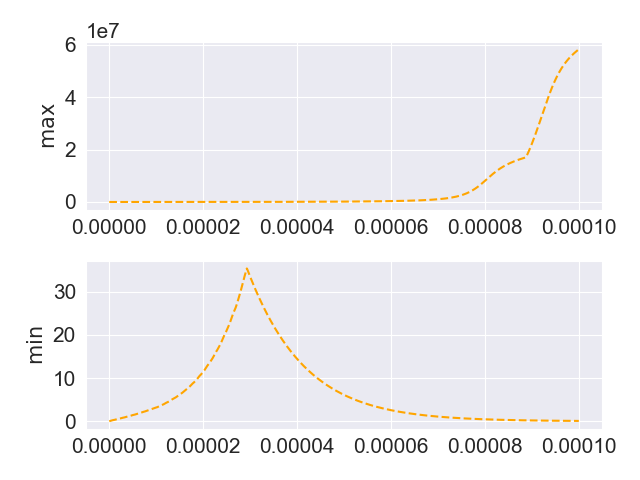}
	\end{subfigure}
	\begin{subfigure}{0.49\textwidth}
		\centering
		\boldmath{$v$}
		
		\includegraphics[scale=0.5]{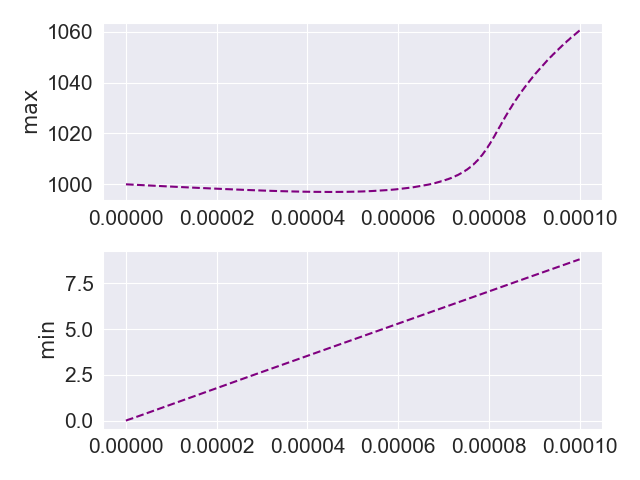}
	\end{subfigure}
	\caption{Minimum and maximum of $u$ and $v$ in the case of
	multiple peaks.}
	\label{fig:min-max_sin-cos}
\end{figure}

\begin{figure}
	\centering
	\includegraphics[scale=0.5]{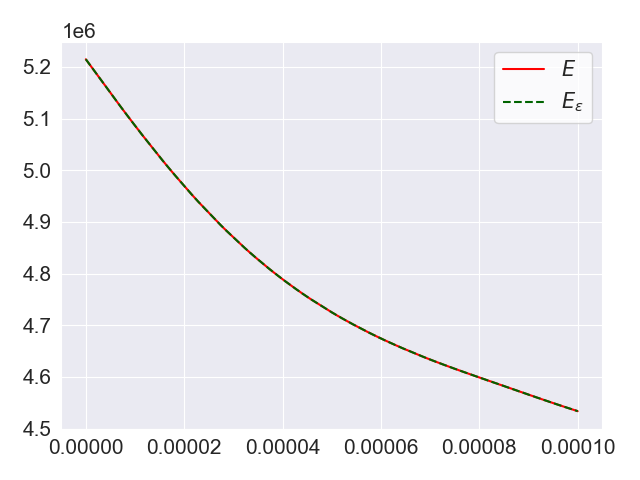}
	\caption{Discrete energy over time in the case of
	multiple peaks.}
	\label{fig:energy_sin-cos}
\end{figure}

\begin{remark}
	\label{rmk:sin-cos}
	This test was also computed with $h\approx 2.828\cdot 10^{-2}$, $\Delta t=2.5\cdot 10^{-6}$ and values of $\varepsilon$ lower than $10^{-10}$. In this case, the minimum value of $u$ tends to $0$ (see Figure~\ref{fig:min-max_sin-cos}), hence \eqref{esquema_DG_upw_KS_non_truncated} is not well suited for too small values of $\varepsilon$ and the convergence of Newton's method (or other iterative methods) to approximate the solution of the nonlinear schemes is not guaranteed as $\lim_{u\to 0}\log(u)=\infty$. Therefore, regularizing the chemical potential of $u$ eases the convergence of the numerical method while only introducing a small error as shown by the results in Table~\ref{table:test_3_eps}.

 	In this table, the difference in $L^2$ and $L^\infty$ norms between the approximation with $\varepsilon=0$ and the approximations with greater values of $\varepsilon$ at $t=2.5\cdot 10^{-4}$ and $t=5.25\cdot 10^{-4}$ (last time step before Newton's method stop converging with $\varepsilon=0$) are shown. In the range of values for $\varepsilon$ taken, $\varepsilon=10^{-12}$ is the lowest value for which Newton's method converges during the $4000$ time iterations computed. The minimum value of $u$ achieved at $t=0.01$, the last time step computed, with $\varepsilon=10^{-12}$ is of order $10^{-167}$. Newton's method stops converging with $\varepsilon=10^{-14}$ at $t=8.275\cdot 10^{-4}$ and at $t=5.525\cdot 10^{-4}$ with $\varepsilon=10^{-16}$ (same time step than with smaller $\varepsilon$ values including $\varepsilon=0$).
\end{remark}

\begin{table}
	\centering
	% \scriptsize
	\begin{tabular}{||c|c||c|c||c|c||}
		\hline
		\multirow{2}*{$t$} & \multirow{2}*{$\varepsilon$} & \multicolumn{2}{|c||}{$\norma{\cdot}_{L^2}$} & \multicolumn{2}{|c||}{$\norma{\cdot}_{L^\infty}$}\\
		\cline{3-6}
		& & $u$ & $v$ & $u$ & $v$\\
		\hline\hline
		\multirow{4}*{$2.5\cdot 10^{-4}$} & $10^{-10}$ & $7.41\cdot 10^{-11}$ & $1.87\cdot 10^{-13}$ & $2.71\cdot10^{-9}$ & $1.14\cdot 10^{-12}$ \\
		\cline{2-6}
		& $10^{-12}$ & $3.56\cdot 10^{-12}$  & $0.0$ & $2.33\cdot 10^{-10}$ & $0.0$\\
		\cline{2-6}
		& $10^{-14}$ & $8.24\cdot 10^{-13}$ & $0.0$ & $5.82e\cdot 10^{-11}$ & $0.0$ \\
		\cline{2-6}
		& $10^{-16}$ & $2.59\cdot 10^{-14}$ & $0.0$ & $1.82e\cdot 10^{-12}$ & $0.0$ \\
		\hline\hline
		\multirow{4}*{$5.25\cdot 10^{-4}$} & $10^{-10}$ & $2.14\cdot 10^{-10}$ & $6.18\cdot 10^{-13}$ & $7.30\cdot 10^{-9}$ & $3.75\cdot 10^{-12}$ \\
		\cline{2-6}
		& $10^{-12}$ & $5.44\cdot 10^{-12}$ & $2.55\cdot 10^{-13}$ & $2.33\cdot 10^{-10}$ & $1.71\cdot 10^{-12}$ \\
		\cline{2-6}
		& $10^{-14}$ & $1.16\cdot 10^{-12}$ & $0.0$ & $5.82\cdot 10^{-11}$ & $0.0$ \\
		\cline{2-6}
		& $10^{-16}$ & $6.70\cdot 10^{-15}$ & $0.0$ & $4.55\cdot 10^{-13}$ & $0.0$ \\
		\hline
	\end{tabular}
	\caption{Difference between approximations of the test in subsection~\ref{test3} ($h\approx 2.828\cdot 10^{-2}$, $\Delta t=2.5\cdot 10^{-6}$) with respect to the solution with $\varepsilon=0$.}
	\label{table:test_3_eps}
\end{table}

\section*{Acknowledgments}
The first author has been supported by \textit{UCA FPU contract UCA/REC14VPCT/2020 funded by Universidad de Cádiz} and by a \textit{Graduate Scholarship funded by the University of Tennessee at Chattanooga}. The second and third authors have been supported by \textit{Proyecto PGC2018-098308-B-I00, funded by FEDER/Ministerio de Ciencia e Innovación - Agencia Estatal de Investigación, Spain}.

\bibliography{biblio_database}

\end{document}